\documentclass[reqno,11pt]{amsart}

\usepackage[english]{babel}
\usepackage[utf8x]{inputenc}
\usepackage[T1]{fontenc}

\usepackage[margin=1in, letterpaper]{geometry}
\usepackage{amssymb,amsmath,amsfonts,amsthm}
\usepackage{bm,bbm,mathrsfs}
\usepackage[colorlinks=true, allcolors=black]{hyperref}
\usepackage{url}
\usepackage{enumitem}
\usepackage{mathtools}
\usepackage{placeins}
\usepackage{extarrows}
\usepackage{subfig}
\usepackage[font=small,labelfont=bf]{caption}
\usepackage{floatrow}
\usepackage{verbatim}
\usepackage{tikz}
\usetikzlibrary{arrows,arrows.meta,cd,decorations.pathmorphing,backgrounds,positioning,fit,matrix}
\usepackage[noabbrev,capitalize]{cleveref}
\crefname{equation}{}{}
\usepackage{xcolor,colortbl,graphics,graphicx}
\definecolor{myGreen}{rgb}{0.87, 1, 0.87}
\definecolor{myLighterGreen}{rgb}{0.94, 1, 0.94}

\newtheorem{theorem}{Theorem}[section]
\newtheorem{proposition}[theorem]{Proposition}
\newtheorem{lemma}[theorem]{Lemma}

\newtheorem{corollary}[theorem]{Corollary}

\newtheorem*{question*}{Question}
\Crefname{question}{Question}{Questions}

\theoremstyle{definition}
\newtheorem{definition}[theorem]{Definition}
\newtheorem{notation}[theorem]{Notation}

\newtheorem{question}{Question}
\newtheorem{example}[theorem]{Example}
\newtheorem*{example*}{Example}

\theoremstyle{remark}
\newtheorem*{remark}{Remark}

\numberwithin{equation}{section}
\allowdisplaybreaks

\renewcommand{\arraystretch}{1.2}
\newcommand{\Z}{\mathbb{Z}}
\newcommand{\R}{\mathbb{R}}

\newcommand{\C}{\mathbb{C}}
\renewcommand{\H}{{\mathbb{H}}}
\newcommand{\Hplus}{{\mathbb{H}^{\textnormal{+}\hspace{-0.03cm}}}}
\newcommand{\Hminus}{{\mathbb{H}^{\textnormal{-}\hspace{-0.03cm}}}}
\newcommand{\E}{\textnormal E}
\newcommand{\W}{\textnormal W}
\newcommand{\mT}{\mathcal{T}}
\newcommand{\mS}{\mathcal{S}}
\newcommand{\one}{\mathbbm{1}}
\renewcommand{\Re}{\textnormal{Re}\ }
\renewcommand{\Im}{\textnormal{Im}\ }

\newcommand{\Hol}{\textnormal{Hol}}
\newcommand{\la}{\left\langle}
\newcommand{\ra}{\right\rangle}

\newcommand{\vect}[2]{\big[{#1};{#2}\big]}
\newcommand{\twoline}[2]{$\substack{\text{#1} \\ \text{#2}}$}
\renewcommand{\hat}{\widehat}

\renewcommand{\emptyset}{\varnothing}
\newcommand{\txt}{\textnormal}

\newenvironment{psmall}
  {\left(\begin{smallmatrix}}
  {\end{smallmatrix}\right)}
\newcommand{\fix}[1]{\textcolor{blue}{\large (#1)\normalsize}}

\title[Several new product identities and  Rogers--Ramanujan type sums]{Several new product identities in relation to Rogers--Ramanujan type sums and mock theta functions}

\author[Alexandru Pascadi]{Alexandru Pascadi}
\address{University of California, Los Angeles, CA 90095, USA}
\email{alexpascadi@gmail.com}

\begin{document}

\begin{abstract}
Product identities in two variables $x, q$ expand infinite products as infinite sums, which are linear combinations of theta functions; famous examples include Jacobi's triple product identity, Watson's quintuple identity, and Hirschhorn's septuple identity. We view these series expansions as representations in canonical bases of certain vector spaces of quasiperiodic meromorphic functions (related to sections of line and vector bundles), and find new identities for two nonuple products, an undecuple product, and several two-variable Rogers--Ramanujan type sums. Our main theorem explains a correspondence between the septuple product identity and the two original Rogers--Ramanujan identities, involving two-variable analogues of fifth-order mock theta functions. We also prove a similar correspondence between an octuple product identity of Ewell and two simpler variations of the Rogers--Ramanujan identities, which is related to third-order mock theta functions, and conjecture other occurrences of this phenomenon. As applications, we specialize our results to obtain identities for quotients of generalized eta functions and mock theta functions.
%Product identities in two variables $x, q$ expand infinite products as infinite sums, which are linear combinations of theta functions; famous examples include Jacobi's triple product identity, Watson's quintuple identity, and Hirschhorn's septuple identity. We view these series expansions as representations in canonical bases of certain vector spaces of quasiperiodic meromorphic functions (related to sections of line and vector bundles), and find new identities for two nonuple products, an undecuple product, and several two-variable Rogers--Ramanujan type sums. Our main theorem explains a correspondence between the septuple product identity and the two original Rogers--Ramanujan identities, which amounts to an unexpected proportionality of canonical basis vectors that can be viewed as two-variable analogues of fifth-order mock theta functions. We also prove a similar correspondence between an octuple product identity of Ewell and two simpler variations of the Rogers--Ramanujan identities, which is related to third-order mock theta functions, and conjecture other occurrences of this phenomenon. As applications, we specialize our results to obtain identities for quotients of generalized eta functions and mock theta functions.
\end{abstract}

\maketitle

\vspace{-0.4cm}

\section{Introduction} \label{sec:intro}
Fix $\tau \in \C$ with $\Im \tau > 0$, and let $z \in \C$. Denote $q := e^{2\pi i \tau}$, $x := e^{2\pi i z}$, $(x; q) := \prod_{n = 0}^\infty \left(1 - q^n x\right)$, and $\la x; q \ra := (x; q)(q/x; q)$ (the standard notations have a subscript of $\infty$, which we drop for brevity; this will be helpful in identities like \cref{eq:nonuple}). Then Jacobi's triple product identity \cite[p.~10]{berndt2006number} reads
\begin{equation}\label{eq:triple}
    (q; q)\la x; q \ra =
    \sum_{n \in \Z} (-1)^n q^{\binom{n}{2}} x^n.
\end{equation}
As functions of $z$, both sides of \cref{eq:triple} converge absolutely and locally uniformly to entire $1$-periodic functions, and satisfy the $\tau$-quasiperiodicity $F(z) = -xF(z+\tau)$; in geometric language, $F$ is a section of a certain holomorphic line bundle over the complex torus $\C/(\Z + \Z\tau)$. This observation already proves \cref{eq:triple} up to a factor depending only on $q$, since the space $\mT_\C(-x) := \{F \in \Hol(\C/\Z) : F(z) = -x F(z+\tau)\}$ can be shown to be one-dimensional over $\C$ by identifying Fourier coefficients (here, $\Hol(\C/\Z)$ denotes the space of entire functions with period $1$). More generally, if $D \subset \C$ is open, connected and closed under $\Z$-translation, and $f \in \Hol(\C/\Z) \setminus \{0\}$, then the space
\[
\mT_D(f) := \left\{F \in \Hol(D/\Z)\ :\ F(z) = f(z) F(z+\tau)\ \text{ when } z, z + \tau \in D\right\}
\]
can be shown to be finite-dimensional over $\C$ for various choices of $D$ and $f$. In particular, given a positive integer $d$ and a nonzero $\alpha \in \C$, the space $\mT_\C\left(\alpha x^d\right)$ has a canonical basis consisting of the theta functions
\[
    \vect{\alpha x^d}{k} = \vect{\alpha x^d}{k}(z) := \sum_{n \in \Z} \alpha^n q^{d\binom{n}{2} + kn} x^{dn+k}, 
    %= (q^a; q^a) \big\langle {-s} x^a q^{k}; q^a \big\rangle x^k,
\]
where $k \in \Z$ varies in any complete residue system modulo $d$ (we introduce this notation because it will admit a natural generalization to other functions of the form $\vect{f}{k}$). Many infinite products can be designed to live in such spaces $\mT_D(f)$ (see \cref{tbl:spaces}, $\S \ref{subsec:notation}$), so it is natural to look for their representations in a canonical basis thereof; e.g., \cref{eq:triple} states that $(q; q)\la x; q \ra = [-x; 0] \in \mT_\C(-x)$. Similarly, the \emph{quintuple} \cite{watson1938ramanujans} and \emph{septuple} \cite{hirschhorn1983simple,chu2007unification,garvan2001generalization,foata2001triple} product identities respectively state that
\begin{align}
    (q; q)\la x; q \ra \la qx^2; q^2 \ra &=  
    \vect{qx^3}{0} - \vect{qx^3}{1}
    \label{eq:quintuple}, \\[4pt]
    (q; q) \la x; q \ra \la qx^2; q^2 \ra \la x^2; q^2 \ra &= - \frac{\vect{-qx^5}{1} + \vect{-qx^5}{2}}{\la q^2; q^5 \ra} + \frac{\vect{-qx^5}{0} + \vect{-qx^5}{3}}{\la q; q^5 \ra}. 
    \label{eq:septuple}
\end{align}

In this paper, we prove several new product identities pertaining to more general spaces $\mT_D(f)$, and study their connection to two-variable Rogers--Ramanujan type sums \cite{sills2007identities}; as applications, we deduce a few one-variable identities of generalized eta functions \cite{yang2004transformation} and mock theta functions \cite{ramanujan1988lost,andrews1989ramanujan}. The latter were introduced by Ramanujan in his last letter to Hardy \cite{ramanujan1988lost}, and are modernly understood as holomorphic parts of harmonic Maass forms \cite{bringmann2017harmonic}. %; a thorough account of harmonic Maass forms and their holomorphic projections is given in \cite{bringmann2017harmonic}.  
We will find that the fifth-order mock theta functions are intimately related to the canonical basis vectors of $\mT_\Hminus\left(qx^2 - x\right)$ (where $\H^- := \{\Im z < 0\}$), while the third-order ones are similarly connected to the space $\mT_\Hminus\left(qx^2 - q^{-1}\right)$.%; see \cref{eq:truncations}.

For a start, in \cref{subsec:higher-order} we continue the sequence of identities \cref{eq:triple}, \cref{eq:quintuple}, \cref{eq:septuple} in a natural way, with two \emph{nonuple} product identities and an \emph{undecuple} identity; we state the first of these below.

\begin{proposition}[First nonuple product identity] \label{prop:nonuple}
As an identity of functions in $\mT_\C\left(q^2x^7\right)$,
\begin{equation} \label{eq:nonuple}
\begin{split}
    (q;q) \la x; q \ra \la qx^2; q^2 \ra^2 \la x^2; q^2 \ra
    &=
    - \frac{\vect{q^2x^7}{1} + \vect{q^2x^7}{2}}{\la q^2 ; q^7 \ra^2 \la q; q^7\ra}
    +
    \frac{\vect{q^2x^7}{0} + \vect{q^2x^7}{3}}{\la q ; q^7 \ra^2 \la q^3; q^7\ra}
    \\[4pt]
    &
    + 
    q\frac{\vect{q^2x^7}{-1} + \vect{q^2x^7}{4}}{\la q^3 ; q^7 \ra^2 \la q^2; q^7\ra} 
    - 
    q\frac{2\vect{q^2x^7}{5}}{\la q; q^7\ra \la q^2; q^7 \ra \la q^3; q^7 \ra }.
\end{split}
\end{equation}
\end{proposition}

%\begin{remark}
%Basis vectors tend to pair up due to symmetries at $x \mapsto x^{-1}$, as we explain in \cref{subsec:dim-bounds}.
%\end{remark}
While $\mT_\C\left(\alpha x^d\right)$ is $d$-dimensional for $d \ge 1$, we have $\mT_\C\left(\alpha x^{-d}\right) = \mT_\C(y) = \{0\}$ for $y \not\in \left\{q^n : n \in \Z\right\}$; this will lead to identities for quotients of double infinite products $\la x; q\ra$, in \cref{prop:fractional}. Passing to spaces such as $\mT_\Hminus\left(qx^2-x\right)$ (where $\H^- := \{\Im z < 0\}$), which is two-dimensional and corresponds to a meromorphic line bundle, we begin to encounter Rogers--Ramanujan type sums. In fact, the two renowned one-variable Rogers--Ramanujan identities \cite[p.~290]{hardy1979introduction},
\begin{equation} \label{eq:rog-ram}
    \sum_{n \geq 0} \frac{q^{n^2}}{(q; q)_n} = \frac{1}{\la q; q^5 \ra} \qquad\quad \text{ and }\qquad\quad 
    \sum_{n \geq 0} \frac{q^{n^2+n}}{(q; q)_n} = \frac{1}{\la q^2; q^5 \ra},
\end{equation}
where $(x; q)_n := \prod_{k = 0}^{n-1} \left(1 - xq^k\right)$, are together equivalent to the following new product identity.

\begin{proposition}[Two-variable statement of \cref{eq:rog-ram}]\label{prop:2var-rog-ram}
In $\mT_\Hminus\left(qx^2-x\right)$, one has 
\begin{equation}\label{eq:2var-rog-ram}
    (q; q) \la x; q \ra (qx; q)\ =  
    \sum_{m \in \Z, n \ge 0} \frac{q^{\binom{m}{2} + \binom{n+1}{2}}}{(q; q)_n} (-x)^{m+n}
    \ =\ 
    \frac{\vect{qx^2 - x}{0}}{\la q^2; q^5 \ra} - \frac{\vect{qx^2 - x}{1}}{\la q; q^5 \ra},
\end{equation}
where for $k \in \{0, 1\}$, $\vect{qx^2 - x}{k}$ is the unique function in $\mT_\Hminus\left(qx^2-x\right)$ whose Fourier expansion in $z \in \H^-$ has coefficient of $x^j$ equal to $\one_{j = k}$, for $j \in \{0, 1\}$. %(hence \cref{eq:rog-ram} follows from \cref{eq:2var-rog-ram} by identifying coefficients of $x^0$ and $x^1$, and more identities follow by isolating $x^n$ for $n \in \Z$).
\end{proposition}

\begin{remark}
The functions $\vect{qx^2-x}{k}$ can be expressed as explicit series in $\H^-$ (see \cref{cor:kind2}), and can then be meromorphically continued to $\C$; their poles must cancel out in the right-hand side of \cref{eq:2var-rog-ram} since the left-hand side is entire. In fact, $\dim \mT_\C\left(qx^2 - x\right) = 1$ while $\dim \mT_\Hminus\left(qx^2 - x\right) = 2$. Also, the Rogers--Ramanujan identities in \cref{eq:rog-ram} follow from the second equality in \cref{eq:2var-rog-ram} by identifying coefficients of $x^0$ and $x^1$. In fact, identifying coefficients of $x^{1-m}$ for $m \in \Z$ leads to the so-called $m$-versions of the Rogers--Ramanujan identities, involving $\sum_{n \ge 0} q^{n^2 + mn}/(q; q)_n$ (see, e.g., \cite[(3.5)]{garrett1999variants}).
\end{remark}

Much of the work in this paper is motivated by a visible connection between \cref{prop:2var-rog-ram} and the septuple identity: the basis representations in \cref{eq:septuple} and \cref{eq:2var-rog-ram} have the same $q$-coefficients, up to a sign. This suggests the more difficult result that sums of canonical basis vectors of $\mT_\C\left(-qx^5\right)$ are proportional to the canonical basis vectors of $\mT_\Hminus\left(qx^2-x\right)$, which is our main theorem.
\begin{theorem}[Septuple identity vs.\ Rogers--Ramanujan] \label{thm:bases-proportional}
For $0 < |q| < 1$, one has 
\[
    %x\left(x^{-1}; q\right)\frac{\vect{-qx^3}{0}+\vect{-qx^3}{1}}{(q; q)}
    %x \la -x; q\ra \la qx^2; q^2 \ra \left(x^{-1}; q\right)
    %-\frac{\la x^2; q^2\ra\la qx^2; q^2\ra}{(qx; q)} 
    -\frac{\la x^2; q\ra}{(qx; q)} 
    \ =\  
    \frac{\vect{-qx^5}{1} + \vect{-qx^5}{2}}{\vect{qx^2 - x}{0}} \ =\  \frac{\vect{-qx^5}{0} + \vect{-qx^5}{3}}{\vect{qx^2 - x}{1}},
\]
as an identity of meromorphic functions of $z \in \C$ (where $x = e^{2\pi i z}$). The left-hand side above equals minus the ratio of the septuple product in \cref{eq:septuple} to the product in \cref{prop:2var-rog-ram}. %An explicit identity for the products obtained by cross-multiplying the ratios above is given in \cref{thm:cross-multiplication}.
\end{theorem}

In particular, under \cref{thm:bases-proportional}, the septuple identity and \cref{eq:2var-rog-ram} are equivalent. Remarkably, an analogous correspondence arises between the \emph{octuple} product identity of Ewell \cite{ewell1982octuple} (given in an equivalent form in this paper's \cref{prop:octuple}) and a variation of \cref{eq:2var-rog-ram}, given in \cref{prop:2var-rog-ram-var}.

\begin{theorem}[Octuple identity vs.\ Rogers--Ramanujan variation] \label{thm:bases-proportional-2}
For $0 < |q| < 1$, one has
\[
    %x\left(x^{-2}; q^2\right) \frac{\vect{-qx^2}{0}}{(q; q)}
    %x \la qx^2; q^2 \ra \frac{\left(x^{-2}; q^{2}\right)}{\left(q; q^2\right)}
    -\frac{\left(-x^{-1}; q\right) \la x; q\ra \la qx^2; q^2 \ra }{(q; q^2)(qx; q)} 
    \ =\ \frac{\vect{-qx^4}{1}}{\vect{qx^2 - q^{-1}}{0}} \ =\  \frac{\vect{-qx^4}{0} + \vect{-qx^4}{2}}{\vect{qx^2 - q^{-1}}{1}},
\]
as an identity of meromorphic functions of $z \in \C$ (where $x = e^{2\pi i z}$). The left-hand side above equals minus the ratio of the octuple product in \cref{eq:octuple} to the product in \cref{prop:2var-rog-ram-var}. Here, $\vect{qx^2 - q^{-1}}{k} \in \mT_\Hminus\left( qx^2 - q^{-1}\right)$ are defined analogously to $\vect{qx^2 - x}{k}$ in \cref{prop:2var-rog-ram}. %; see \cref{prop:canonical-basis-vectors} for details.
\end{theorem}

%We establish \cref{thm:bases-proportional} and \cref{thm:bases-proportional-2} in \cref{sec:rog-ram}, the former being significantly more difficult. In \cref{thm:cross-multiplication}, we also give explicit identities for the products given by cross-multiplying the ratios of basis vectors in these theorems.

%\begin{remark} 
%The nonuple identity \cref{eq:nonuple} is also linked to Rogers--Ramanujan type sums, which suggests a general phenomenon encompassing \cref{thm:bases-proportional,thm:bases-proportional-2}; we detail this observation in \cref{qtn:characters}.
%\end{remark}

One can also specialize our two-variable results at particular values of $x$, to produce one-variable identities of modular forms and mock theta functions; we further mention one corollary about each, proven using \cref{prop:nonuple} and \cref{thm:bases-proportional} respectively. First, denote the Dedekind eta function and the generalized Dedekind eta functions of level $N \ge 1$ (following \cite[Corollary 2]{yang2004transformation}) by
\begin{equation} \label{eq:generalized-eta}
    \eta(\tau) := e^{\pi i \tau / 12}(q; q),
    \qquad\qquad 
    \E_g(\tau) := e^{N B(g/N) \pi i \tau} \la q^g; q^N \ra,
\end{equation}
where $B(t) := t^2 - t + \frac{1}{6}$, $g \in \Z$, and the level $N$ must be specified a priori (the generalized eta functions $\E_g(\tau)$ should not be confused with the Eisenstein series $E_{2k}(\tau)$). It should not be surprising, based on \cref{eq:generalized-eta}, that specializing product identities at $x = \pm q^g$ recovers identities for products and quotients of eta functions; what is more intriguing is the simplicity and symmetrical structure of these identities (which are related to their behavior under modular transformations; see \cref{subsec:generalized-eta} for more details and examples). \cref{cor:eta-poly} provides one such result.

\begin{corollary}[Eta quotient polynomial]\label{cor:eta-poly}
For $\tau \in \H^+$, let $\E_g(\tau)$ denote the generalized eta functions of level $N = 7$. Then one has the polynomial factorization
\[
    X^3
    - 2\frac{\eta(14\tau)^2}{\eta(7\tau)^2} X^2
    - \frac{\eta(2\tau)\eta(7\tau)^3}{\eta(\tau)\eta(14\tau)^3} X
    + \frac{\eta(\tau)\eta(14\tau)}{\eta(2\tau)\eta(7\tau)}
    = \prod_{g = 1}^3 \left(X - \frac{\E_{g}(2\tau)}{\E_{3g}(\tau)}\right).
\]
\end{corollary}

%\begin{remark}
%Of the three Vieta identities that are together equivalent to \cref{cor:eta-poly}, one is almost immediate, one is due to Hickerson \cite{hickerson1988seventh}, and one follows from our nonuple identity \cref{eq:nonuple}.
%\end{remark}
%\begin{remark}
%The left-hand sides of the triple, quintuple, septuple, octuple, nonuple and undecuple product identities all have a zero at $x = 1$ due to the factor of $\la x; q\ra$. Among these, the fact that the right-hand side does not trivially reduce to zero at $x = 1$ is unique to the two nonuple identities.
%\end{remark}
Our second corollary involves four of Ramanujan's famous fifth-order mock theta functions \cite{watson1937mock,andrews1989ramanujan}:
\begin{equation} \label{eq:mock-theta}
f_j(q) := \sum_{n \ge 0} \frac{q^{n^2+jn}}{(-q; q)_n},
\qquad \
\psi_j(q) := \sum_{n+j \ge 1} q^{\binom{n+1}{2}} (-q; q)_{n+j-1}, \qquad\text{for } j \in \{0, 1\}.
\end{equation}

\begin{corollary}[Fifth-order mock theta sums]\label{cor:mock-theta-fifth}
For $j \in \{0, 1\}$, one has
\[
    (q; q)^2\left(q; q^2\right)\Big( f_j(q) + 2\psi_j(q) \Big) =  \sum_{n \in \Z} q^{\frac{n(5n+2j+1)}{2}} (10 n + 2j + 1).
\]
%\begin{equation} \label{eq:mock-theta-ratio}
%    \frac{f_0(q) + 2\psi_0(q)}{f_1(q) + 2\psi_1(q)}
%    =
%    \frac{\sum_{n \in \Z} q^{n(5n-1)/2} (10 n - 1)}
%    {\sum_{n \in \Z} q^{n(5n-3)/2} (10 n - 3)}.
%\end{equation}
%EXPAND IN POWER SERIES IN q as a numerical example!!!!
\end{corollary}
%https://www.sciencedirect.com/science/article/pii/0001870889900704
%^you did this by identifying powers x^k with k >= 1 in [qx^2 - x; 0] (= sum over n >= 1 of ...) and [qx^2 - x; 1] (= sum over n >= 0 of ...)
\begin{remark}
%In \cref{prop:mock-theta-sums}, we recover explicit identities for $f_j + 2\psi_j$ when $j \in \{0, 1\}$, which are part of Watson's work on mock theta functions from the same family. 
\cref{cor:mock-theta-fifth} follows from \cref{thm:bases-proportional} by taking $x \to -1$, and \cref{thm:bases-proportional-2} is similarly related to third-order mock theta functions; we also obtain individual identities for $\psi_j$ in \cref{prop:mock-theta-individual}. More difficult identities for $f_j$ and $\psi_j$ are given by Ramanujan's famous mock theta conjectures \cite{andrews1989ramanujan}, proven by Hickerson \cite{hickerson1988proof} and more recently by Folsom \cite{folsom2008short} using Maass forms.
%We note that much more difficult identities for $f_j$ and $\psi_j$ are the content of Ramanujan's famous mock theta conjectures \cite{andrews1989ramanujan}; these were proven initially by Hickerson \cite{hickerson1988proof} using Hecke-type identities, and more recently by Folsom \cite{folsom2008short} using harmonic Maass forms.
%https://www.sciencedirect.com/science/article/pii/0001870889900704
%\fix{To prove mock theta we could try to reverse-engineer in Andrews' M4 and M9: add your matrix times a vector, and let $q \mapsto \sqrt{q}$. The vector you should add to the LHS is below}
%\[
%    \left(q^{10}; q^{10}\right)
%    \begin{pmatrix}
%        q\la q; q^{10}\ra / \la q^2; q^5\ra \\
%        \la q^3; q^{10}\ra / \la q; q^5\ra
%    \end{pmatrix}.
%\]
\end{remark}

The table below provides short $q$-expansions of the relevant series in \cref{cor:mock-theta-fifth}.

\begin{center} \small
\renewcommand{\arraystretch}{1.5}
\begin{tabular}{c|c|c|c|c}
     $(q; q)^2 \left(q; q^2\right)$ & $f_0(q) + 2\psi_0(q)$ & $f_1(q) + 2\psi_1(q)$ & $\sum q^{n(5n+1)/2} (10 n + 1)$ & $\sum q^{n(5n+3)/2} (10 n + 3)$\\
     \hline
     \footnotesize{$1 - 3q + q^2 + 2q^3$}
     &
     \footnotesize{$1 + 3q - q^2 + 3q^3$}
     &
     \footnotesize{$3 + 2q + 3q^2 + q^3$}
     &
     \footnotesize{$1 - 9q^2 + 11q^3 - 19q^9$}
     &
     \footnotesize{$3 - 7q + 13q^4 - 17q^7$}
     \\
     \footnotesize{$+\ 2q^4 - q^5 - \cdots$}
     &
     \footnotesize{$+\ 2q^4 + q^6 + \cdots$}
     &
     \footnotesize{$+\ 3q^4 + q^5 + \cdots$}
     &
     \footnotesize{$+\ 21 q^{11} - 29 q^{21} + \cdots$}
     &
     \footnotesize{$+\ 23q^{13} - 27 q^{18} + \cdots$}
\end{tabular}
\end{center}

Finally, while proving \cref{thm:bases-proportional,thm:bases-proportional-2}, we find it useful to work with three different bases of $\mT_\C\left(\alpha x^d\right)$ and to determine the corresponding change-of-basis matrices; their entries turn out to be two-variable Rogers--Ramanujan type sums. We gather the resulting matrix identities in \cref{thm:change-of-basis}, but for now we only mention a particular case which is closely related to \cref{thm:bases-proportional}.

\begin{proposition}[Two-variable $2 \times 2$ determinant identity] \label{prop:2x2-determinant}
For $x, q \in \C$ with $|q| < 1$, one has 
\[
    \det 
    \begin{pmatrix}
    \sum_{n \ge 0} \frac{q^{n(n+1)}}{(q; q)_{2n}} x^{2n}
    & 
    \sum_{n \ge 0} \frac{q^{(n+1)^2}}{(q; q)_{2n+1}} x^{2n+1}
    \vspace{0.3cm} \\
    \sum_{n \ge 0} \frac{q^{n(n+1)}}{(q; q)_{2n+1}} x^{2n+1}
    &
    \sum_{n \ge 0} \frac{q^{n^2}}{(q; q)_{2n}} x^{2n}
    \end{pmatrix}
    = 1.
\]
%At $x = 1$, this implies an identity of generalized eta functions of level $N = 10$ (see \cref{prop:eta-level-10}):
%\[
%    \E_1(\tau) \E_2(\tau) \E_3(2\tau) \E_4(2\tau)
%    -
%    \E_1(2\tau)\E_2(2\tau)\E_3(\tau)\E_4(\tau)
%    =
%    \frac{\eta(\tau)^2}{\eta(10\tau)^2}.
%\]
%can deduce an unexpected identity related to the undecuple product and to generalized eta functions of level $10$; see \cref{prop:eta-level-10}. 
\end{proposition}

We conclude the introduction with an aforementioned open question that arises from our work:

\begin{question} \label{qtn:characters}
As noted by Andrews, Schilling and Warnaar \cite[(5.2)]{andrews1999a2}, the Rogers--Ramanujan products $\langle q; q^5 \rangle^{-1}$ and $\langle q^2; q^5 \rangle^{-1}$ are characters of a certain $\W_2$ algebra, which they denote by $M(2, 5)_2$. We remarked that these products coincide with the $q$-coefficients of the canonical basis vectors in the septuple identity \cref{eq:septuple}, and explained this correspondence through \cref{thm:bases-proportional}. A similar phenomenon occurs for the octuple (\cref{prop:octuple}) and nonuple (\cref{prop:nonuple}) identities:
\begin{center}
\renewcommand{\arraystretch}{1.5}
\begin{tabular}{c|c|c|c|c}
     Algebra & Product identity & Rog.--Ram.\ type id. & 2-Var.\ generaliz. & Correspondence\\
     \hline 
     $M(2, 5)_2$ & Septuple, \cref{eq:septuple} & \cref{eq:rog-ram} & \cref{prop:2var-rog-ram} & \cref{thm:bases-proportional} \\
     \hline 
     $M(2, 4)_2$ & Octuple, \cref{eq:octuple} & \cref{eq:rog-ram-var}& 
     \cref{prop:2var-rog-ram-var} & \cref{thm:bases-proportional-2}  \\
     \hline 
     $M(3, 7)_3$ 
     & Nonuple, \cref{eq:nonuple} & \cite[Theorem 1.1]{corteel2019varvec} &
     (Open question) & (Open question)
\end{tabular}
\end{center}
Indeed, the $q$-coefficients in the right-hand side of the nonuple identity \cref{eq:nonuple} coincide with the 4 characters of the $\W_3$ algebra $M(3, 7)_3$ (up to factors of $\pm 1, \pm q, \pm 2q$), and \cite[Theorem 5.2]{andrews1999a2} gives $A_2$ Rogers--Ramanujan identities for 3 out of these characters (see also \cite{warnaar2006hall}); the more recent work of Corteel and Welsh also provides the fourth identity \cite[Theorem 1.1]{corteel2019varvec}. Can one complete the bottom row of the table above with an analogue of \cref{thm:bases-proportional}, leading to a new proof of Theorem 1.1 from \cite{corteel2019varvec}, and possibly to applications on higher-order mock theta functions?

\begin{remark}
There is also a product identity whose $q$-coefficients are the characters of $M(2, u^2 + v^2)_2$, for any relatively prime positive integers $u, v$ such that $u + v$ is odd; see our \cref{cor:character-identity}. These might be connected as in \cref{qtn:characters} to the Andrews--Gordon identities \cite[Theorem 7.8]{andrews1998theory}, which generalize \cref{eq:rog-ram}; \cref{thm:bases-proportional} would correspond to the case $(u, v) = (1, 1)$. As pointed out by Professor Ole Warnaar (private communication), the analogue of \cref{prop:2var-rog-ram} in this case could be a summation of variants of the Andrews--Gordon identities as in \cite[(3.21)]{berkovich2001variants}; the difficulty lies in relating such a summation to an infinite product in a suitable $\mT_\Hminus(f)$ space.
\end{remark}
\end{question}

\section{Overview and Notation}

\subsection{Methods}
Given an entire $1$-periodic function $f \in \Hol(\C/\Z)$, there are two main ways to construct a function with a $\mT_D(f)$-type quasiperiodicity from scratch: the expressions %via an infinite product or an infinite sum. Indeed, expressions of the form
\begin{equation} \label{eq:two-ways}
    \prod_{n \ge 0} f(z + n\tau)
    \qquad\quad \text{and} \qquad\quad
    \sum_{n < 0}\ \prod_{-n \le j < 0} f(z + j\tau)^{-1}
    \ +\
    1
    \ +\
    \sum_{n > 0}\ \prod_{0 \le j < n} f(z + j\tau)
\end{equation}
are, if well-defined, quasiperiodic with $F(z) = f(z)F(z + \tau)$. Given that many of the spaces $\mT_D(f)$ are finite-dimensional (as we prove in \cref{subsec:dim-bounds}), identities relating infinite products to infinite sums as above are bound to arise. 
Taking $f(z) = 1 - x$ and $f(z) = \alpha x^d$, one recovers $(x; q) = \prod_{n \ge 0} (1 - q^n x)$ and $\vect{\alpha x^d}{0} = \sum_{n \in \Z} \alpha^n q^{dn(n-1)/2} x^{dn}$ as the product, respectively the sum in \cref{eq:two-ways}; the same sum leads to formulae for other basis vectors $\vect{f}{k}$, as we show in \cref{subsec:canonical-basis-vectors}. 

The constructions in \cref{eq:two-ways} also correspond to two ways to produce new functions in $\mT_D(f)$ from old ones. Firstly, the product of a function in $\mT_D(f)$ and a function in $\mT_D(g)$ is a function in $\mT_D(fg)$; for instance, $\vect{\alpha q^k x^d}{0} \cdot x^k = \vect{\alpha x^d}{k}$. %by multiplying the series $\vect{\alpha q^k x^d}{0} \in \mT_\C\left(\alpha q^kx^d\right)$ with $x^k \in \mT_\C\left(q^{-k}\right)$, one recovers $\vect{\alpha x^d}{k}$. 
Going further, the product of two basis vectors $\vect{\alpha x^a}{k}$ and $\vect{\beta x^b}{j}$ (where $a, b, k, j \in \Z$, $a, b \ge 1$, $\alpha, \beta \in \C^\times$) is an element of $\mT_\C\left(\alpha \beta x^{a+b}\right)$, so we must have
\begin{equation}\label{eq:m-coeffs}
    \vect{\alpha x^a}{k}\vect{\beta x^b}{j} = \sum_{0 \le \ell < a+b} M_\ell(q) \vect{\alpha \beta x^{a+b}}{\ell},
\end{equation}

for some \emph{M-coefficients} $\{M_\ell(q)\}_\ell$. Closely related identities were noted by various authors in different forms \cite{chu2007unification,adiga2014identities,yan2009several}; we formulate a generalization of \cref{eq:m-coeffs} in \cref{lem:m-identities}, and use it together with the triple and quintuple product identities to deduce the nonuple and undecuple identities.

In a similar spirit, \emph{twisting} the $n$th term in a series as in \cref{eq:two-ways} by a factor of $w(z + n\tau)$ (for suitable functions $w$ with period $1$) should preserve the quasiperiodicity at $z \mapsto z + \tau$. In particular, by twisting the series defining $\vect{\alpha x^d}{k}$, we should be able to write (when well-defined)
\begin{equation}\label{eq:w-coeffs}
    \sum_{n \in \Z} \alpha^n q^{d\binom{n}{2} + kn} x^{dn+k}\  w(z+ n\tau) = \sum_{0 \le \ell < d} W_\ell(q) \vect{\alpha x^d}{\ell},
\end{equation}
for some \emph{W-coefficients} $\{W_\ell(q)\}_\ell$, provided that the left-hand side is entire. We identify a class of such identities in \cref{lem:t-identities}, leading to one method of proof for \cref{thm:bases-proportional,thm:bases-proportional-2,thm:change-of-basis}.

\subsection{Layout of paper} 
In \cref{sec:preliminaries}, we formally define and study a generalization of the vector spaces $\mT_D(f)$, then give a few applications including the triple and quintuple product identities. %, and the anticipated quotient identities from \cref{prop:fractional}. 
\cref{sec:higher-prod-id} states and proves the nonuple and undecuple product identities alongside other similar results,
%, puts them into perspective with similar results from literature,
and then uses them to produce identities of generalized eta functions such as \cref{cor:eta-poly}. Finally, \cref{sec:rog-ram} deals with applications on Rogers--Ramanujan type sums, including the proofs of \cref{prop:2var-rog-ram,prop:2x2-determinant}, \cref{thm:bases-proportional,thm:bases-proportional-2}, and results on mock theta functions such as \cref{cor:mock-theta-fifth}. The relationships between our main results are collected in \cref{fig:layout}.

\begin{figure}[ht] 
\begin{center}
\begin{tikzpicture}
  %%%%%%%%%%%%%%%%%% NODES:
  \node[draw,rectangle, fill=myGreen](structure) {
  \twoline{Structure of $\mT_D$}
  {spaces (\S \ref{subsec:dim-bounds})}
  };
  \node[draw,rectangle, fill=myLighterGreen](fractional) [right = 2.05cm of structure] {
  \twoline{Ident.\ for quotients of double}
  {products (Prop.\ \ref{prop:fractional})}
  };
  \node[draw,rectangle, fill=myLighterGreen](multiplication) [xshift=-1cm][below = 0.7cm of structure] {
  \twoline{$\mT_\C$ Multiplication}
  {ident.\ (\cref{lem:m-identities})}
  };
  \node[draw,rectangle, fill=myGreen](twisted) [right = 0.2cm of multiplication] {
  \twoline{$\mT_\C$ Twisted sum}
  {ident.\ (\cref{lem:t-identities})}
  };
  \node[draw,rectangle, fill=myGreen](formulae) [right = 0.2cm of twisted] {
  \twoline{Formulae for canonical $\vphantom{\mT_\C}$}{$\mT_\Hminus$ basis vectors (\S \ref{subsec:canonical-basis-vectors})}
  };
  \node[draw,rectangle](preliminary) [left = 0.2cm of multiplication] {
  \twoline{Preliminary product}
  {ident.\ (\S \ref{subsec:identification})}
  };
  \node[draw,rectangle](septuple) [below = 1.3cm of preliminary] {
  \twoline{Septuple ident.}
  {(see \cref{eq:septuple})}
  };
  \node[draw,rectangle](rogers) [left = 0.2cm of septuple] {
  \twoline{An id.\ of Rogers}
  {(see Prop.\ \ref{prop:rogers})}
  };
  \node[draw,rectangle](octuple) [right = 0.4cm of septuple] {
  \twoline{Octuple ident.}
  {(see Prop.\ \ref{prop:octuple})}
  };
  \node[draw,rectangle, fill=myGreen](nonuple1) [right = 0.4cm of octuple] {
  \twoline{First nonuple}
  {ident.\ (Prop.\ \ref{prop:nonuple})}
  };
  \node[draw,rectangle, fill=myGreen](nonuple2) [right = 0.4cm of nonuple1] {
  \twoline{Second nonuple}
  {ident.\ (Prop.\ \ref{prop:nonuple2})}
  };
  \node[draw,rectangle, fill=myGreen](high-order) [right = 0.4cm of nonuple2] {
  \twoline{Other high-order}
  {ident.\ (\S \ref{subsec:higher-order})}
  };
  \node[draw,rectangle, fill=myLighterGreen](rog-ram-var) [below = 1.5cm of octuple] {
  \twoline{2-Var.\ Rog.--Ram.}
  {variation (Prop.\ \ref{prop:2var-rog-ram-var})}
  };
  \node[draw,rectangle, fill=myLighterGreen](rog-ram) [left = 0.6cm of rog-ram-var] {
  \twoline{2-Var.\ Rog.--Ram.}
  {ident.\ (Prop.\ \ref{prop:2var-rog-ram})}
  };
  \node[draw,rectangle](andrews) [right = 0.6cm of rog-ram-var] {
  \twoline{$A_2$ Rog.--Ram.}
  {ident.\ \cite[Thm.\ 5.2]{andrews1999a2}}
  };
  \node[draw,rectangle,fill=myGreen](change-of-basis) [xshift = 0.2cm] [below = 3.63cm of formulae] {
  \twoline{$\mT_\C$ Change-of-basis}
  {ident.\ (Thm.\ \ref{thm:change-of-basis})}
  };
  \node[draw,rectangle](watson) [xshift = -0.3cm] [below = 1cm of rog-ram] {
  \twoline{2 mock-theta ident.\ of}
  {Watson (see \cref{eq:watson})}
  };
  \node[draw,rectangle,fill=myGreen](twisted-rog-ram) [right = 1.8cm of watson] {
  \twoline{Twisted 2-var.\ Rog.--Ram.}
  {ident.\ (Thm.\ \ref{thm:twisted-rog-ram})}
  };
  \node[draw,rectangle](slater) [right = 1.6cm of twisted-rog-ram] {
  \twoline{4 Rog.--Ram.\ type ident.}
  {of Slater (see \cref{eq:slater})}
  };
  \node[draw,rectangle,fill=myGreen](twisted-rog-ram-var) [below = 2.52cm of twisted-rog-ram] {
  \twoline{Twisted 2-var.\ Rog.--Ram.}
  {variation (Thm.\ \ref{thm:twisted-rog-ram-var})}
  };
  \node[draw,rectangle,double, fill=myGreen](proportionality1) [below = 0.3cm of watson] {
  \twoline{Septuple to Rog.--Ram.}
  {proportionality\ (Thm.\ \ref{thm:bases-proportional})}
  };
  \node[draw,rectangle,double,fill=myGreen](proportionality2) [below = 0.3cm of proportionality1] {
  \twoline{Octuple to Rog.--Ram.}
  {proportionality\ (Thm.\ \ref{thm:bases-proportional-2})}
  };
  \node[draw,rectangle,dashed](proportionality3) [below = 0.3cm of proportionality2] {
  \twoline{Nonuple to $A_2$ Rog.--Ram.}
  {proportionality\ (Qtn.\ \ref{qtn:characters})}
  };
  \node[draw,rectangle,fill=myGreen](more-rog-ram) [below = 0.3cm of slater] {
  \twoline{Imaginary Rog.--Ram.}
  {ident.\ (Cor.\ \ref{cor:more-rog-ram})}
  };
  \node[draw,rectangle,fill=myLighterGreen](mock-theta) [below = 0.3cm of more-rog-ram] {
  \twoline{More ident.\ of mock-theta}
  {functions (Cor.\ \ref{cor:mock-theta-fifth}; \S \ref{subsec:mock-theta})}
  };
  \node[draw,rectangle,fill=myGreen](eta-functions) [below = 0.3cm of mock-theta] {
  \twoline{Ident.\ of generalized eta}
  {functions (Cor.\ \ref{cor:eta-poly}; \S \ref{subsec:generalized-eta})}
  };

  %%%%%%%%%%%%%%%%%% EDGES:
  \draw [-triangle 45]
  (structure) -- (fractional);
  \draw [-triangle 45]
  (structure) -- (formulae.north);
  \draw [-triangle 45]
  (structure) -- (preliminary.north);
  \draw [-triangle 45]
  (structure) -- (multiplication.north);
  \draw [-triangle 45]
  (structure) -- (twisted.north);
  \draw [-]
  (structure) -| ([yshift=1.85cm]rogers.north);
  \draw [- triangle 45]
  ([yshift=1.58cm]rogers.north) -- (rogers.north);
  \draw [-triangle 45]
  (rogers) |- (rog-ram);
  \draw [-triangle 45]
  ([xshift=1.31cm] preliminary.south) -- ([xshift=1.31cm, yshift=-3.65cm] preliminary.south);
  \draw [-triangle 45]
  ([xshift=0.4cm] formulae.south) -- ([xshift=0.2cm] change-of-basis.north);
  %\draw [-triangle 45]
  %(preliminary.west) |-
  %  ++(-1.2cm,0) -- 
  %  ++(-1.2cm,0) |- node[xshift = 0.5cm, yshift = 0.2cm]{\tiny [cite]}
  %  node[xshift = 0.5cm, yshift = -0.3cm]{\small \twoline{using}{\cref{eq:rog-ram}}}
  %  (watson.west);
  \draw [-]
  (twisted.south) |-
    ([xshift=0.4cm, yshift=-0.35cm] formulae.south);
  \draw [- triangle 45]
  ([xshift=0.35cm]watson.east) -- node[xshift=-0.1cm, yshift = 0.3cm]{{\small \twoline{using}{Lem.\ \ref{lem:m-identities},}}} 
  node[xshift=-0.1cm, yshift = -0.3cm]{{\small \twoline{Lem.\ \ref{lem:proofs-by-value}}{}}}
  (twisted-rog-ram);
  \draw [triangle 45 -]
  (rog-ram) -| 
  ([xshift = 0.35cm] watson.east);
  \draw [triangle 45 -]
  ([yshift=-0.2cm]watson.east) -|
  ([xshift=0.35cm] watson.east);
  \draw [triangle 45 - triangle 45]
  (twisted-rog-ram.east) -- node[yshift = 0.3cm]{{\small \twoline{using}{Lem.\ \ref{lem:t-identities}}}}
  (slater.west);
  \draw [triangle 45 - triangle 45]
  (septuple) -- node[xshift = -0.7cm]{{\small \twoline{by}{Thm.\ \ref{thm:bases-proportional}}}}
  (rog-ram);
  \draw [triangle 45 - triangle 45]
  (octuple) -- node[xshift = 0.65cm]{{\small \twoline{by}{Thm.\ \ref{thm:bases-proportional-2}}}}
  (rog-ram-var);
  \draw [open triangle 45 - open triangle 45, dashed]
  (nonuple1) -- node[xshift = -0.55cm]{\tiny Qtn.\ \ref{qtn:characters}}
  (andrews);
  \draw [triangle 45 - triangle 45]
  ([xshift = -1cm]twisted-rog-ram.south) -- node[yshift = -0.4cm]{\small \twoline{using \S \ref{subsec:canonical-basis-vectors},}{Lem.\ \ref{lem:m-identities}}}
  (proportionality1.east);
  \draw [triangle 45 - triangle 45]
  ([xshift = -1cm]twisted-rog-ram-var.north) -- node[yshift = 0.4cm]{\small \twoline{using \S \ref{subsec:canonical-basis-vectors},}{Lem.\ \ref{lem:m-identities}}}
  (proportionality2.east);
  \draw [- triangle 45]
  (preliminary.west) |-
    ++(-1.18cm,0) -- 
    ++(-1.18cm,0) |-
    ++(0, -10.4cm)
  %([xshift=-2.4cm, yshift=-6.5cm] preliminary.west) |- ++(0, -4.1cm)
  -|
  node[xshift=-2.4cm, yshift = 0.3cm]{\small \twoline{using}{Lem.\ \ref{lem:t-identities}}}
  ([xshift=-0.2cm] twisted-rog-ram-var.south);
  \draw [-]
  (change-of-basis.east) -|
  ([xshift=0.5cm] slater.east);
  \draw [- triangle 45]
  (slater.east) -|
  ++(0.5cm, 0)
  |-
  ([yshift=0.1cm] eta-functions.east);
  \draw [- triangle 45]
  ([xshift=1cm] high-order.south) |-
  ([yshift=-0.1cm] eta-functions.east);
  \draw [-]
  ([xshift=1cm] nonuple2.south) |-
  ([xshift=-1.65cm, yshift=-0.5cm] high-order.south);
  \draw[-]
  ([xshift=-1.4cm, yshift=-0.5cm] high-order.south) --
  ([xshift=1cm, yshift=-0.5cm] high-order.south);
  \draw[-]
  (preliminary.south) |-
  ([xshift=1.2cm, yshift=-0.35cm] preliminary.south);
  \draw[-]
  ([xshift=1.4cm, yshift=-0.35cm] preliminary.south) -|
  (multiplication.south);
  \draw[-]
  ([xshift=-0.8cm, yshift=-0.35cm] multiplication.south) --
  ([xshift=-0.8cm, yshift=-0.7cm] multiplication.south);
  \draw[triangle 45 -]
  (septuple.north) |-
  ([xshift=1.2cm, yshift=0.6cm] septuple.north);
  \draw[- triangle 45]
  ([xshift=1.4cm, yshift=0.6cm] septuple.north) -|
  (octuple.north);
  \draw[- triangle 45]
  ([yshift=0.6cm] octuple.north) -|
  (nonuple1.north);
  \draw[- triangle 45]
  ([yshift=0.6cm] nonuple1.north) -|
  (nonuple2.north);
  \draw[-]
  ([yshift=0.6cm] nonuple2.north) -|
  ([xshift=1.25cm, yshift=0.6cm] nonuple2.north);
  \draw[- triangle 45]
  ([xshift=1.52cm, yshift=0.6cm] nonuple2.north) -|
  (high-order.north);
  \draw[-]
  ([xshift=1cm] nonuple1.south) |-
  ([xshift=1cm, yshift=-0.5cm] nonuple2.south);
  \draw[-]
  ([xshift=1cm,yshift=-0.7cm] twisted-rog-ram.south)
  --
  ([xshift=1cm,yshift=0.7cm] twisted-rog-ram-var.north);
  \draw[- triangle 45]
  ([xshift=1cm] twisted-rog-ram.south)
  |-
  (more-rog-ram.west);
  \draw[- triangle 45]
  ([xshift=1cm] twisted-rog-ram-var.north)
  |-
  (mock-theta.west);
  \draw[- triangle 45]
  (formulae.east) -|
  ++(1.57cm, 0)
  |-
  ([xshift=0.2cm, yshift=-0.34cm] twisted-rog-ram-var.south)
  --
  ([xshift=0.2cm] twisted-rog-ram-var.south);
  \draw[-]
  ([xshift=-0.4cm] rogers.south)
  |-
  ([xshift=1.5cm, yshift=-0.5cm] rog-ram.south);
  \draw[- triangle 45]
  ([xshift=1.75cm, yshift=-0.5cm] rog-ram.south)
  -|
  (slater.north);
  \draw [- open triangle 45, dashed]
  (fractional.east) -- ++(1.3cm, 0) 
  node[xshift=-0.5cm][yshift=-0.9cm]{\tiny Qtn.\ \ref{qtn:double-twist}}
  |-
  (mock-theta.east);
\end{tikzpicture}
\end{center}
\caption{Relationships between results (\emph{arrows show logical implications or equivalences; green marks new results, to the best of the author's knowledge; `2-var.' $=$ `two-variable'}).}
\label{fig:layout}
\end{figure}
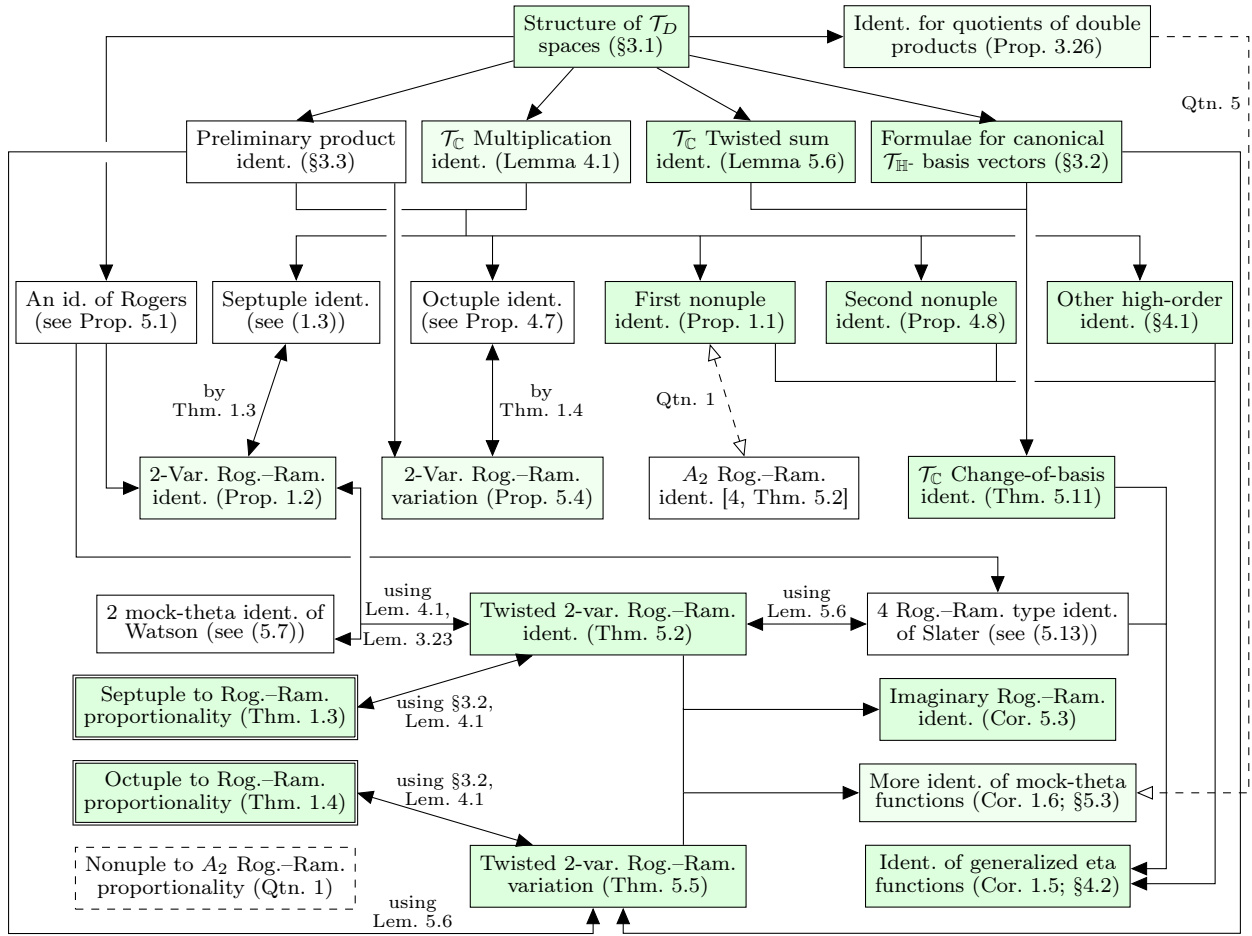
\FloatBarrier

In \cref{fig:layout}, if an arrow has multiple tails, the meaning is that all of the tails collectively imply each of the heads. If an arrow has two-sided heads (and no tails), then the heads facing in one direction are together equivalent to the heads facing in the other direction. So for instance, by going from left to right through \cref{thm:twisted-rog-ram} (which is ultimately equivalent to \cref{thm:bases-proportional}), one obtains a new proof of four Rogers--Ramanujan type identities from Slater's famous list \cite[(94),(96),(98),(99)]{slater1952further}. Going backwards, one has a proof that these four identities of Slater imply the Rogers--Ramanujan identities \cref{eq:rog-ram} and several mock theta identities. Also, some nodes in \cref{fig:layout} are colored with a lighter green to indicate that the results therein are only partly new, or that closely related results have appeared in literature, albeit in different formulations.

\iffalse
\begin{remark}
The seminal work of Slater \cite{slater1952further} and other authors \cite{laughlin2019rogers,sills2007identities,andrews1999a2} on Rogers--Ramanujan type identities is based mostly on transformation formulae such as Bailey's lemma. Following Rogers \cite{rogers1893second}, in \cref{sec:rog-ram} we focus instead on two-variable generalizations of these identities that lie within suitable $T_D$ function spaces.
\end{remark}
\fi

\subsection{Notation} \label{subsec:notation}
Denote $\C^\times := \C \setminus \{0\}$, $\H^+ := \{z \in \C : \Im z > 0\}$, $\H^- := \{z \in \C : \Im z < 0\}$. Throughout the paper we keep $\tau \in \H^+$ fixed and let $z \in \C$ vary, and write $x = e^{2\pi i z}$, $q = e^{2\pi i \tau}$ and $q^r = e^{2\pi i r \tau}$ for $r \in \R$. Given a domain $D$ (an open connected subset of $\C$), denote by $\Hol(D)$ the space of all holomorphic functions on $D$; if $D$ is closed under $\Z$-translation, denote by $\Hol(D/\Z)$ the space of all functions in $\Hol(D)$ with period $1$. Given $w \in \C$ and $u \in \Z \setminus \{0\}$, define the linear maps $T_w : \Hol(D) \to \Hol(D - w)$ and $S_u : \Hol(D) \to \Hol(u^{-1} D)$ by
\[
    (T_w f)(z) := f(z+w), \qquad\qquad (S_u f)(z) := f(u z),
\]
which are a translation and a scaling of the function's argument. Since $u \in \Z \setminus \{0\}$, these also induce natural maps $T_w : \Hol(D/\Z) \to \Hol((D-w)/\Z)$ and $S_u : \Hol(D/\Z) \to \Hol(u^{-1}D/\Z)$ when $D$ is $\Z$-invariant. Given such $D$ and a \emph{meromorphic} function $f$ on $\C/\Z$, $f \not\equiv 0$, we denote
\begin{align}
    \mT_D(f) &:= \left\{F \in \Hol(D/\Z)\ :\ F = f \cdot T_\tau F \text{ on } D \cap (D - \tau)\right\}, \label{eq:T-spaces} \\
    \mS_D(f) &:= \left\{F \in \Hol(D/\Z)\ :\ F = f \cdot S_{-1} F \text{ on } D \cap (-D)\right\}, \label{eq:S-spaces}
\end{align}
where $F = f \cdot T_\tau F$ and $F = f \cdot S_{-1}F$ are understood as equalities of meromorphic functions. We regard $\mT_D(f)$ and $\mS_D(f)$ as complex vector spaces; this generalizes the notation used in \cref{sec:intro}. 

Next, we extend the notation $(x; q)_n := \prod_{k=0}^{n-1} \left(1 - xq^k\right)$ from \cref{sec:intro} to all $n \in \Z$ by setting
\begin{equation} \label{eq:finite-product-negative}
    \forall n < 0, \qquad\quad \frac{1}{(x; q)_n} := \frac{(q^n x; q)}{\left(x; q\right)}
    =
    \left(q^n x; q\right)_{-n}.
\end{equation}
The point of this convention is to extend the relation $(x; q)_{n+1} = (1-x)(qx; q)_n$ to all $n \in \Z$. As before, we write $(x; q) := \lim_{n \to \infty} (x; q)_n$, implying that $(x; q) \in \mT_\C(1-x)$. The double product $\la x; q\ra := (x; q) (q/x; q)$ is commonly referred to as a modified theta function, and often found in literature as $\la x; q \ra_\infty$ \cite{chu2007unification} or $\theta(x; q)$ \cite[(11.2.1)]{gasper2004basic}. It shortly follows that $\la \alpha x^d; q^d \ra \in \mT_\C\left(-\alpha x^d\right)$ for any $\alpha \in \C^\times$, a symmetry which is crucial to most proofs of product identities. But more surprisingly, if $\alpha = -sq^b$ for $s \in \{\pm 1\}$ and $b \in \Z$ such that $d \mid 2b$, a short computation left to the reader shows that $\la \alpha x^d; q^d \ra$ also satisfies a simple $\mS$-type symmetry recorded in \cref{tbl:spaces}.
%\vspace{-0.4cm}
\begin{table}[ht]
\captionsetup{width=\linewidth}
\centering
\renewcommand{\arraystretch}{1.5}
\begin{tabular}{c|c|c|c|c|c}
Function
& $x^d$ 
& $\left(\alpha x^{d}; q^d\right)$ 
& $\left(\beta q^d x^{-d}; q^d\right)^{-1}$ 
& $\la -sq^b x^d; q^d \ra$
& $\vect{sq^cx^d}{k} + t\vect{sq^cx^d}{\ell}$ %{d-2c-k}$
\\\hline
$\mT$-Space 
& $\mT_\C\left(q^{-d} \right)$
& $\mT_\C\left(1 - \alpha x^d\right)$
& $\mT_\Hminus\left(1 - \beta x^{-d}\right)$
& $\mT_\C\left(sq^bx^d \right)$
& $\mT_\C\left(sq^cx^d\right)$ 
\\\hline
$\mS$-Space 
& $\mS_\C\left(x^{2d} \right)$
& --
& --
& $\mS_\C\left(s^{1-2b/d} x^{d-2b} \right)$
& $\mS_\C\left(t x^{d-2c}\right)$
\end{tabular}
\caption{Function spaces ($b, c, d, k, \ell \in \Z$, $1 \le d \mid 2b$, $k + \ell = d - 2c$; $\alpha, \beta \in \C^\times$, $|\beta| \ge |q|^{-d}$; $s, t \in \{\pm 1\}$).}
\label{tbl:spaces}
\end{table}

In taking products of such functions it is fairly easy to bookkeep the spaces $\mT_D(f)$ and $\mS_D(f)$ involved, by simply multiplying the factors $f$ therein.
One can think of these factors as units of measurement, which are multiplied in the left-hand sides of product identities (such as \cref{eq:triple}--\cref{eq:nonuple}), and which must be homogeneous in the sums from the right-hand-sides. Also, the $\mS_\C$-symmetries of the double infinite products $\la sq^b x^d; q^d \ra$ explain why the basis vectors $\vect{\pm q^c x^d}{k}$ pair up in the quintuple, septuple and nonuple identities as in the rightmost column of \cref{tbl:spaces}. For completion, we also state three easy manipulations of infinite products as in \cref{tbl:spaces}, which may be used implicitly throughout the paper: for any positive integer $d$, one has
\[
    \big(x^d; q^d\big) 
    =
    \prod_{\zeta^d = 1} \left(\zeta x; q\right),
    \qquad\qquad 
    \big\langle x^d; q^d\big\rangle 
    =
    \prod_{\zeta^d = 1} \la \zeta x; q\ra,
    \qquad\qquad 
    \big( x; q \big)
    =
    \prod_{j=0}^{d-1} 
    \big(q^j x; q^d \big)
    ,
\]
where $\zeta$ ranges over the $d$th complex roots of unity.

The notation $\vect{f}{k}$ will be used for certain canonical elements of $\mT_\Hminus(f)$, generalizing the functions $\vect{\alpha x^d}{k} := \sum_{n \in \Z} \alpha^n q^{d\binom{n}{2} + kn} x^{dn+k}$ from \cref{sec:intro}; see \cref{prop:canonical-basis-vectors}. By abuse of notation we may write an expression depending on $z$ in place of $f$, e.g.\ $\vect{f}{0} = \vect{qx^2-x}{0}$ when $f(z) = qx^2 - x$. As a word of caution, this notation does \emph{not} obey the rule of substitution in $z$, so for instance $\vect{qx^2-x}{0}(2z) \neq \vect{qx^4-x^2}{0}$; the expression before the semicolon indicates a $\tau$-quasiperiodicity factor (e.g., $\vect{qx^2-x}{0}(z) = (qx^2-x) \cdot \vect{qx^2-x}{0}(z+\tau)$), not an argument of the function.

Given a nonempty open horizontal strip $S \subset \C$, a function $F \in \Hol(S/\Z)$ and $n \in \Z$, write
\[
    \hat{F}(n) := \int_w^{w+1} F(z) e^{-2\pi i n z} dz,
    \qquad\qquad 
    F(z) = \sum_{n \in \Z} \hat{F}(n) x^n,
\] 
for any $w \in S$ (the choice of $w$ is irrelevant by contour-shifting). The Fourier series of $F$ converges absolutely and locally uniformly in $S$, so in particular we can multiply the Fourier series of two such functions using the natural sum rearrangements. If $F(z)$ is a Laurent polynomial in $x$, we write $\deg_x F$ for its degree in $x$, i.e.\ the largest $n$ such that $\hat{F}(n) \neq 0$. 

Finally, we write $\binom{n}{k} := n(n-1)\cdots(n-k+1)/k!$ for $n \in \Z$ and $k \ge 0$, and $\one_{P}$ for the truth value of a proposition $P$ (e.g., $\one_{x < y}$ equals $1$ whenever $x < y$, and $0$ otherwise). We leave the notations concerning generalized eta functions and mock theta functions to \cref{subsec:generalized-eta,subsec:mock-theta}.

\section{Structure of the relevant function spaces and first applications} \label{sec:preliminaries}

\subsection{Dimensionality bounds} \label{subsec:dim-bounds}
We start by generalizing the spaces $\mT_D(f)$.

\begin{definition}[Extended quasiperiodic function spaces] \label{def:T-spaces}
Let $D \subset \C$ be a $\Z$-invariant domain, $m$ a positive integer, and $f_1, \ldots, f_m$ meromorphic functions on $\C/\Z$ (i.e., on $\C$ with period $1$), with $f_m \not\equiv 0$ (i.e., $f_m$ is not the zero function). We define the complex vector space
\[
    \mT_D(f_1, \ldots, f_m) := \left\{F \in \Hol(D/\Z) : F = f_1 \cdot T_\tau F + \cdots + f_m \cdot T_{m\tau} F \right\},
\] 
where the equality holds as an identity of meromorphic functions on $D \cap (D - \tau) \cap \cdots \cap (D - m\tau)$. More generally, given an $m \times m$ matrix $M(z)$ whose entries are meromorphic functions on $\C/\Z$ such that $\det M(z) \not\equiv 0$, let
\[
    \mT_D(M) := 
    \left\{ 
    \vec{F} \in \Hol(D/\Z)^m
    :
    \vec{F} = M \cdot T_\tau \vec{F}
    \right\},
\]
where the equality holds as an identity of $m$ meromorphic functions on $D \cap (D - \tau)$, and $(T_\tau \vec{F})(z) = \vec{F}(z+\tau)$ as before. In particular, we have a linear bijection
\[
    \mT_D\left(f_1, \ldots, f_m\right) 
    \ 
    \cong
    \
    \mT_{D \cap (D - \tau) \cap \cdots \cap (D - (m-1)\tau)}
    \begin{psmall}
    f_1 & f_2 & \cdots & f_{m-1} & f_m
    \\
    1 & 0 & \cdots & 0 & 0 
    \\
     &  & \ddots &  &   
    \\
    0 & 0 & \cdots & 1 & 0
    \end{psmall}
    \qquad 
    \text{by} \quad
    F \mapsto 
    \begin{psmall}
    F \\
    T_\tau F \\
    \vdots \\
    T_{(m-1)\tau} F
    \end{psmall}.
\]
While much of this paper is concerned with the spaces $\mT_D(f)$ (in fact, \cref{sec:higher-prod-id} only works with $\mT_D\left(\alpha x^d\right)$), we will encounter more general spaces $\mT_D(f_1, \ldots, f_m)$ and $\mT_D(M)$ in \cref{sec:rog-ram}.
\end{definition}

%\begin{remark} 
%The set $D$ in \cref{def:T-spaces} will usually be either $\C$, $\H^-$, $\H^+$, or a horizontal strip. %Although we will mostly be using $m = 1$, we will encounter the case $m > 1$ in several identities. %\fix{finitized Jacobi, Rogers, the one in the proof of the determinant, Rogers--Ramanujan functions (next subsection)}
%\end{remark}

\begin{example}[Two-variable Rogers--Ramanujan fraction] \label{ex:two-var-rog-ram}
Consider the entire $1$-periodic function $F(z) := \sum_{n \ge 0} q^{n^2} (q; q)_n^{-1} x^n$ (this is $\chi(x/q)$ from \cite[p.~329]{rogers1893second}). A short computation shows that
\[
    \forall z \in \C : \quad F(z) = F(z+\tau) + qxF(z+2\tau),
\]
so $F = T_\tau F + qxT_{2\tau} F$ and thus $F \in \mT_\C\left(1, qx\right)$. This vector space turns out to be one-dimensional, which can be used to prove an identity due to Rogers generalizing the Rogers--Ramanujan identities in \cref{eq:rog-ram}; see \cref{prop:rogers}. By iterating this functional equation in matrix form, one can also recover the Rogers--Ramanujan continued fraction identity \cite[p.~328 (4)]{rogers1893second},
\[
    \begin{pmatrix}
    F(z)
    \\
    F(z+\tau) 
    \end{pmatrix}
    =
    \prod_{n \ge 1}
    \begin{pmatrix}
    1 & q^n x 
    \\
    1 & 0
    \end{pmatrix}
    \cdot
    \begin{pmatrix}
    1 
    \\
    1
    \end{pmatrix}
    \qquad 
    \Longleftrightarrow 
    \qquad 
    \frac{F(z)}{F(z+\tau)} = 1 + \frac{qx}{1 + \frac{q^2x}{1 + \frac{q^3 x}{1 + \cdots}}}.
\]
where the infinite product of matrices above is computed from the left to the right.
\end{example}

\begin{example}[Line bundles and theta functions]
Take $D = \C$ and $m = 1$ in \cref{def:T-spaces}, and let $f$ be any invertible function in $\Hol(\C/\Z)$. For $m, n \in \Z$ and $z \in \C$, define $\phi_{m\tau + n}(z)$ as $\prod_{k=0}^{m-1} f(z+k\tau)^{-1}$ if $m \ge 0$, and as $\prod_{k=-m}^{-1} f(z+k\tau)$ otherwise. Then $(\phi_\gamma)_{\gamma \in \Gamma}$ form a system of multipliers for the lattice $\Gamma := \Z\tau + \Z$, meaning that they are holomorphic invertible and they satisfy the cocycle condition $\phi_{\gamma+ \delta}(z) = \phi_\gamma(z+\delta) \phi_\delta(z)$ (see, for example, \cite[\S 2.3]{beauville2013theta}). Such a system induces a holomorphic line bundle $L \cong (\C \times \C)/\Gamma$ over $\C/\Gamma$, where $\gamma \in \Gamma$ acts on $\C \times \C$ by $\gamma(z, t) = (z + \gamma, \phi_\gamma(z)t)$; a section of this line bundle has the form \[
    \C/\Gamma \to L \cong \left(\C \times \C\right)/\Gamma, 
    \qquad\qquad 
    z \mapsto (z, F(z)),
\]
where $F \in \Hol(\C)$ satisfies $F(z + \gamma) = \phi_\gamma(z) F(z)$, and this reduces to $F(z) = f(z)F(z+\tau)$ by our choice of $\left\{\phi_\gamma\right\}_{\gamma \in \Gamma}$. Hence the sections of $L \to C$ are canonically identified with the elements $F$ of $\mT_\C(f)$, and these are called the \emph{theta functions} associated to $L$. Canonical examples are Jacobi's theta functions (see \cite{jacobi1829fundamenta} or \cite[p.~464]{whittaker2020course}), and more generally the functions $\vect{\alpha x^d}{k}$.

While the triple, quintuple, septuple, etc.\ products all lie in such vector spaces induced by invertible functions $f$, allowing $f$ to have zeros at the cost of reducing the domain of holomorphicity $D$ of $F$ is essential for stating and proving results like \cref{prop:2var-rog-ram}, \cref{thm:bases-proportional} or \cref{thm:bases-proportional-2}. If $D$ contains a horizontal strip of length $\ge \Im \tau$ (thus a copy of $\C/\Gamma$), the more general spaces $\mT_D(f)$ correspond similarly to \emph{meromorphic} line bundles, while the spaces $\mT_D(M)$ (where $M$ is an $m \times m$ matrix) correspond to \emph{vector} bundles of rank $m$ (of the form $E \to \C/\Gamma$ where $E \cong \left(\C \times \C^m\right)/\Gamma$) and \emph{generalized} (or \emph{non-abelian}) theta functions (see \cite[\S 6.2]{beauville2013theta} or \cite{gunning1982generalized}). It can be helpful, however, to consider smaller domains $D$ not containing any copy of $\C/\Gamma$, as we illustrate \cref{prop:fractional}.
%https://math.unice.fr/~beauvill/pubs/thetaon.pdf
%https://www.math.purdue.edu/~arapura/preprints/abelian.pdf
%https://pdfs.semanticscholar.org/9aa8/eb933ee06e16530a523c5d84af711ac5f2ca.pdf, page 15
For the reader familiar with this language, the properties in the following lemma will come naturally; e.g., \cref{lem:basic}.(iii) is motivated by the fact that multiplying a vector bundle by a line bundle yields back a vector bundle of the same rank.
%\fix{CHECK (as a remark) IF YOU CAN APPLY RIEMANN-ROCH TOGETHER WITH THE BOUND TO GET $\dim \mT_\C(\alpha x^d) = d$!! ...what's the canonical divisor? RRoch may give you 0 = 0... check all occurences of RRoch in that paper... or maybe ask Tao!}
\end{example}

\begin{lemma}[Basic facts about $\mT_D$] \label{lem:basic}
Let $F \in \mT_D(f_1, \ldots, f_m)$ and $G \in \mT_D(g)$ be as in \cref{def:T-spaces}.
\begin{enumerate}[label=(\roman*).]
    \item If $E$ is another $Z$-invariant domain and $\emptyset \neq D \subset E$, then $\mT_E(f_1, \ldots, f_m) \subset \mT_D(f_1, \ldots, f_m)$,  where a function in $\Hol(E/\Z)$ is uniquely identified with its restriction to $\Hol(D/\Z)$.
    \item For $w \in \C$, one has $ T_w F \in \mT_{D-w}\left(T_w f_1, \ldots, T_w f_m\right)$.
    \item Letting $g_n := \prod_{j=0}^{m-1} T_{j\tau} g$, one has
    $FG \in \mT_D\left(f_1 g_1, f_2 g_2, \ldots, f_m g_m\right)$;
    so if $m = 1$, then $FG \in \mT_D(f_1g)$. Also, if $G$ has no zeros in $D$, then $G^{-1} \in \mT_D(g^{-1})$.
    \item If $u \in \Z$, $u \ge 1$ and $D \cap T_{-\tau} D \cap \cdots \cap T_{-u\tau} D \neq \emptyset$, then
    \[
        S_u G \in \mT_{u^{-1}D}\Bigg(\prod_{j=0}^{u-1} S_u T_{j\tau}  g\Bigg).
    \]
    \item If $D$ contains an open horizontal strip of width $> m\cdot \Im\tau$, then $F$ has a meromorphic continuation to all of $\C$, which still satisfies $F = F_1 \cdot T_\tau F + \cdots + f_m \cdot T_{m\tau} F$.
    \item If $D$ contains an open horizontal strip of width $> \Im \tau$, and $0 \neq H \in \mT_D(g)$, then $G/H$ extends to an elliptic function on $\C$ with periods $1, \tau$.
\end{enumerate}
\end{lemma}

\begin{proof}
Statements (i) and (ii) are immediate, while (iii) and (iv) follow by repeatedly applying (ii) for $w = j\tau$. For (v), let $S \subset D$ be an open horizontal strip of width $> m \cdot \Im \tau$; then we can iteratively define $F := f_1 \cdot T_\tau F + \cdots + f_m \cdot T_{m\tau} F$ on the strips $S + n\tau$ for $n \in \{1, 2, 3, \ldots\}$, and
\[
    F := T_{-m\tau} \frac{F - f_1 \cdot T_\tau F - \cdots - f_m T_{m\tau} F}{f_m}
\]
on $S - n\tau$ for $n \in \{1, 2, 3, \ldots\}$; note that $F$ is well-defined on $\C$ by the uniqueness of meromorphic continuation. Finally, (vi) follows from (v) applied to $G$ and $H$, and then (iii) applied for $D' = \C \setminus \{\textnormal{poles of $G$, $H$ or $H^{-1}$}\}$ (one has $H^{-1} \in \mT_{D'}(g^{-1})$, and so $G/H \in \mT_{D'}(1)$).
\end{proof}

\begin{remark} 
By \cref{lem:basic}.(v), we can treat the elements of $\mT_D(f_1, \ldots, f_m)$ as meromorphic functions on $\C$ if $D$ is wide enough; the same is true in more general spaces $\mT_D(M)$, where the continuation requires inverting the matrix $M$. The restriction of holomorphicity on $D$, however, will be crucial.
\end{remark}

\begin{example}[Spaces given by constants] \label{ex:constants}
Suppose a domain $D$ contains a horizontal strip $S$ of length $> \Im \tau$. Note that $\mT_D(1) = \left\{ F \in \Hol(D) : F = T_1 F = T_\tau F \text{ in } D \cap (D - \tau)\right\}$.

The proof of \cref{lem:basic}.(ii) then shows that any function $F \in \mT_D(1)$ can be continued to an \emph{entire} elliptic function on $\C$, thus a constant. More generally, if $f_1 \equiv \alpha$ is any nonzero constant, then for $F \in \mT_D(\alpha)$, considering the Fourier series of $F(z) = \alpha F(z+\tau)$ in $S \cap (S - \tau) \neq \emptyset$ yields 
\[
    \hat{F}(n) = \alpha q^n \hat{F}(n)
\] 
for any $n \in \Z$. So if $\alpha = q^{-n}$ for some $n \in \Z$, we find that $\mT_D(\alpha)$ is one-dimensional, spanned by $x^n$; otherwise, $\mT_D(\alpha) = \{0\}$. We give a more general upper bound for $\dim \mT_D$ in \cref{prop:upper-bounds}.
\end{example}

%Now consider the case when $f_1, \ldots, f_m$ are rational functions in $x$. By rescaling numerators, we can assume that the denominators of $f_1, \ldots, f_m$ are all equal, and that at least one polynomial in the numerators or common denominator of $f_1, \ldots, f_m$ has a nonzero constant term. The next proposition gives a formal statement about $\mT_D(f_1, \ldots, f_m)$ in this context.

\begin{proposition}[Upper bounds for $\dim \mT_D$] \label{prop:upper-bounds}
Let $P_0(z), P_1(z), \ldots, P_m(z)$ be complex polynomials in $x = e^{2\pi i z}$ with $P_0 \neq 0$ and $\big(\hat{P_0}(0), \ldots, \hat{P_m}(0)\big) \neq (0, \ldots, 0)$, and let $d := \max(\deg_x P_0, \ldots, \deg_x P_m)$. For $1 \le j \le m$, take $f_j := P_j/P_0$. Suppose $D \subset \C$ is a $\Z$-invariant domain, containing an open horizontal strip $S$ of width $> m \cdot \Im \tau$. Letting $V := \mT_D\left(f_1, \ldots, f_m\right)$, the following hold true.
\begin{enumerate}[label=(\roman*).]
    \item
One has $\dim V \le d + m$. 
    \item
For $n \in \Z$, denote $A_n := P_0 - q^n P_1 - \cdots - q^{mn} P_m$. Assume additionally that $\exists n_0 \in \Z$,
\[
\forall n \ge n_0 + d :\  \hat{A_n}(0) \neq 0  \qquad \text{ and } \qquad \forall n < n_0 :\  \hat{A_n}(d) \neq 0.
\]
Then one has $\dim V \le d$, with equality iff $\forall k \in \{0, 1, \ldots, d-1\}$, $\exists F_k \in V$ such that
\[
    \hat{F_k}(j) = \one_{j = n_0 + k} \quad \text{ in } S, \quad\quad\hspace{0.2cm}  \forall j \in \{0, 1, \ldots, d-1\},
\]
In this case, $\{F_0, \ldots, F_{d-1}\}$ is a (uniquely determined) basis of $V$.
\end{enumerate}
\end{proposition}

\begin{example} \label{ex:constants-ctd}
For $m = 1$, $P_0(z) = 1$ and $P_1(z) = \alpha \in \C^\times$, part (i) implies that $\dim \mT_D(\alpha) \le 1$ whenever $D$ contains a horizontal strip of width $> \Im \tau$. If $\alpha \not\in \left\{q^n : n \in \Z\right\}$, then in part (ii) we have $A_n(z) = 1 - \alpha q^n$, so any choice of $n_0$ yields $\dim \mT_D(\alpha) = 0$. This recovers \cref{ex:constants}.
\end{example}

\begin{example} \label{ex:two-var-rog-ram-ctd}
For $m = 2$, $P_0(z) = P_1(z) = 1$ and $P_2(z) = qx$, we have $A_n(z) = 1 - q^n - q^{2n+1}x$. Hence using $n_0 = 0$ in \cref{prop:upper-bounds}.(ii), we find that $\dim \mT_\Hplus(1, qx) \le 1$. In \cref{ex:two-var-rog-ram} we found a nonzero element $G$ of this space, so it must be the unique such function up to scalars.
\end{example}

\begin{proof}[Proof of \cref{prop:upper-bounds}]
For $F \in \mT_D(f_1, \ldots, f_m)$, we have $F = f_1 \cdot T_\tau F + \cdots + f_m \cdot T_{m\tau} F$ and thus $P_0 \cdot F = P_1 \cdot T_\tau F + \cdots + P_m \cdot T_{m\tau} F$. Taking the Fourier series in $S \cap (S - m\tau) \neq \emptyset$ then yields 
\[
    \forall n \in \Z : \qquad\quad \sum_{k=0}^d \hat{P_0}(k) \hat{F}(n-k)
    =
    \sum_{j=1}^m \sum_{k=0}^d \hat{P_j}(k) q^{j(n-k)} \hat{F}(n-k),
\]
which rearranges (by subtracting the right side from the left side and swapping sums) to
\begin{equation}\label{eq:recursion}
    \sum_{k=0}^d \hat{A_{n-k}}(k) \hat{F}(n-k) = 0.
\end{equation}
If $d = 0$, we find that $\hat{F}(n) = 0$ whenever $\hat{A_n}(0) \neq 0$. So (ii) holds trivially since if all $\hat{A_n}(0) \neq 0$, we must have $F = 0$. For statement (i), note that one can have $\hat{A_n}(0) = 0$ for at most $m$ values of $n$, say among $\{n_1, \ldots, n_m\}$, since $\hat{A_n}(0)$ is a nonzero polynomial in $q^n$ of degree $\le m$. Hence $V$ injects linearly into a subspace of $\C^m$ by the map $F \mapsto \big(\hat{F}(n_1), \ldots, \hat{F}(n_m) \big)$, proving $\dim V \le m$.

Now assume that $d \ge 1$. When $\hat{A_n}(0) \neq 0$, \cref{eq:recursion} implies that $\hat{F}(n)$ is a linear function of $\hat{F}(n-1), \ldots, \hat{F}(n-d)$. Similarly, when $\hat{A_n}(d) \neq 0$, $\hat{F}(n)$ is a linear function of $\hat{F}(n+1), \ldots, \hat{F}(n+d)$. Hence under the assumption of (ii), $V$ injects linearly into $\C^d$ by $F \mapsto \big(\hat{F}(n_0), \ldots, \hat{F}(n_0 + d-1)\big)$, proving that $\dim V = d$ with equality iff each standard basis vector of $\C^d$ is attained by this injection; the preimages of these vectors correspond to $F_0, \ldots, F_{d-1}$.

To prove (i), note as before that $\hat{A_n}(0)$ and $\hat{A_n}(d)$ can each be zero for at most $m$ values of $n$ (since $\hat{A_n}(d)$ is also a polynomial in $q^n$ of degree $\le m$, which is not constantly $0$ by the definition of $d$). Say that $\hat{A_n}(d)$ has zeros only if $n \in \{n_1, \ldots, n_m\}$, and let $N \in \Z$ be large enough such that $\hat{A_n}(0) \neq 0$ for all $n \ge N+d$. Then by our previous reasoning, $V$ injects linearly into $\C^{d+m}$ by $F \mapsto \big(\hat{F}(N), \ldots, \hat{F}(N+d-1), \hat{F}(n_1), \ldots, \hat{F}(n_m) \big)$, proving $\dim V \le d + m$.
\end{proof}

\begin{remark}
The bound in \cref{prop:upper-bounds}.(i) may be sharpened by determining the zeros of $\hat{A_n}(0)$ and $\hat{A_n}(d)$ in particular cases. Also, the linear relation in $\cref{eq:recursion}$ gives a formula for the Fourier coefficients $\hat{F_k}(n)$ (for $k \in \{0, \ldots, d-1\}$ and $n \in \Z$) in terms of finite matrix products; however, applications will require alternative formulae for $F_k$, which we give in \cref{subsec:canonical-basis-vectors}.
\end{remark}

It is now convenient to briefly study the relationship between the $\mT_D$ and $\mS_D$ spaces; recall \cref{eq:S-spaces}.
\begin{lemma}[Basic facts about $\mS_D$] \label{lem:basic-S}
Let $D$ be a $\Z$-invariant domain and $f, g$ be meromorphic functions on $\C/\Z$.
\begin{enumerate}[label = (\roman*).]
    \item If $F \in \mS_D(f)$ and $G \in \mS_D(g)$, then $FG \in \mS_D(fg)$. 
    \item If $D \cap (-D) \neq \emptyset$ and $\mS_D(g) \neq \{0\}$, then one has $g(z)g(-z) = 1$ wherever defined.
    \item If $D \cap (-D) \cap (D-\tau) \cap (-D-\tau) \neq \emptyset$ and $\mT_D(f) \cap \mS_D(g) \neq \{0\}$, then wherever defined,
    \[
        f(z)f(-z-\tau) = g(z)g(-z-\tau).
    \]
    %\item Assuming the hypothesis of \cref{prop:upper-bounds}.(ii) for $n_0 = \cdots$, one has
    %\[
    %    \dim\left(\mT_D(f_1, \ldots, f_m) \cap \mS_D(x^?)\right) \le ?, \qquad\text{ and }\qquad \dim\left(\mT_D(f_1, \ldots, f_m) \cap \mS_D(-x^?)\right) \le ?.
    %\]
\end{enumerate}
\end{lemma}
\begin{proof}
Statement (i) is clear, and (ii) follows by iterating the functional equation of any $F \in \mS_D(f)$ with $F \neq 0$; note that having $f(z) f(-z) = 1$ on a nonempty open set extends to $f(z)f(-z) = 1$ everywhere. For (iii), let $F \in \mT_D(f) \cap \mS_D(g) \setminus \{0\}$, and use these two symmetries of $F$ to obtain
\[
\begin{aligned}
    g(z)F(-z) = F(z) &= f(z)F(z+\tau) \\
    &= f(z)g(z+\tau) F(-z-\tau)
    =
    f(z)g(z+\tau)f(-z-\tau)F(-z),
\end{aligned}
\]
as an identity of meromorphic functions in $z \in D \cap (-D) \cap (D-\tau) \cap (-D-\tau)$. Since this is a nonempty open set and $F$ is not the zero function, we can cancel $F(-z)$ and multiply by $g(-z-\tau)$ to obtain the desired claim (note that $g(z+\tau)g(-z-\tau) = 1$ as meromorphic functions by (ii)). 
\end{proof}

\begin{remark}
A choice of meromorphic functions $f$, $g$ which satisfy the conditions in \cref{lem:basic-S}.(ii)-(iii) is given by $f(z) = \pm q^{c} x^d$ and $g(z) = \pm x^{d-2c}$, for $c, d \in \Z$. These come up in the next result.
\end{remark}

\begin{proposition}[Spaces given by monomials] \label{prop:monomial-spaces}
Let $d$ be any integer, and $D$ be a $\Z$-invariant domain containing an open horizontal strip $S$ of width $> \Im \tau$.
\begin{itemize}
\item[(i)] For any $\alpha \in \C^\times$ (which may depend on the fixed $q$), one has
\[
    \dim \mT_D\big(\alpha x^d\big) = 
    \begin{cases} 
    d, \qquad &d > 0, \\
    \one_{\alpha \in \left\{q^n : n \in \Z \right\}}, \qquad &d = 0, \\
    0, \qquad &d < 0.
    \end{cases}
\]
Moreover, for $d \ge 0$, any $F \in \mT_\C\left(\alpha x^d\right) \setminus \{0\}$ has exactly $d$ zeros (counting multiplicities) in every fundamental region $\{w - t - u\tau : t, u \in [0, 1) \}$, for $w \in \C$. 
\item[(ii)] For $c \in \Z$ and $s, t \in \{\pm 1\}$,
\begin{equation} \label{eq:dimension}
    \dim \left( \mT_\C \big(sq^c x^d \big) \cap \mS_\C \big(tx^{d-2c}\big) \right) = 
    \max\left( 1 + \left\lfloor \frac{d}{2} \right\rfloor 
    -
    \one_{2 \mid d} \one_{t = -1}
    -
    \one_{st = -1} ,\ 0\right).
\end{equation}
So for $d \in \{1, 2\}$, one obtains $\mT_\C\left(s q^cx\right) \subset \mS_\C\left(s x^{1-2c}\right)$ and $\mT_\C\left(q^cx^2\right) \subset \mS_\C\left(x^{2-2c}\right)$. Moreover, all elements of $\mT_\C \big(sq^c x^d \big) \cap \mS_\C \big(tx^{d-2c}\big)$ have common zeros at:
\[
\begin{cases}
z = m\tau + n, &\textnormal{if }\ t = -1, \\
z = m\tau + n + \frac{1}{2}, &\textnormal{if }\ t = (-1)^{d+1}, 
\\
z = \left(m + \frac{1}{2}\right)\tau + \frac{n}{2}, &\textnormal{if }\ st = -1,
\qquad\qquad\qquad \forall m,n \in \Z.
\end{cases}
\]

\end{itemize}
\end{proposition}

\begin{remark}
The content of this proposition generalizes that of Theorems 1.7 and 1.8 from \cite{hickerson1988proof}, which take $D = \C$, $t = -1$, $c = 0$, $d > 0$ and $d$ odd for the second part. The case $d > 0$ from (i) is also generalized by \cite[Theorem 3.5]{beauville2013theta} in the language of line bundles.
\end{remark}

\begin{proof}
For part (i), case $d = 0$ was settled in \cref{ex:constants}. If $d < 0$, suppose that $\dim \mT_D\left(\alpha x^{d}\right) \neq 0$, and let $0 \neq F \in \mT_D\left(\alpha x^{d}\right)$. Then as in \cref{eq:recursion} we have $\hat{F}(n) = \alpha q^{n-d} \hat{F}(n-d)$ for all $n \in \Z$. Fixing a nonzero Fourier coefficient $\hat{F}(n_0)$, we find inductively that $\hat{F}(n_0 - nd)$ grows on the order of $q^{dn^2/2}$, which contradicts the convergence of the Fourier series of $F$ anywhere since $\left\vert q^d\right\vert > 1$.

Now suppose that $d > 0$. Take $m = 1$, $P_0(z) = 1$ and $P_1(z) = \alpha x^d$ in \cref{prop:upper-bounds}, and consider $A_n(z) = 1 - \alpha q^n x^d$. We have $\hat{A_n}(0) \neq 0$ and $\hat{A_n}(d) \neq 0$ for all $n \in \Z$, so \cref{prop:upper-bounds}.(ii) yields that $\dim \mT_D\left(\alpha x^d \right) \le d$, with equality iff there exist $\{F_k\}_{0 \le k < d} \subset \mT_D\left(\alpha x^d \right)$ such that 
\begin{equation} \label{eq:fourier}
    \hat{F_k}(j) = \one_{j=k}, \qquad \forall j, k \in \{0, 1, \ldots, d-1\},
\end{equation}
where the Fourier series are taken in $S$. Such functions are the canonical basis vectors from \cref{sec:intro},
\[
    \left\{\vect{\alpha x^d}{k}\right\}_{0 \le k < d}
    =
    \left\{\sum_{n \in \Z} \alpha^n q^{d\binom{n}{2}+k_jn}x^{dn+k}\right\}_{0 \le k < d},
\]
given by locally uniformly convergent Fourier series in $z \in \C \supset D$. Thus these are entire $1$-periodic functions satisfying $\vect{\alpha x^d}{k}(z) = \alpha x^d \vect{\alpha x^d}{k}(z+\tau)$ and \cref{eq:fourier}, proving that $\dim \mT_D\left(\alpha x^d\right) = d$. In fact, for any $k \in \Z$, $\vect{\alpha x^d}{k}$ has nonzero Fourier coefficients only at multiples of $d$ plus $k$, so the functions $\{\vect{\alpha x^d}{k_j}\}_{0 \le j < d}$ give a basis of $\mT_D\left(\alpha x^d\right)$ for any complete residue system $\{k_0, \ldots, k_{d-1}\}$ modulo $d$. Since we have $\mT_\C\left(\alpha x^d\right) \subset \mT_D\left(\alpha x^d \right)$ by \cref{lem:basic}.(i), the equality of dimensions forces an equality of vector spaces; i.e.\ each function in $\mT_D\left(\alpha x^d \right)$ can be holomorphically continued to $\C$.

Concerning the claim about zeros in (i), the case $d = 0$ is immediate since then $F$ can only be a scalar multiple of $x^n$ for some $n \in \Z$. For $d > 0$, note that $\la -\alpha x^d; q^d \ra \in \mT_\C\left(\alpha x^d\right)$ has zeros whenever $x^d = -\alpha^{-1}q^{nd}$ for $n \in \Z$, thus exactly $d$ zeros in any fundamental region as in part (i). Since the ratio of any two nonzero functions in $\mT_\C\left(\alpha x^d\right)$ is an elliptic function (by \cref{lem:basic}.(vi)), the same is true for any $F \in \mT_\C\left(\alpha x^d\right) \setminus \{0\}$.

For part (ii), the case $d \le 0$ is easily verified, so take $d > 0$. We can also assume WLOG that $c = 0$ due to the linear bijection given by multiplication by $x^c$,
\[
     \mT_\C\big(sq^c x^d \big) \cap \mS_\C\big(tx^{d-2c}\big)
     \xlongrightarrow{x^c} 
     \mT_\C\big(s x^d \big) \cap \mS_\C\big(tx^{d}\big),
\]
which also preserves canonical bases and zeros; note that $x^c \in \mT_\C\left(q^{-c}\right) \cap \mS_\C\left(x^{2c}\right)$. A quick computation shows that $S_{-1}\vect{sx^d}{k} = x^{-d} \vect{sx^d}{d-k}$ for $k \in \Z$ (using that $s \in \{\pm 1\}$), and so $\vect{sx^d}{k} + t\vect{sx^d}{d-k} \in \mS_\C\left(tx^d\right)$ as anticipated in \cref{tbl:spaces}. In fact, we claim that the functions
\begin{equation} \label{eq:pair-basis}
    \left\{\vect{sx^d}{k} + t\vect{sx^d}{d-k}\right\}_{0 \le k \le \lfloor d/2 \rfloor}
\end{equation}
span $\mT_\C\left(sx^d\right) \cap \mS_\C\left(tx^d\right)$. Indeed, any $F \in \mT_\C\left(sx^d\right)$ can be written as $\sum_{k=0}^{d-1} \hat{F}(k) \vect{sx^d}{k}$ and $\sum_{k=1}^d \hat{F}(k) \vect{sx^d}{k}$, by identifying Fourier coefficients in two basis representations of $F$. Hence
\[
    F =
    \frac{1}{2} \sum_{k=0}^{d-1} \left(\hat{F}(k) \vect{sx^d}{k}  + \hat{F}(d-k) \vect{sx^d}{d-k} \right).
\]
Assuming additionally that $F \in \mS_\C\left(tx^d\right)$, we find that $\hat{F}(k) = t\hat{F}(d-k)$ by identifying Fourier series, and thus $F$ is spanned by the sums in \cref{eq:pair-basis}. Since $\vect{sx^d}{k} + \vect{sx^d}{d-k}$ has null Fourier coefficients anywhere else other than at multiples of $d$ plus-minus $k$, it is clear that the vectors in \cref{eq:pair-basis} are linearly independent \emph{provided} that they are nonzero. But to have $\vect{sx^d}{k} + t\vect{sx^d}{d-k} = 0$ we would need in particular that $k \equiv d-k \pmod{d}$, so $k \in \{0, d/2\}$. Treating these two cases separately leads to the two terms subtracted from $1 + \lfloor d/2 \rfloor$ in \cref{eq:dimension}. For $d \in \{1, 2\}$ one has $1 + \lfloor d/2 \rfloor = d$, so for the right choices of $s, t \in \{\pm 1\}$ one can force an equality of vector spaces $\mT_\C\left(sq^c x^d \right) \cap \mS_\C\left(tx^{d-2c} \right) = \mT_\C\left(sq^c x^d \right)$, proving the claimed inclusions.

Finally, for the statement about common zeros in (ii), consider the maps
\begin{center}
\begin{tikzcd} [column sep = huge, row sep = tiny]
    \mT_\C\big({-s}x^{d-1}\big) \cap \mS_\C\big(x^{d-1}\big)
    \arrow{r}{\la x; q \ra}
    &
    \mT_\C\big(sx^d\big) \cap \mS_\C\big({-x^d}\big),
    \\
    \mT_\C\big(sx^{d-1}\big) \cap \mS_\C\big({(-1)^{d+1}x^{d-1}}\big)
    \arrow{r}{\la -x; q \ra}
    &
    \mT_\C\big(sx^d\big) \cap \mS_\C\big({(-1)^{d+1}x^d}\big),
    \\
    \mT_\C\big({-s}x^{d-2}\big) \cap \mS_\C\big({-s}x^{d-2}\big)
    \arrow{r}{x\la qx^2; q^2 \ra}
    &
    \mT_\C\big(sx^d\big) \cap \mS_\C\big({-sx^d}\big),
\end{tikzcd}
\end{center}
where we used that $\la \pm x; q\ra \in \mT_\C\left(\mp x\right) \cap \mS_\C\left(\mp x\right)$ and $x\la qx^2; q^2\ra \in \mT_\C\left(-x^2\right) \cap \mS_\C\left(x^2\right)$. These are injective linear maps of vector spaces of equal finite dimensions by \cref{eq:dimension} (e.g., for the first map both spaces have dimension $\max\left(\lfloor(d+1)/2\rfloor - \one_{s = 1}, 0\right)$). Thus the three maps above are bijections, and determining the zeros of $\la x; q\ra$, $\la -x; q\ra$ and $\la qx^2; q^2 \ra$ completes our proof.
%Finally, the common zeros mentioned in part (ii) can be verified by a short computation for any sum of the form $\vect{sx^d}{k} + t\vect{sx^d}{d-k}$ with $k \in \Z$, which is enough since these sums span $\mT_\C\left(sx^d\right) \cap \mS_\C\left(tx^d\right)$; we leave these computations to the reader.
\end{proof}

\begin{corollary} \label{cor:bijections}
For $c, d \in \Z$ with $d \ge 1$, the map
\begin{center}
\begin{tikzcd} [column sep = huge]
    \mT_\C\big({-q^c x^d}\big) \cap \mS_\C\big(x^{d-2c}\big)
    \arrow{r}{x^{-1}\la qx^2; q^2 \ra^{-1}}
    &
    \mT_\C\big({q^c x^{d-2}}\big) \cap \mS_\C\big(x^{d-2-2c}\big)
\end{tikzcd}
\end{center}
which divides by $x\la qx^2; q^2 \ra$, is a bijection of $\lfloor d/2 \rfloor$-dimensional vector spaces. If $d$ is odd,
\begin{center}
\small
\begin{tikzcd} [column sep = huge]
    \mT_\C\big({-q^c x^d}\big) \cap \mS_\C\big(x^{d-2c}\big)
    \arrow{r}{x^{-1}\la qx^2; q^2 \ra^{-1}}
    \arrow{rd}[swap]{\la -x; q \ra^{-1}}
    &
    \mT_\C\big({q^c x^{d-2}}\big) \cap \mS_\C\big(x^{d-2c-2}\big)
    \arrow{r}{\la -x; q \ra^{-1}}
    &
    \mT_\C\big({q^c x^{d-3}}\big) \cap \mS_\C\big(x^{d-2c-3}\big)
    \\
    & 
    \mT_\C\big({-q^c x^{d-1}}\big) \cap \mS_\C\big(x^{d-2c-1}\big)
    \arrow{ru}[swap]{x^{-1}\la qx^2; q^2 \ra^{-1}}
\end{tikzcd}
\end{center}
give more such bijections by multiplicative factors.
\end{corollary}

\iffalse
\begin{lemma}[Existence of functions in $\mT_D(f)$] \label{lem:existence}
Let $f(z)$ be a polynomial in $x = e^{2\pi i z}$.
\begin{enumerate}[label=(\roman*).]
    \item If $\hat{f}(0) = 1$, then $\dim \mT_\C(f) > 0$.
    \item If $\hat{f}(0) = 0$ and $f$ has no zeros in $\{\Im z < - \Im \tau\}$, then $\dim \mT_\Hminus(f) > 0$.
\end{enumerate}

\end{lemma}
\begin{proof}
As anticipated in $\cref{eq:two-ways}$, consider the expressions
\[
    \prod_{n \ge 0} f(z + n\tau), \qquad \text{ respectively } \qquad \sum_{n > 0} \prod_{j=0}^{n-1} f(z + j\tau) + 1 + \sum_{n < 0} \prod_{j=-n}^{-1} f(z + j\tau)^{-1}.
\]
Under the assumptions in (i), respectively (ii), these expressions converge absolutely and locally uniformly in $\C$, respectively $\H^-$, and satisfy the symmetries of $\mT_\C(f)$, respectively $\mT_\Hminus(f)$. They are also nonzero since their zeroth Fourier coefficients are $1$, which proves (i) and (ii).
\end{proof}
\fi

\begin{proposition}[Spaces given by polynomials] \label{prop:polynomial-spaces}
Suppose $f(z)$ is a Laurent polynomial in $x = e^{2\pi i z}$ with $\deg_x f = d \ge 1$, such that $f(z)$ has no zeros in $\{\Im z < -\Im \tau\}$. Then one has
\[
    \dim \mT_\Hminus\left(f \right) = d \qquad \text{ and } \qquad 
    \dim \mT_\C\left(f \right) \le d,
\]
and if $d \ge 2$, equality is reached in the right side if and only if $f$ is a monomial in $x$. Moreover, any nonzero $F \in \mT_\Hminus(f)$ has exactly $d$ zeros (counting multiplicity) in any fundamental region $\{w - t - u\tau : t, u \in [0, 1) \}$, for $w \in \H^-$.
\end{proposition}

\begin{proof}
Write $f(z) = \alpha x^d \cdot g(-z)$, where $\alpha = \hat{f}(d)$ and $g(z)$ is a polynomial in $x$ with free coefficient $\hat{g}(0) = 1$. Then consider the function
\[
    G(z) := \prod_{n \ge 1} g(-z + n\tau) = \prod_{n \ge 1} \left(1 + \hat{g}(1) q^nx^{-1} + \hat{g}(2) q^{2n}x^{-2} + \cdots \right),
\]
where the product converges locally uniformly in $\C$ (given that $\sum_{n \ge 1} q^n$ converges); note that $G$ is designed to satisfy $G(z)g(-z) = G(z+\tau)$, hence $G \in \mT_\C\left(g(-z)^{-1}\right)$. Moreover, since $f(z)$ is nonzero on $\{\Im z < - \Im \tau\}$, each $g(-z+n\tau)$ is nonzero on $\{\Im z < (n-1)\Im \tau\}$, so $G$ is nonzero on $\{\Im z < 0\} = \H^-$ by standard properties of infinite products, and thus $G^{-1} \in \mT_\Hminus\left(g(-z)\right)$. Applying \cref{lem:basic}.(iii) twice, we find that multiplication by $G(z)$ induces a bijective linear map
\begin{equation}\label{eq:bijection}
     \mT_\Hminus (f)
     \xlongrightarrow{\cdot\ G(z)}
     \mT_\Hminus\big(\alpha x^d \big)
     =
     \mT_\C\big(\alpha x^d \big),
\end{equation}
proving $\dim \mT_\Hminus(f) = \dim \mT_\C\left(\alpha x^d\right) = d$ by \cref{prop:monomial-spaces}. By \cref{lem:basic}.(i), this immediately implies that $\dim \mT_\C(f) \le d$. Now say $f(z)$ is not a monomial in $x$, so it must have a zero $z_0 \in \C$, which results in a zero $w_0$ of $G(z)$. For $k \in \Z$, consider the product
\[
    F_k(z) := x^k \big\langle {-\alpha} q^k x^d; q^d \big\rangle \in \mT_\C\big(q^{-k} \cdot \alpha q^k x^d\big) = \mT_\C\big(\alpha x^d\big).
\]
Then $F_0(z)$ has zeros precisely when $-\alpha x^d = q^{nd}$ for some $n \in \Z$, i.e.\ when $z = u + (dn\tau + 2m\pi i)/d$ for some $m \in \Z$ and a fixed $u \in \C$. Hence $F_k(z) = x^k F_0(z + k\tau/d) = 0$ iff $z = u + ((dn-k)+ 2m\pi i)/d$ for some $m \in \Z$, which implies that $F_j$ and $F_k$ have no common zeros when $j \not\equiv k \pmod{d}$. So assuming $d \ge 2$, there is some $F_j$ with $F_j(w_0) \neq 0$, which implies that
\[
    F_j \cdot G^{-1} \in \mT_\Hminus(f) \setminus \mT_\C(f)
    \qquad 
    \Rightarrow 
    \qquad 
    \dim \mT_\C(f) < d.
\]
This settles the equality case of $\dim \mT_\C(f) \le d$ (together with \cref{prop:monomial-spaces}.(i)). The statement about zeros in fundamental regions of $\H^-$ follows from the analogous statement in \cref{prop:monomial-spaces}.(i) and the bijection in \cref{eq:bijection} (since $G$ has no zeros in $\H^-$).
\end{proof}

\begin{remark}
Combining \cref{cor:bijections} with the inclusion $\mT_\C\left(qx^2\right) \subset \mS_\C(1)$ from \cref{prop:monomial-spaces}, and the map in \cref{eq:bijection}, we obtain bijections by multiplicative factors
\begin{center}
\begin{tikzcd} [column sep = huge, row sep = tiny]
    \mT_\C\big({-q}x^5\big) \cap \mS_\C\big(x^3\big)
    \arrow{rr}{x^{-1}\left([-qx^3;0] + [-qx^3;1]\right)^{-1}}
    &&
    \mT_\C\big(qx^2\big)
    \arrow{rr}{(q; q)\left(x^{-1}; q\right)^{-1}}
    \arrow{rrd}[swap,xshift=0.3cm]{(q; q)\left(x^{-2}; q^2\right)^{-1}}
    &&
    \mT_\Hminus\big(qx^2 - x\big)
    \\
    \mT_\C\big({-q}x^4\big) \cap \mS_\C\big(x^2\big)
    \arrow{rru}[swap]{x^{-1}[-qx^2;0]^{-1}}
    && && 
    \mT_\Hminus\big(qx^2 - q^{-1}\big)
\end{tikzcd}
\end{center}
where we used that $\left(q^2; q^2\right)\la qx^2; q^2 \ra = \vect{-qx^2}{0}$ and $(q; q)\la -x; q\ra\la qx^2; q^2\ra = \vect{-qx^3}{0} + \vect{-qx^3}{1}$ by \cref{eq:triple,eq:quintuple} with adequate substitutions. What is special about the two resulting bijections (following the top and the bottom maps above) is that they preserve canonical bases: $\left\{\vect{-qx^5}{1} + \vect{-qx^5}{2}, \vect{-qx^5}{0} + \vect{-qx^5}{3}\right\}$ go to $\left\{\vect{qx^2 - x}{0}, \vect{qx^2-x}{1}\right\}$ by the top chain of maps, while $\left\{\vect{-qx^4}{1}, \vect{-qx^4}{0} + \vect{-qx^4}{2}\right\}$ go to $\left\{\vect{qx^2 - q^{-1}}{0}, \vect{qx^2-q^{-1}}{1}\right\}$ by the bottom one. This is the content of our main results in \cref{thm:bases-proportional,thm:bases-proportional-2}.
\begin{figure}[ht]
\centering
\includegraphics[scale = 0.6]{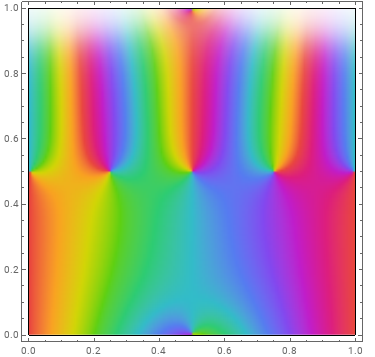} \qquad\qquad
\includegraphics[scale = 0.6]{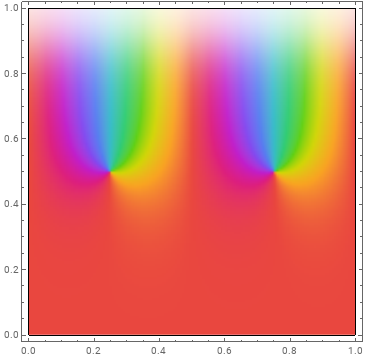}
\caption{Plots of $\vect{-qx^5}{1} + \vect{-qx^5}{2}$ (left) and $(q; q)\left(x^{-1}; q\right)\vect{qx^2 - x}{0}$ (right) on $\{\Re z, \Im z \in [0, 1)\}$, for $\tau = i$ (\emph{using domain coloring; plot made in Wolfram Mathematica}).}
\label{fig:plots}
\end{figure}
\FloatBarrier
As predicted by \cref{prop:monomial-spaces}.(i), $\vect{-qx^5}{1} + \vect{-qx^5}{2} \in \mT_\C\left(-qx^5\right) \cap \mS_\C\left(x^3\right)$ has five zeros in the fundamental region $\{t + u\tau : t,u \in [0, 1)\}$. Three of these zeros, at $z \in \left\{1/2, \tau/2, (1+\tau)/2\right\}$, correspond to the zeros of $\vect{-qx^3}{0} + \vect{-qx^3}{1}$ and are eliminated by the bijection in \cref{cor:bijections}; the remaining two nontrivial zeros coincide with the zeros of $\left(x^{-1}; q\right)\vect{qx^2 - x}{0} \in \mT_\C\left(qx^2\right)$, as seen in \cref{fig:plots}. For $\tau = i$, the nontrivial zeros lie approximately $0.24657 + 0.5i$ and $0.75343 + 0.5i$; note that they sum up to $1 + i$ since $\left(x^{-1}; q\right)\vect{qx^2 - x}{0}$ has a symmetry by a nonzero multiplicative factor at $z \mapsto 1 + \tau - z$ (by combining $1$-periodicity with the symmetries from $\mT_\C\left(qx^2\right) \subset \mS_\C(1)$).
%$0.2465672774964 + 0.5i$ and $0.7534327225036 + 0.5i$
\end{remark}

\subsection{Canonical basis vectors} 
\label{subsec:canonical-basis-vectors}
We are now ready to define the promised generalization of the functions $\vect{\alpha x^d}{k}$, $\vect{qx^2-x}{k}$ and $\vect{qx^2-q^{-1}}{k}$ from \cref{sec:intro} (and to show in particular that the latter two functions are well-defined).

\begin{proposition}[Canonical basis]\label{prop:canonical-basis-vectors}
Let $f(z)$ be a polynomial in $x = e^{2\pi i z}$ with $\deg_x f = d \ge 1$, such that $f(z)$ has no zeros in $\{\Im z < -\Im \tau\}$, and $\hat{f}(0) \not\in \left\{q^{-n} : n \ge d \right\}$. Then $\mT_\Hminus(f)$ has a canonical basis $\{\vect{f}{k}\}_{0 \le k < d}$ uniquely determined by
\[
    \hat{\vect{f}{k}}(j) = \one_{j=k} \quad \text{ in } \H^-, \qquad\quad  \forall j, k \in \{0, 1, \ldots, d-1\}.
\] 
\end{proposition}
\begin{proof}
\cref{prop:polynomial-spaces} implies that $\dim \mT_\Hminus(f) = d$, so we aim to apply the `only if' part of the equality case in \cref{prop:upper-bounds}.(ii). We have $P_0(z) = 1$ and $P_1(z) = f(z)$, so for all $n \in \Z$,
\[
    A_n(z) = P_0(z) - q^n P_1(z) 
    =
    1 - q^n \left(\hat{f}(0) + \hat{f}(1)x + \cdots + \hat{f}(d)x^d \right).
\]
Our assumptions imply that $\hat{A_n}(0) = 1 - q^n \hat{f}(0) \neq 0$ for all $n \ge d$, and $\hat{A_n}(d) = q^n \hat{f}(d) \neq 0$ for all $n \in \Z$. Hence using $n_0 = 0$, the conditions of \cref{prop:upper-bounds}.(ii) are fulfilled.
\end{proof}

\begin{remark}
In the proof of \cref{prop:polynomial-spaces}, dividing the canonical basis of $\mT_\C\left(\alpha x^d\right)$ by $G(z)$ yields a basis of $\mT_\Hminus(f)$, but this usually differs from the canonical basis $\{\vect{f}{k}\}_{0 \le k < d}$ (when well-defined); in fact, computing the relationship between these two bases is one of the main difficulties in proving \cref{thm:bases-proportional,thm:bases-proportional-2}. The advantage of expressing product identities in terms of the canonical basis is that for any $F \in \mT_\Hminus(f)$, by identifying Fourier coefficients one can write
\begin{equation} \label{eq:fourier-coeffs-basis}
    F = \sum_{k=0}^{d-1} \hat{F}(k) \vect{f}{k},
\end{equation}
and the coefficients $\{\hat{F}(k)\}_k$ may be objects of interest (in \cref{prop:2var-rog-ram-var} for example, they are the two Rogers--Ramanujan series from \cref{eq:rog-ram}). In particular, if $f(z)$ is a second-degree polynomial in $x$ with no zeros in $\{\Im z < -\Im \tau\}$ and $\hat{f}(0) \not\in \left\{q^{-n} : n \ge 2\right\}$, \cref{prop:polynomial-spaces,prop:canonical-basis-vectors} imply that $\dim \mT_\C(f) \le 1$ and $\dim \mT_\Hminus(f) = 2$, with a basis $\{\vect{f}{0}, \vect{f}{1}\}$. So if we are given an entire function $F \in \mT_\C(f)$, it must be unique up to a scalar, and it is natural to express it as in \cref{eq:fourier-coeffs-basis}; this is the subject of \cref{prop:2var-rog-ram} for $f(z) = qx^2 - x$.%, and of \cref{prop:2var-rog-ram-var} for $f(z) = qx^2 - q^{-1}$. (THIS IS FALSE!!)
\end{remark} 

\begin{example} \label{ex:plots}
For $f(z) = qx^2 - x$ and $k \in \{0, 1\}$, $\vect{qx^2-x}{k}$ is a well-defined holomorphic function in $\H^-$ by \cref{prop:polynomial-spaces}. By \cref{lem:basic}.(v), it can be continued meromorphically to $\C$, with simple poles at $z = n\tau + m$ for $n, m \in \Z$ with $n \ge 0$; these poles are cancelled when multiplying by $\left(x^{-1}; q\right)$, as expected since we know $\left(x^{-1}; q\right)\vect{qx^2 - x}{k} \in \mT_\C\left(qx^2\right)$ is entire. Using \cref{eq:triple}, \cref{eq:quintuple}, and our main result in \cref{thm:bases-proportional}, one can easily compute the residues of these functions at the periodic poles $z \in \Z$; up to a factor of $2\pi i$, these are
\[
    \lim_{x \to 1} (x-1)\vect{qx^2 - x}{k}
    =
    \begin{cases}
    \la q; q^5 \ra^{-1}, &k = 0, \\
    \la q^2; q^5 \ra^{-1}, &k = 1,
    \end{cases}
\]
which is consistent with the pole cancellation in \cref{prop:2var-rog-ram}. One can also verify this numerically as an identity of $q$-series: expand $F(z) := \vect{qx^2-x}{k} = \sum_{n \in \Z} \hat{F}(n) x^n$ in $\H^-$ and write
\[
    \lim_{x \to 1} (x-1)\vect{qx^2 - x}{k} =
    \lim_{n \to -\infty} \hat{F}(n),
\]
since $(x-1)\vect{qx^2 - x}{k}$ has the Fourier series $\sum_{n \in \Z} \big(\hat{F}(n-1) - \hat{F}(n)\big) x^n$ in $\left\{\Im z < \Im \tau \right\} \supset \Z$, and $\hat{F}(n) \to 0$ as $n \to \infty$ (note that $\sum_{n \in \Z} \hat{F}(n) x^n$ converges absolutely for $z \in \H^-$, i.e.\ $|x| > 1$). 
\begin{table}[ht]
    \centering \tiny
    \captionsetup{width=\linewidth}
    \begin{tabular}{r|rrrrrrrrrrrrrrrrrrrrrrrrrrrrrrrrrrrrrrrrrrrrrrrrrrrrrrrrrrrrrrrrrrrrrrrrrrrrrrrrrrrrrr}
    \textbf{11} &    &    &    &    &    &    &    &    &    &    &    &    &    &    &    &    &    &    &    &    &    &    &    &    \\
    \textbf{10} &    &    &    &    &    &    &    &    &    &    &    &    &    &    &    &    &    &    &    &    &    &    &    &    \\
    \textbf{9} &    &    &    &    &    &    &    &    &    &    &    &    &    &    &    &    &    &    &    &    &    & -1 & -1 & -1 \\
    \textbf{8} &    &    &    &    &    &    &    &    &    &    &    &    &    &    &    &    &  1 &    &  1 &  1 &  2 &  1 &  2 &  1 \\
    \textbf{7} &    &    &    &    &    &    &    &    &    &    &    &    &    & -1 & -1 & -1 & -1 & -1 & -1 & -1 &    & -1 &    &    \\
    \textbf{6} &    &    &    &    &    &    &    &    &    &  1 &    &  1 &  1 &  1 &    &  1 &    &    &    &    &    &    &    &    \\
    \textbf{5} &    &    &    &    &    &    &    & -1 & -1 &    & -1 &    &    &    &    &    &    &    &    &    &    &    &    &    \\
    \textbf{4} &    &    &    &    &  1 &    &  1 &    &    &    &    &    &    &    &    &    &    &    &    &    &    &    &    &    \\
    \textbf{3} &    &    &    & -1 &    &    &    &    &    &    &    &    &    &    &    &    &    &    &    &    &    &    &    &    \\
    \textbf{2} &    &  1 &    &    &    &    &    &    &    &    &    &    &    &    &    &    &    &    &    &    &    &    &    &    \\
    \textbf{1} &    &    &    &    &    &    &    &    &    &    &    &    &    &    &    &    &    &    &    &    &    &    &    &    \\
    \textbf{0} &  1 &    &    &    &    &    &    &    &    &    &    &    &    &    &    &    &    &    &    &    &    &    &    &    \\
    \textbf{-1} &  1 &    &    &    &    &    &    &    &    &    &    &    &    &    &    &    &    &    &    &    &    &    &    &    \\
    \textbf{-2} &  1 &  1 &    &    &    &    &    &    &    &    &    &    &    &    &    &    &    &    &    &    &    &    &    &    \\
    \textbf{-3} &  1 &  1 &  1 &    &    &    &    &    &    &    &    &    &    &    &    &    &    &    &    &    &    &    &    &    \\
    \textbf{-4} &  1 &  1 &  1 &  1 &  1 &    &    &    &    &    &    &    &    &    &    &    &    &    &    &    &    &    &    &    \\
    \textbf{-5} &  1 &  1 &  1 &  1 &  2 &  1 &  1 &    &    &    &    &    &    &    &    &    &    &    &    &    &    &    &    &    \\
    \textbf{-6} &  1 &  1 &  1 &  1 &  2 &  2 &  2 &  1 &  1 &  1 &    &    &    &    &    &    &    &    &    &    &    &    &    &    \\
    \textbf{-7} &  1 &  1 &  1 &  1 &  2 &  2 &  3 &  2 &  2 &  2 &  2 &  1 &  1 &    &    &    &    &    &    &    &    &    &    &    \\
    \textbf{-8} &  1 &  1 &  1 &  1 &  2 &  2 &  3 &  3 &  3 &  3 &  3 &  3 &  3 &  2 &  1 &  1 &  1 &    &    &    &    &    &    &    \\
    \textbf{-9} &  1 &  1 &  1 &  1 &  2 &  2 &  3 &  3 &  4 &  4 &  4 &  4 &  5 &  4 &  4 &  3 &  3 &  2 &  2 &  1 &  1 &    &    &    \\
    \textbf{-10} &  1 &  1 &  1 &  1 &  2 &  2 &  3 &  3 &  4 &  5 &  5 &  5 &  6 &  6 &  6 &  6 &  6 &  5 &  5 &  4 &  4 &  3 &  2 &  1 \\
    \textbf{-11} &  1 &  1 &  1 &  1 &  2 &  2 &  3 &  3 &  4 &  5 &  6 &  6 &  7 &  7 &  8 &  8 &  9 &  8 &  9 &  8 &  8 &  7 &  7 &  5 \\
    \textbf{-12} &  1 &  1 &  1 &  1 &  2 &  2 &  3 &  3 &  4 &  5 &  6 &  7 &  8 &  8 &  9 & 10 & 11 & 11 & 12 & 12 & 13 & 12 & 12 & 11 \\
    \textbf{-13} &  1 &  1 &  1 &  1 &  2 &  2 &  3 &  3 &  4 &  5 &  6 &  7 &  9 &  9 & 10 & 11 & 13 & 13 & 15 & 15 & 17 & 17 & 18 & 17 \\
    \textbf{-14} &  1 &  1 &  1 &  1 &  2 &  2 &  3 &  3 &  4 &  5 &  6 &  7 &  9 & 10 & 11 & 12 & 14 & 15 & 17 & 18 & 20 & 21 & 23 & 23 \\
    \textbf{-15} &  1 &  1 &  1 &  1 &  2 &  2 &  3 &  3 &  4 &  5 &  6 &  7 &  9 & 10 & 12 & 13 & 15 & 16 & 19 & 20 & 23 & 24 & 27 & 28 \\
    \textbf{-16} &  1 &  1 &  1 &  1 &  2 &  2 &  3 &  3 &  4 &  5 &  6 &  7 &  9 & 10 & 12 & 14 & 16 & 17 & 20 & 22 & 25 & 27 & 30 & 32 \\
    \textbf{-17} &  1 &  1 &  1 &  1 &  2 &  2 &  3 &  3 &  4 &  5 &  6 &  7 &  9 & 10 & 12 & 14 & 17 & 18 & 21 & 23 & 27 & 29 & 33 & 35 \\
    \textbf{-18} &  1 &  1 &  1 &  1 &  2 &  2 &  3 &  3 &  4 &  5 &  6 &  7 &  9 & 10 & 12 & 14 & 17 & 19 & 22 & 24 & 28 & 31 & 35 & 38 \\
    \textbf{-19} &  1 &  1 &  1 &  1 &  2 &  2 &  3 &  3 &  4 &  5 &  6 &  7 &  9 & 10 & 12 & 14 & 17 & 19 & 23 & 25 & 29 & 32 & 37 & 40 \\
    \textbf{-20} &  1 &  1 &  1 &  1 &  2 &  2 &  3 &  3 &  4 &  5 &  6 &  7 &  9 & 10 & 12 & 14 & 17 & 19 & 23 & 26 & 30 & 33 & 38 & 42 \\
    \hline
        &\ \textbf{0} &\ \textbf{1} &\ \textbf{2} &\ \textbf{3} &\ \textbf{4} &\ \textbf{5} &\ \textbf{6} &\ \textbf{7} &\ \textbf{8} &\ \textbf{9} &\textbf{10} &\textbf{11} &\textbf{12} &\textbf{13} &\textbf{14} &\textbf{15} &\textbf{16} &\textbf{17} &\textbf{18} &\textbf{19} &\textbf{20} &\textbf{21} &\textbf{22} &\textbf{23} 
    \end{tabular}
    \caption{2D Fourier coefficients of $\vect{qx^2-x}{0}$ in $\H^-$ (coeff.\ of $x^m q^n$ shown on line $m$, column $n$).}
    \label{tbl:plot1}
\end{table}
Indeed, we have $\la q; q^5 \ra^{-1} = 1 + q + q^2 + q^3 + 2q^4 + 2q^5 + 3q^6 + 3q^7 + 4q^8 + 5q^9 + 6q^{10} + \cdots$, which is approached by the coefficients of $x^{-n}$ for large $n$ in \cref{tbl:plot1} (these are related to the coefficients $b_{n+1}(q)$ from \cite{garrett1999variants}). Expansions of $\left(x^{-1}; q\right)\vect{qx^2 - x}{k}$ in the canonical basis of $\mT_\C\left(qx^2\right)$, leading to the proof of \cref{thm:bases-proportional}, are given in \cref{thm:twisted-rog-ram}; there is an analogous story for $\left(x^{-2}; q^2\right)\vect{qx^2 - q^{-1}}{k}$ in \cref{thm:twisted-rog-ram-var}, which is equivalent to \cref{thm:bases-proportional-2}. Another example concerns $f(z) = qx^2 - 2x + q^{-1}$, where the basis vector $\vect{qx^2-2x+q^{-1}}{1}$ turns out to be entire; the relevant identity here is
\[
    x(qx; q)^2 = \vect{qx^2-2x+q^{-1}}{1},
\]
which follows since both sides lie in $\mT_\Hminus\left(qx^2 - 2x + q^{-1}\right)$ and satisfy $\hat{F}(0) = 0$ and $\hat{F}(1) = 1$. By contrast, $\vect{qx^2 - 2x + q^{-1}}{0}$ can be shown to have double poles at $z = n\tau + m$ for $n, m \in \Z$, $n \ge 0$, which are cancelled in the multiplication by $G(z) = \left(x^{-1}; q\right)^2$.
\end{example}

With this motivation, we seek exact formulae for $\vect{f}{k}$, covering at least the canonical basis vectors in \cref{thm:bases-proportional} and \cref{thm:bases-proportional-2}. The next two propositions provide such results for two kinds of polynomials $f$; both generalize the original formula $\vect{\alpha x^d}{k} = \sum_{n \in \Z} \alpha^n q^{d\binom{n}{2} + kn} x^{dn+k}$.

\begin{proposition}[Formulae for canonical vectors, first kind] \label{prop:exact-formulae-kind1}
Suppose $f(z) = \alpha x^d + \beta$ is a polynomial in $x$ of degree $d \ge 1$ and free coefficient $\beta \in \C \setminus \left\{ q^{-n} : n \ge d \right\}$, such that $f(z) \neq 0$ in $\left\{\Im z < -\Im \tau\right\}$. Then for $0 \le k < d$, one has
\begin{equation} \label{eq:kind1}
    \vect{\alpha x^d + \beta}{k}
    \ =\
    \sum_{n \in \Z} \alpha^n \frac{q^{d\binom{n}{2}+kn}}{\left(\beta q^{d+k}; q^d\right)_n} x^{dn+k}
    \ =\
    \frac{\vect{\alpha x^d}{k}}{\left(\beta q^{d+k}; q^d \right) \left(-\alpha^{-1}\beta x^{-d} q^d; q^d \right)}.
\end{equation}
\end{proposition}
\begin{proof}
Note that $\left(\beta q^{d+k}; q^d\right)_n$ has no zeros for $n \ge 0$ by our assumption on $\beta$, and in fact we have $\left(\beta q^{d+k}; q^d\right)_n^{-1} = 1 + O(q) = O(1)$ for $n \ge 0$; hence the coefficients of $x^{dn+k}$ in the middle series above decrease like $q^{dn^2/2}$. For $n < 0$, we defined $\left(\beta q^{d+k}; q^d\right)_n = \left(\beta q^{d+k-dn}; q^d\right)_{-n}^{-1}$, and thus 
\[
\begin{aligned}
    \frac{\alpha^n q^{d\binom{n}{2}+kn}}{\big(\beta q^{d+k}; q^d \big)_n} =
    \alpha^n q^{d\binom{n}{2}+kn}
    \prod_{j=1}^{|n|} \left(1 - \beta q^{d+k-dj}\right)
    &=
    \alpha^{-|n|} q^{d\binom{|n|+1}{2}-k|n|} \cdot
    O\left(|\beta|^{|n|} |q|^{\sum_{j=1}^{|n|} (d+k-dj)} \right)
    \\
    &=
    O\left(|q^d\beta/\alpha|^{|n|} \right).
\end{aligned}
\]
assuming without loss of generality that $\beta \neq 0$ (when the given series coincides with $\vect{\alpha x^d}{k}$). But since $\alpha x^d + \beta \neq 0$ in the range $\{\Im z < -\Im \tau\}$ (when $|x^d|$ spans $\left(|q|^{-d}, +\infty\right)$), we must have $|\beta/\alpha| \le |q|^{-d}$, i.e.\ $|q^d \beta / \alpha| \le 1$. Hence the bound above reduces to $O(1)$.

This shows that the middle series in \cref{eq:kind1} converges absolutely and locally uniformly to a holomorphic function $R(z)$ for $z \in \H^-$ (using that $|x^{-1}| < 1$ to bound $\sum_{n < 0} x^{dn+k}$, and the quadratic-exponent decrease of $q^{dn(n-1)/2} x^n$ for $n \ge 0$). We clearly have $\hat{R}(j) = \one_{j=k}$ for $0 \le j < d$, so to prove the first equality in \cref{eq:kind1} it remains to show that $R \in \mT_{\Hminus}(\alpha x^d + \beta)$. Indeed, we have
\[
\begin{aligned}
    \left(\alpha x^d + \beta\right) R(z+\tau)
    &=
    \sum_{n \in \Z}
    \alpha^{n+1}
    \frac{q^{d\binom{n+1}{2}+k(n+1)}}{\left(\beta q^{d+k}; q^d\right)_n} x^{d(n+1)+k}
    +
    \beta \sum_{n \in \Z} \alpha^n \frac{q^{d\binom{n+1}{2}+k(n+1)}}{\left(\beta q^{d+k}; q^d\right)_n} x^{dn+k}
    \\
    &=
    \sum_{n \in \Z} \alpha^n \frac{q^{d\binom{n}{2}+kn}}{\left(\beta q^{d+k}; q^d\right)_n} x^{dn+k} \left(1 - \beta q^{d+k+d(n-1)} + \beta q^{dn+k} \right)
    \quad =
    R(z).
\end{aligned}
\]
For the second equality in \cref{eq:kind1}, we require two simple identities that will be proven (non-circularly) in the next subsection: \cref{eq:basic1} and \cref{eq:basic2}. For $z \in \H^-$, consider the function
\[
    F(z) := \frac{\vect{\alpha x^d}{k}}{\left(-\alpha^{-1}\beta x^{-d} q^d; q^d\right)}
    \
    \stackrel{\text{by \cref{eq:basic2}}}{=\joinrel=\joinrel=}
    \
    \sum_{n \in \Z} \alpha^n q^{d\binom{n}{2} + kn} x^{dn+k} \cdot \sum_{n \ge 0} \frac{1}{\left(q^d; q^d\right)_n} \big({-\alpha^{-1}\beta} x^{-d} q^d \big)^n.
\]
We have $F \in \mT_\Hminus\left(\alpha x^d\left(1 + \alpha^{-1}\beta x^{-d}\right) \right) = \mT_\Hminus\left(\alpha x^d + \beta \right)$ by \cref{tbl:spaces}, and thus $F$ has a basis representation as in \cref{eq:fourier-coeffs-basis}. But by multiplying out the Fourier series above, we see that
\[
    \forall j \in \{0, 1, \ldots, d-1\} : \qquad\quad 
    \hat{F}(j) = 
    \begin{cases}
        0, &j \neq k, \\
        \sum_{n \ge 0} \frac{q^{d\binom{n}{2}}}{\left(q^d; q^d\right)_n} \left(-\beta q^{d+k}\right)^n,
        &j = k,
    \end{cases}
\]
and the latter sum reduces to $\left(\beta q^{d+k}; q^d \right)$ by \cref{eq:basic1}. We thus have $F(z) = \left(\beta q^{d+k}; q^d \right) \vect{\alpha x^d + \beta}{k}$, completing our proof.
\end{proof}

\begin{corollary} \label{cor:kind1}
For $|y| \le |q|^{-1}$, the canonical basis vectors of $\mT_\Hminus\left(qx^2 - y\right)$ are
\[
\begin{aligned}
    \vect{qx^2 - y}{0} 
    &= \sum_{n \in \Z} \frac{q^{n^2}}{\left(-q^2y; q^2\right)_n} x^{2n}
    \ \ \hspace{0.11cm} =
    \frac{\vect{qx^2}{0}}{\left(-q^2y; q^2\right)\left(qx^{-2}y; q^2\right)}, 
    \\
    \vect{qx^2 - y}{1} 
    &= \sum_{n \in \Z} \frac{q^{n^2+n}}{\left(-q^3y; q^2\right)_n} x^{2n+1}
    =
    \frac{\vect{qx^2}{1}}{\left(-q^3y; q^2\right)\left(qx^{-2}y; q^2\right)}.
\end{aligned}
\]
\end{corollary}

\begin{remark}
Since we know that $\vect{\alpha x^d + \beta}{k}$ exists based on more abstract arguments (recall the proof of \cref{prop:canonical-basis-vectors}), we could have identified its Fourier series in \cref{eq:kind1} directly from the recurrence relation in \cref{eq:recursion}. However, the first part of the proof of \cref{prop:exact-formulae-kind1} is instructive for the next proposition, where the analogous summation is usually not a Fourier series; also, the fact that the conditions of \cref{prop:canonical-basis-vectors} barely make the series in \cref{eq:kind1} converge shows that they are sharp.
\end{remark}
\begin{proposition}[Formulae for canonical vectors, second kind] \label{prop:exact-formulae-kind2}
Suppose $f(z)$ is a polynomial in $x$ of degree $d \ge 1$ with null free coefficient, and that $f(z)$ has no zeros in $\{\Im z < - \Im \tau\}$. Write $f(z) = \alpha x^d g(-z)$, where $\alpha = \hat{f}(d)$, and let $G(z) := \prod_{n \ge 1} g(-z+n\tau)$. Then for $0 \le k < d$,
\begin{equation} \label{eq:kind2}
    \vect{f}{k}
    =
    \sum_{n \in \Z} \alpha^n q^{d\binom{n}{2} + kn} x^{dn+k} \frac{G\left(z+(n-1)\tau\right)}{G(z)} g_k\left(-z-(n-1)\tau\right),
\end{equation}
where $g_k(z) := \sum_{j=0}^k \hat{g}(j) x^{j}$ is the truncation of $g$ to degree $\le k$. In particular, for $\deg_x g \le k < d$,
\[
    \vect{f}{0} =
    \sum_{n \in \Z} \alpha^n q^{d\binom{n}{2}} x^{dn} \frac{G(z+(n-1)\tau)}{G(z)}
    \qquad 
    \text{and} 
    \qquad 
    \vect{f}{k} =
    \sum_{n \in \Z} \alpha^n q^{d\binom{n}{2} + kn} x^{dn+k} \frac{G(z+n\tau)}{G(z)}.
\]
\end{proposition}
\begin{remark}
For $\deg_x g \le k < d$, if $R(z)$ denotes the right-hand side of \cref{eq:kind2}, then $x^{-k} R(z)$ is precisely the second construction in \cref{eq:two-ways} for the factor $q^k f(z)$.
\end{remark}
\begin{proof}
First, recall from the proof of \cref{prop:polynomial-spaces} that $G \in \mT_\C\left(g(-z)^{-1}\right)$ is well-defined and multiplication by $G$ induces a bijection $\mT_\Hminus(f) \to \mT_\Hminus(\alpha x^d)$ as in $\cref{eq:bijection}$.
We leave to the reader to verify that the series in the right-hand side of \cref{eq:kind2} converges absolutely and locally uniformly for $z \in \H^-$ to a function $R(z)$, which is similar to the proof of \cref{prop:exact-formulae-kind1}; the fact that $\hat{f}(0) = 0$ and thus $\deg_x g \le d-1$ is crucial here, to ensure that $\hat{R}(n)$ decreases like $q^{cn^2}$ as $n \to \infty$ for some $c > 0$. Now $R(z)G(z)$ is just a \emph{twisted} version of the series $\sum_{n \in \Z} \alpha^n q^{d\binom{n}{2}+kn} x^{dn+k}$ in the sense of \cref{eq:w-coeffs}; thus $R \cdot G \in \mT_\Hminus\left(\alpha x^d\right)$, and so $R \in \mT_\Hminus (f)$. 

It remains to identify the Fourier coefficients $\hat{R}(j)$ for $0 \le j < d$. Note that the Fourier expansions of $G(z + (n-1)\tau)/G(z)$ and $g_k(-z-(n-1)\tau)$ only contain negative powers of $x$, so only the terms $n \ge 0$ in $\cref{eq:kind2}$ contribute to the Fourier coefficients $\{\hat{R}(j)\}_{0 \le j < d}$. But for $n \ge 1$, one has
\[
    x^{dn+k}
    \frac{G\left(z+(n-1)\tau\right)}{G(z)} g_k\left(-z-(n-1)\tau\right)
    =
    x^{dn+k}
    \Bigg(\prod_{j=2-n}^0 g(-z+j\tau)\Bigg)
    \Bigg( \sum_{j=0}^k \hat{g}(j) \left(q^{n-1} x\right)^{-j}\Bigg),
\]
which only contains powers of $x$ with exponents $\ge dn+k - (d-1)(n-1) - k \ge d$. Hence the only relevant term is the one given by $n = 0$, which is
\[
    x^k \frac{G(z-\tau)}{G(z)} g_k(-z+\tau)
    =
    x^k \frac{g_k(-z+\tau)}{g(-z+\tau)}
    =
    x^k - x^k\frac{g(-z+\tau) - g_k(-z+\tau)}{g(-z+\tau)}.
\] 
But by the definition of $g_k$, the difference $g(-z+\tau) - g_k(-z+\tau)$ only contains powers of $x$ with exponents less than $-k$. We conclude that $\hat{R}(j) = \one_{j=k}$ for $0 \le j < d$, and thus $R = \vect{f}{k}$.
\end{proof}

\begin{corollary} \label{cor:kind2}
For $|y| \le 1$, the canonical basis vectors of $\mT_\Hminus\left(qx^2 - yx\right)$ are
\[
    \vect{qx^2-yx}{0} = \sum_{n \in \Z}
    q^{n^2} x^{2n} \frac{(q^{-n+1}x^{-1}y; q)}{(x^{-1}y; q)}
    , 
    \qquad\quad
    \vect{qx^2-yx}{1} = \sum_{n \in \Z}
    q^{n^2+n} x^{2n+1} \frac{(q^{-n}x^{-1}y; q)}{(x^{-1}y; q)}.
\]
%The case $y = 1$ is relevant for \cref{thm:bases-proportional}.
\end{corollary}

\subsection{How to identify two functions in the same space} \label{subsec:identification}

In proving identities of functions in a finite-dimensional space $\mT_D(f_1, \ldots, f_m)$, one natural approach is to identify enough Fourier coefficients. We start with a fact that was implicit in a previous proof, but which is worth emphasizing.

\begin{lemma}[Fourier identification] \label{lem:proofs-by-fourier}
Suppose $\mT_D(f_1, \ldots, f_m)$, $S \subset D$, $d \ge 1$ and $n_0 \in \Z$ are as in the hypothesis of \cref{prop:upper-bounds}.(ii). Then given $F, G \in \mT_D(f_1, \ldots, f_m)$, one has $F = G$ iff $\hat{F}(k) = \hat{G}(k)$ for all $k \in \{n_0, n_0 + 1, \ldots, n_0 + d-1\}$ (taking Fourier series in the strip $S$). If $d = 1$ and the limits below exist and are nonzero, one also has
\[
    F = G \qquad 
    \iff 
    \qquad \lim_{n \to \infty} \hat{F}(n) 
    =
    \lim_{n \to \infty} \hat{G}(n).
\]
\end{lemma}
\begin{proof}
In the proof of \cref{prop:upper-bounds}.(ii), we showed (using the linear relation of Fourier coefficients \cref{eq:recursion}) that $\mT_D(f_1, \ldots, f_m)$ injects into $\C^d$ by $F \mapsto \big(\hat{F}(n_0), \ldots, \hat{F}(n_0 + d-1)\big)$, proving the first claim. Note that this does \emph{not} require an equality in $\dim \mT_D(f_1, \ldots, f_m) \le d$.

For the second claim, having $d = 1$ implies that $\dim \mT_D(f_1, \ldots, f_m) \le 1$ by \cref{prop:upper-bounds}.(ii), which easily establishes the desired equivalence.
\end{proof}

To showcase the use of \cref{lem:proofs-by-fourier}, we further prove two results which imply Jacobi's triple product identity and which will be helpful in \cref{sec:higher-prod-id,sec:rog-ram}. The first is Ramanujan's famous $_1\psi_1$ summation, whose history, applications and further extensions are reviewed in \cite{warnaar2013ramanujan}. Our proof is essentially equivalent to that given by Adiga, Berndt, Bhargava and Watson in \cite[Entry 17]{adiga1985ramanujan}, but it serves as a good illustration of the generality of \cref{prop:upper-bounds} and \cref{lem:proofs-by-fourier}.

\begin{proposition}[Ramanujan's $_1\psi_1$ summation \cite{warnaar2013ramanujan}] \label{prop:ramanujan-1psi1}
Let $a, b \in \C$ such that $a, q/b \not\in \{q^n : n \ge 1\}$ and $a \neq 0$. Then for $|b/a| < |x| < 1$, one has
\begin{equation} \label{eq:ramanujan-1psi1}
    \sum_{n \in \Z} 
    \frac{(a; q)_n}{(b; q)_n} x^n 
    =
    \frac{(q; q)(b/a; q)}{(b; q)(q/a; q)} \frac{\la ax; q \ra}{(x; q)(b/ax; q)}.
\end{equation}
Taking $x \mapsto x/a$ and then $a \to \infty$ recovers Jacobi's triple product identity from \cref{eq:triple}.
\end{proposition}
\begin{proof}[Proof in our framework]
It suffices to prove the claim when $|b/a| < |q|$ and $a \not \in \{q^n : n \in \Z\}$, by the uniqueness of analytic continuation in $a$. Letting $D := \{z : |b/a| < |x| < 1\}$, we note that $D$ includes a horizontal strip of length $> \Im \tau$, and that both sides of \cref{eq:ramanujan-1psi1} yield holomorphic functions of $z$ inside $D$. Using \cref{tbl:spaces}, the right-hand side lies in
\[
    \mT_D\left(\frac{-ax}{(1-x)(1-b/aqx)^{-1}} \right)
    =
    \mT_D\left(\frac{(b/q) - ax}{1-x}\right),
\]
and the same is easily verified for the left-hand side. In the context of \cref{prop:upper-bounds}.(ii), we have $A_n(z) = 1 - x - q^n((b/q) + ax)$; hence $d = 1$, $\hat{A_n}(0) = 1 - q^{n-1}b \neq 0$ for any $n \ge 1$, and $\hat{A_n}(d) = q^n a - 1 \neq 0$ for any $n < 0$. Thus the conditions in \cref{prop:upper-bounds}.(ii) are fulfilled for $n_0 = 0$, and so by \cref{lem:proofs-by-fourier} it suffices to check that $\lim_{n \to \infty} \hat{F}(n) = \lim_{n \to \infty} \hat{G}(n)$, where $F$ and $G$ denote the left-hand side and the right-hand side of \cref{eq:ramanujan-1psi1}. But clearly
\[
    \lim_{n \to \infty} \hat{F}(n) = \frac{(a; q)}{(b; q)} \neq 0.
\]
Concerning the right-hand side, note that $(1-x)G(z)$ extends to a holomorphic function on $D' := \{z : |b/a| < |x| < |q|^{-1}\}$; hence the Fourier series $(1-x)G(z) = \sum_{n \in \Z} \left(\hat{G}(n) - \hat{G}(n-1)\right) x^n$ extends to (and converges absolutely and locally uniformly in) $D'$. Plugging in $x = 1$ yields
\[
    \lim_{x \to 1} (1-x)G(z) 
    =
    \sum_{n \in \Z} \left(\hat{G}(n) - \hat{G}(n-1)\right),
\]
where the sum converges absolutely. But the fact that $\sum_{n \in \Z} \hat{G}(n) x^n$ also converges absolutely for some $|x| < 1$ forces $\lim_{n \to -\infty} \hat{G}(n) = 0$, and thus
\[
    \lim_{n \to \infty} \hat{G}(n) = \lim_{x \to 1} (1-x)G(z)
    =
    \frac{(q; q)(b/a; q)}{(b; q)(q/a; q)} \frac{(a; q)(q/a; q)}{(q; q)(b/a; q)}
    =
    \frac{(a; q)}{(b; q)},
\]
which completes our proof.
\end{proof}

\begin{corollary}[Basic infinite product expansions] \label{cor:basic-identities}
For $z \in \C$, one has \cite[p.~19]{andrews1998theory}
\begin{equation} \label{eq:basic1}
(x; q) = \sum_{n \geq 0} \frac{q^{\binom{n}{2}}}{(q; q)_n} (-x)^n,
\end{equation}
More generally, for $z \in \H^+$ and $y \in \C$, one has the $q$-binomial theorem \cite[Theorem 2.4]{andrews1974applications}
\begin{equation} \label{eq:basic2}
    \frac{(xy; q)}{(x; q)} = \sum_{n \geq 0} \frac{(y; q)_n}{(q; q)_n} x^n
    \qquad\quad \xRightarrow{y \mapsto 0}
    \qquad\quad
    \frac{1}{(x; q)} = \sum_{n \geq 0} \frac{1}{(q; q)_n} x^n.
\end{equation} 
If additionally $|y| < 1$, then
\begin{equation} \label{eq:basic3}
    \frac{(xy; q)}{(x; q)(y; q)} 
    =
    \sum_{m, n \ge 0} \frac{q^{mn} x^m y^n}{(q; q)_m (q; q)_n}.
\end{equation}
\end{corollary}

\begin{proof}[Proof in our framework]
\cref{eq:basic1} and \cref{eq:basic2} follow as particular cases of \cref{prop:ramanujan-1psi1}, but can also be deduced by working in the one-dimensional spaces $\mT_\C(1 - x)$ and $\mT_\Hplus \left((1-xy)/(1-x)\right)$.

Finally, one easily checks that both sides of \cref{eq:basic3} lie in the space $\mT_\Hplus\left((1 - xy)/((1 - x)(1 - y))\right)$, and by \cref{lem:proofs-by-fourier} it suffices to identify the coefficients of $x^0$. These Fourier coefficients are $(y; q)^{-1}$ and respectively $\sum_{n \ge 0} \frac{1}{(q; q)_n} y^n$, which coincide due to the second identity in \cref{eq:basic2}.
\end{proof}

\cref{lem:proofs-by-fourier} also leads to a proof of Jacobi's triple product identity via a finitized version of the result, due to Cauchy \cite{cauchy1893second} (see also \cite[(6)]{bressoud1983easy}); we recount this identity below.

\begin{proposition}[Cauchy's finite triple product identity \cite{cauchy1893second}]
\label{prop:Cauchy} %http://www.bdim.eu/item?id=BUMI_2007_8_10B_3_867_0&fmt=pdf
For $z \in \C$ and any $N \ge 1$,
\[
(x; q)_N \left(q/x; q\right)_N = \sum_{n = -N}^N \frac{(q; q)_{2N}}{(q; q)_{N-n} (q; q)_{N+n}} (-1)^n q^{\binom{n}{2}} x^n.
\]
Taking $N \to \infty$ recovers Jacobi's triple product identity from \cref{eq:triple}.
\end{proposition}

\begin{proof}[Proof in our framework]
We leave to the reader to verify that both sides above lie in
\[
    \mT_\C\left(\frac{q^N - x}{1 - q^{N}x}\right).
\]
Going back to \cref{prop:upper-bounds}.(ii), we have $d = 1$ and $A_n(z) = 1 - q^{N}x - \left(q^N - x\right)q^n = \left(1 - q^{N+n}\right) - \left(q^{N} - q^{n} \right)x$ in this case. But for any choice of $n_0 \in [-N, N]$, we have $1 - q^{N+n} \neq 0$ for all $n \ge n_0 + 1$, and $q^{N} - q^{n} \neq 0$ for all $n < n_0$. So by \cref{lem:proofs-by-fourier}, it suffices to identify the coefficients of $x^N$ in the two sides of Cauchy's identity, which are $(-1) \cdot (-q) \cdots (-q)^{N-1} = (-1)^N q^{\binom{N}{2}}$.
\end{proof}

\begin{remark}
Taking $x \mapsto \alpha x^d$ and $q \mapsto q^d$ in Jacobi's triple product identity \cref{eq:triple} for some $\alpha \in \C^\times$ yields that
\[
    \big(q^d; q^d\big)\big\langle {-\alpha} x^d; q^d \big\rangle = 
    \sum_{n \in \Z} \alpha^n q^{d\binom{n}{2}} x^{dn} 
    =
    \vect{\alpha x^d}{0},
\]
and taking $z \mapsto z + k\tau/d$ for $k \in \Z$ gives the more general formula
\begin{equation}\label{eq:canonical-formula}
        \vect{\alpha x^d}{k} = \big(q^d; q^d\big)\big\langle {-\alpha} q^k x^d ; q^d \big\rangle x^k.
\end{equation}
Note that this is consistent with the vector spaces in \cref{tbl:spaces}. In particular, we know exactly where the canonical basis vectors $\vect{\alpha x^d}{k}$ have zeros (recall the proof of \cref{prop:polynomial-spaces}), although the zeros of a sum of two such functions can be nontrivial (recall \cref{ex:plots}). 
\end{remark}
Moving on to proofs by specialization or value identification, it is often the case that verifying a product identity at sufficiently many values of $z$ is enough to recover the general case. The following lemma gives such a criterion for the spaces $\mT_\C\left(\alpha x^d \right)$, previously characterized in \cref{prop:monomial-spaces}.
%\fix{This corresponds to the general method of specialization invoked by Macdonald \cite{Macdonald, Macdonald2}, but having studied the structure of ..., we can develop more precise tools based on specializing $x$ at $a^{th}$ roots of the unity.}
\begin{lemma}[Value identification]  \label{lem:proofs-by-value}
Let $d$ be a positive integer, $\alpha \in \C^\times$, and $F, G \in \mT_\C\left(\alpha x^d\right)$. If $\alpha = -q^n$ for some $n \in \Z$, assume additionally that $\hat{F}(d-n) = \hat{G}(d-n)$. Then one has $F = G$ iff $F(j/d) = G(j/d)$ for all integers $0 \le j < d$.
\end{lemma}

\begin{remark}
When $z = j/d$, $x = e^{2\pi i j/d}$ is a $d$th root of unity, which is helpful since the powers of $x$ occurring in a canonical basis vector $\vect{\alpha x^d}{k}$ are all multiples of $d$ plus $k$. Indeed, \cref{eq:canonical-formula} implies that
\begin{equation} \label{eq:canonical-root}
    \forall j, k \in \Z: \qquad\quad \vect{\alpha x^d}{k}\left(j/d\right) =
    \big( q^d; q^d \big) \big\langle {-\alpha} q^k; q^d \big\rangle e^{2\pi i (jk/d)}.
\end{equation}
Specializing product identities at roots of unity to prove them is, of course, not a new idea \cite{cao2012applications}, but it is helpful to phrase it in terms of the $d$-dimensional spaces $\mT_\C\left(\alpha x^d\right)$.
\end{remark}

\begin{proof}
Take $G = 0$ without loss of generality, and suppose $F \in \mT_\C\left(\alpha x^d\right)$ satisfies $F(j/d) = 0$ for all $0 \le j < d$. Writing $F = \sum_{k=0}^{d-1} \hat{F}(k) \vect{f}{k}$ as in \cref{eq:fourier-coeffs-basis}, the computation in \cref{eq:canonical-root} implies that
\[
    \begin{pmatrix}
    1 & 1 & \cdots & 1 \\
    1 & \zeta & \cdots & \zeta^{d-1} \\
    & & \ddots \\
    1 & \zeta^{d-1} & \cdots & \zeta^{(d-1)(d-1)}
    \end{pmatrix}
    \begin{pmatrix}
    \hat{F}(0) \big\langle {-\alpha}; q^d \big\rangle \\
    \hat{F}(1) \big\langle {-\alpha} q; q^d \big\rangle \\
    \vdots \\
    \hat{F}(d-1) \big\langle {-\alpha} q^{d-1}; q^d \big\rangle
    \end{pmatrix}
    =
    \begin{pmatrix}
    0 \\
    0 \\
    \vdots \\
    0
    \end{pmatrix}, \qquad \text{where } \zeta = e^{2\pi i /d}.
\]
But the matrix above is Vandermonde and invertible, and hence $\hat{F}(k) \la -\alpha q^k; q^d \ra = 0$ for all $0 \le k < d$. Now note that $\la -\alpha q^k; q^d \ra = 0$ only when $\alpha = -q^{d(m+1)-k}$ for some $m \in \Z$, and in this case the assumption of the lemma ensures that $\hat{F}(k-dm) = 0$. But $F \in \mT_\C\left(\alpha x^d\right)$ implies that $\hat{F}(n) = \alpha q^{n-d} \hat{F}(n-d)$ for all $n \in \Z$, and hence $\hat{F}(k) = 0$. We thus have $\hat{F}(k) = 0$ for all $0 \le k < d$, which forces $F = 0$ by its canonical basis representation (or by \cref{lem:proofs-by-fourier}).
\end{proof}

We can now give short proofs of the quintuple and septuple product identities from \cref{sec:intro} (essentially equivalent to, but more compact than those in \cite{cao2012applications}):
\begin{proof}[Proof of the quintuple identity, \cref{eq:quintuple}]
To show that $(q; q) \la x; q\ra \la qx^2; q^2\ra = \vect{qx^3}{0} - \vect{qx^3}{1}$, it is easier to take $x \mapsto -x$ and prove 
\[
    (q; q) \la -x; q\ra \la qx^2; q^2\ra \stackrel{?}{=} \left(q^3; q^3\right)\left(\la qx^3; q^3 \ra + x\la q^2 x^3; q^3 \ra \right)
\]
instead (where we used \cref{eq:canonical-formula} on the right-hand side). Both sides lie in $\mT_\C\left(-qx^3\right) \cap \mS_\C\left(x\right)$ by multiplying the respective factors in \cref{tbl:spaces}, which is one-dimensional by \cref{prop:monomial-spaces}, and spanned by $\vect{-qx^3}{0} + \vect{-qx^3}{1}$. Hence both sides have $\hat{F}(2) = 0$, and by \cref{lem:proofs-by-value} it suffices to check the equality when $x = \zeta$ is a cube root of unity. The latter reduces to
\[
    (q; q) \left( -\zeta; q\right) \left( -q\zeta^2; q\right) 
    \left( q\zeta^2; q^2\right) \left( q\zeta; q^2\right) 
    \stackrel{?}{=} \left(q^3; q^3\right)\left(\la q; q^3 \ra + \zeta \la q^2; q^3 \ra \right).
\]
Using $(-\zeta; q) = (1 + \zeta)(-q\zeta; q)$ and simplifying by $(q; q)(1+\zeta)$, this further reduces to
\[
    \left( -q\zeta; q\right) \left( -q\zeta^2; q\right) 
    \left( q\zeta; q^2\right)\left( q\zeta^2; q^2\right)
    \stackrel{?}{=} 1.
\]
Finally, using $(-x; q) = \left(x^2; q^2\right)/(x; q)$, the left-hand side becomes \[
    \frac{\left( q^2\zeta^2; q^2\right) \left( q^2\zeta; q^2\right) 
    \left( q\zeta; q^2\right)\left( q\zeta^2; q^2\right)}
    {\left( q\zeta; q\right) \left( q\zeta^2; q\right)}
    =
    \frac{\left( q\zeta^2; q\right) \left( q\zeta; q\right)}
    {\left( q\zeta; q\right) \left( q\zeta^2; q\right)} = 1,
\]
as we wanted.
\end{proof}

\iffalse
\begin{remark}
Taking $x \mapsto -x$ as in the proof above, the quintuple product identity can be rephrased as
\[
    \la qx^2; q^2 \ra = \frac{\vect{-qx^3}{0} + \vect{-qx^3}{1}}{\vect{x}{0}},
\]
which is strikingly similar to \cref{thm:bases-proportional,thm:bases-proportional-2}, and which induces a bijection by a multiplicative factor
\[
    \mT_\C\left(-qx^3\right) \cap \mS_\C(x) \xrightarrow{\cdot h(z)} 
    \mT_\Hminus (x),
\]
preserving the canonical bases of the two sides, as the ones in \cref{eq:thm-bijections}. 
\end{remark}
\fi

\begin{proof}[Proof of the septuple identity, \cref{eq:septuple}]
Note that both sides of $\cref{eq:septuple}$ lie in $\mT_\C\left(-qx^5\right) \cap \mS_\C\left(x^3\right)$ by \cref{tbl:spaces}, a space which is spanned by $\vect{-qx^5}{0} + \vect{-qx^5}{3}$ and $\vect{-qx^5}{1} + \vect{-qx^5}{2}$ by \cref{prop:monomial-spaces}. Hence both sides have $\hat{F}(4) = 0$, and by \cref{lem:proofs-by-value} it suffices to verify \cref{eq:septuple} when $x = \zeta$ is a fifth root of unity, which reduces by \cref{eq:canonical-root} to
\[
    (q; q) \la \zeta; q \ra \la \zeta^2; q^2 \ra \la q\zeta^2; q^2 \ra 
    \stackrel{?}{=}
    \frac{\left(q^5; q^5\right)\la q; q^5 \ra \left(1 + \zeta^3\right)}{\la q; q^5 \ra} - \frac{\left(q^5; q^5\right)\la q^2; q^5\ra \left(\zeta + \zeta^2\right)}{\la q^2; q^5 \ra}.
\]
But the right-hand side is just $\left(q^5; q^5\right)\left(1 - \zeta - \zeta^2 + \zeta^3\right)$, whereas grouping $\la \zeta^2; q^2 \ra \la q\zeta^2; q^2 \ra = \la \zeta^2; q \ra$, the left-hand side becomes
\[
    (q; q) (1 - \zeta) (q \zeta; q) \left(q\zeta^4; q\right) \left(1 - \zeta^2\right) \left(q \zeta^2; q\right) \left(q\zeta^3; q\right)
    =
    (1 - \zeta)\left(1 - \zeta^2\right) \left(q^5; q^5\right),
\]
by factoring $1 - q^{5n} = \prod_{k=0}^4 \left(1 - \zeta^k q^n\right)$. Expand $\left( 1- \zeta\right)\left(1 - \zeta^2\right) = 1 - \zeta - \zeta^2 + \zeta^3$ to finish.
\end{proof}

\begin{remark}
\cref{lem:proofs-by-value} does not lead to an easy proof of our two nonuple product identities (\cref{prop:nonuple} and \cref{prop:nonuple2}), which will require different methods based on \cref{lem:proofs-by-fourier}. Nevertheless, below is another application of \cref{lem:proofs-by-value}, which will be helpful in relation to \cref{thm:bases-proportional-2}.
\end{remark}

\begin{proposition}[Squared triple product identity]
\label{prop:squared-triple}
In $\mT_\C\left(x^2\right) \cap \mS_\C\left(x^2\right)$, one has
\[
    \frac{1}{\left(q^2; q^2\right)}\left((q; q)\la x; q \ra\right)^2 = \left(-q; q^2\right)^2 \vect{x^2}{0} - 2\left(-q^2; q^2\right)^2\vect{x^2}{1}.
\]

\end{proposition}
\begin{proof}
By \cref{lem:proofs-by-value}, it suffices to check the identity for $x \in \{\pm 1\}$, which is easy using \cref{eq:canonical-root}; we leave the details to the reader.
\end{proof}

%\fix{By similar means one can give a 3-variable analogue of this identity, for $(q; q)^2 ... (Macdonald)$ (viewed as a function of $z$ where $y$ is constant, so that our framework is still applicable; it's due to Macdonald.. check! But our later computations will only require this version}

%\fix{also all the identities that Watson uses; check!}

Finally, our third method of identification is well-suited for meromorphic functions that have poles in any strip of length $> \Im \tau$, and leads to identities for quotients of double products.

\begin{lemma}[Pole identification] \label{lem:proofs-by-pole}
Let $\alpha \in \C^\times$ and $d \in \Z$ with $d \le 0$; if $d = 0$, assume additionally that $\alpha \not\in \left\{q^n : n \in \Z\right\}$. Let $S$ be an open horizontal strip of length $> \Im \tau$, and suppose $F, G \in \mT_{S \setminus I}\left(\alpha x^d\right)$ where $I \subset S$ is a set of isolated points. Then one has $F = G$ iff $F - G \in \Hol(S)$. Hence if $F$ and $G$ only have simple poles, it suffices to check that their residues at those poles agree.
\end{lemma}

\begin{proof}
Note that $F - G \in \Hol(S)$ implies $F - G \in \mT_S\left(\alpha x^d\right)$, so it remains to show that $\mT_S\left(\alpha x^d\right) = \{0\}$. This was established in \cref{prop:monomial-spaces}.(i) (under the given condition if $d = 0$).
\end{proof}

\begin{proposition}[Quotients of double products] \label{prop:fractional}
For $x, q \in \C$ with $0 < |q| < |x| < 1$, one has (as noted, for instance, in \cite[para.\ 486]{tannery1898elements})
\begin{equation} \label{eq:inverse}
    \frac{(q; q)^2}{\la x; q \ra} 
    =
    \sum_{m, n \geq 0} \left( q^{(2m+1)n} - q^{(2n+1)(m+1)}\right) x^{m-n},
\end{equation}
Moreover, one has
\begin{equation} \label{eq:inverse2}
    \frac{(q; q)^4}{\la x; q\ra^2}
    =
    \sum_{m, n \geq 0} \left(m+n+1\right) q^{(m+1)n} x^{m-n}.
\end{equation}
If additionally $y \in \C$ such that $|q| < |y| < 1$, then as noted by Hickerson \cite[Theorem 1.5]{hickerson1988proof},
\begin{equation} \label{eq:fractional}
   (q; q)^2 \frac{ \la xy; q \ra}{\la x; q \ra \la y; q \ra} = 
    \sum_{m, n \ge 0} q^{mn} x^m y^n - \sum_{m, n \ge 1} q^{mn} x^{-m} y^{-n}.
\end{equation}
\end{proposition}

\begin{remark}
In most product identities, the nonzero 2D Fourier coefficients of entire and half-plane-entire functions tend to cluster in parabolas, such as in \cref{eq:triple,eq:quintuple,eq:septuple,eq:nonuple} or (partly) in \cref{tbl:plot1}. But the series in \cref{prop:fractional} have poles when $x = q^n$ ($n \in \Z$), so their nonzero coefficients cluster in a collection of lines rather than parabolas; each line corresponds to a series of the type $(1-q^nx)^{-1} = \sum_{m \ge 0} q^{mn}x^m$ or $(1-q^nx)^{-2} = \sum_{m \ge 0} (m+1) q^{mn}x^m$, indicating the order of the associated pole.
\end{remark}
\begin{table}[ht]
    \centering \tiny
    \captionsetup{width=\linewidth}
    \begin{tabular}{r|rrrrrrrrrrrrrrrrrrrrrrrrrrrrrrrrrrrrrrrrrrrrrrrrrrrrrrrrrrrrrrrrrrrrrrrrrrrrrrrrrrrrrr}
    \textbf{15} & 16 &    &    &    &    &    &    &    &    &    &    &    &    &    &    &    &    & 18 &    &    &    &    &    &    \\
    \textbf{14} & 15 &    &    &    &    &    &    &    &    &    &    &    &    &    &    &    & 17 &    &    &    &    &    &    &    \\
    \textbf{13} & 14 &    &    &    &    &    &    &    &    &    &    &    &    &    &    & 16 &    &    &    &    &    &    &    &    \\
    \textbf{12} & 13 &    &    &    &    &    &    &    &    &    &    &    &    &    & 15 &    &    &    &    &    &    &    &    &    \\
    \textbf{11} & 12 &    &    &    &    &    &    &    &    &    &    &    &    & 14 &    &    &    &    &    &    &    &    &    &    \\
    \textbf{10} & 11 &    &    &    &    &    &    &    &    &    &    &    & 13 &    &    &    &    &    &    &    &    &    &    &    \\
    \textbf{9} & 10 &    &    &    &    &    &    &    &    &    &    & 12 &    &    &    &    &    &    &    &    &    &    &    &    \\
    \textbf{8} &  9 &    &    &    &    &    &    &    &    &    & 11 &    &    &    &    &    &    &    &    &    &    &    & 13 &    \\
    \textbf{7} &  8 &    &    &    &    &    &    &    &    & 10 &    &    &    &    &    &    &    &    &    &    & 12 &    &    &    \\
    \textbf{6} &  7 &    &    &    &    &    &    &    &  9 &    &    &    &    &    &    &    &    &    & 11 &    &    &    &    &    \\
    \textbf{5} &  6 &    &    &    &    &    &    &  8 &    &    &    &    &    &    &    &    & 10 &    &    &    &    &    &    &    \\
    \textbf{4} &  5 &    &    &    &    &    &  7 &    &    &    &    &    &    &    &  9 &    &    &    &    &    &    &    &    &    \\
    \textbf{3} &  4 &    &    &    &    &  6 &    &    &    &    &    &    &  8 &    &    &    &    &    &    &    &    & 10 &    &    \\
    \textbf{2} &  3 &    &    &    &  5 &    &    &    &    &    &  7 &    &    &    &    &    &    &    &  9 &    &    &    &    &    \\
    \textbf{1} &  2 &    &    &  4 &    &    &    &    &  6 &    &    &    &    &    &    &  8 &    &    &    &    &    &    &    &    \\
    \textbf{0} &  1 &    &  3 &    &    &    &  5 &    &    &    &    &    &  7 &    &    &    &    &    &    &    &  9 &    &    &    \\
    \textbf{-1} &    &  2 &    &    &  4 &    &    &    &    &  6 &    &    &    &    &    &    &  8 &    &    &    &    &    &    &    \\
    \textbf{-2} &    &    &  3 &    &    &    &  5 &    &    &    &    &    &  7 &    &    &    &    &    &    &    &  9 &    &    &    \\
    \textbf{-3} &    &    &    &  4 &    &    &    &    &  6 &    &    &    &    &    &    &  8 &    &    &    &    &    &    &    &    \\
    \textbf{-4} &    &    &    &    &  5 &    &    &    &    &    &  7 &    &    &    &    &    &    &    &  9 &    &    &    &    &    \\
    \textbf{-5} &    &    &    &    &    &  6 &    &    &    &    &    &    &  8 &    &    &    &    &    &    &    &    & 10 &    &    \\
    \textbf{-6} &    &    &    &    &    &    &  7 &    &    &    &    &    &    &    &  9 &    &    &    &    &    &    &    &    &    \\
    \textbf{-7} &    &    &    &    &    &    &    &  8 &    &    &    &    &    &    &    &    & 10 &    &    &    &    &    &    &    \\
    \textbf{-8} &    &    &    &    &    &    &    &    &  9 &    &    &    &    &    &    &    &    &    & 11 &    &    &    &    &    \\
    \textbf{-9} &    &    &    &    &    &    &    &    &    & 10 &    &    &    &    &    &    &    &    &    &    & 12 &    &    &    \\
    \textbf{-10} &    &    &    &    &    &    &    &    &    &    & 11 &    &    &    &    &    &    &    &    &    &    &    & 13 &    \\
    \textbf{-11} &    &    &    &    &    &    &    &    &    &    &    & 12 &    &    &    &    &    &    &    &    &    &    &    &    \\
    \textbf{-12} &    &    &    &    &    &    &    &    &    &    &    &    & 13 &    &    &    &    &    &    &    &    &    &    &    \\
    \textbf{-13} &    &    &    &    &    &    &    &    &    &    &    &    &    & 14 &    &    &    &    &    &    &    &    &    &    \\
    \textbf{-14} &    &    &    &    &    &    &    &    &    &    &    &    &    &    & 15 &    &    &    &    &    &    &    &    &    \\
    \textbf{-15} &    &    &    &    &    &    &    &    &    &    &    &    &    &    &    & 16 &    &    &    &    &    &    &    &    \\
    \hline
        &\ \textbf{0} &\ \textbf{1} &\ \textbf{2} &\ \textbf{3} &\ \textbf{4} &\ \textbf{5} &\ \textbf{6} &\ \textbf{7} &\ \textbf{8} &\ \textbf{9} &\textbf{10} &\textbf{11} &\textbf{12} &\textbf{13} &\textbf{14} &\textbf{15} &\textbf{16} &\textbf{17} &\textbf{18} &\textbf{19} &\textbf{20} &\textbf{21} &\textbf{22} &\textbf{23} 
    \end{tabular}
    \caption{2D Fourier coefficients in \cref{eq:inverse2} for $|q| < |x| < 1$ (coeff.\ of $x^m q^n$ shown on line $m$, column $n$).}
    \label{tbl:plot2}
\end{table}
\begin{proof}[Proof of \cref{prop:fractional}]
We view the expressions in \cref{eq:inverse,eq:inverse2,eq:fractional} as functions of $z$ (as usual), and regard $q$ and $y$ as constants. Let $L(z)$ denote the left-hand side and $R(z)$ the right-hand side of \cref{eq:inverse}. One can check that both sides define holomorphic functions on $\{0 < \Im z < \Im \tau\}$, with boundaries imposed on $R(z)$ by the convergence of $\sum_{m \ge 0} x^m$ (coming from $n = 0$), respectively $\sum_{n \ge 0} q^n x^{-n}$ (coming from $m = 0$); in fact, $L(z)$ is a meromorphic in all of $\C$. Now rewrite
\[
    R(z) = \frac{1}{1 - x} + \sum_{n \ge 1} U_n(z) - \sum_{n \ge 0} V_n(z)
    \qquad\qquad 
    \text{for } 0 < \Im z < \Im \tau,
\]
where $U_n(z) = \sum_{m \ge 0} q^{(2m+1)n} x^{m-n}$ and $V_n(z) = \sum_{m \ge 0} q^{(2n+1)(m+1)} x^{m-n}$. This gives a meromorphic continuation of $R(z)$ to $S := \{-\Im \tau < \Im z < \Im \tau\}$, with poles only at $x = 1$ (i.e., at $z = 2\pi i n$, $n \in \Z$). Moreover, on $S \cap (S - \tau) = \{-\Im \tau < \Im z < 0\}$, this continuation satisfies
\[
\begin{aligned}
    -x^{-1} R(z + \tau)
    &=
    -\frac{x^{-1}}{1 - qx} - 
    \sum_{n \ge 1} \left(V_n(z) + x^{-n-1}\right)
    +
    \sum_{n \ge 0} U_{n+1}(z)
    \\
    &= R(z) - \frac{1}{1-x} -\frac{x^{-1}}{1 - qx} + V_0(z) + \frac{x^{-2}}{1 - x^{-1}} \\
    &= R(z) - \frac{1}{1-x} -\frac{x^{-1}}{1 - qx} + \frac{q}{1-qx} + \frac{x^{-1}}{x - 1}
    \quad =
    R(z) + x^{-1} - x^{-1} \ = R(z).
\end{aligned}
\]
Thus $R \in \mT_{S \setminus I}\left(-x^{-1}\right)$, where $I = \{x = 1\}$. But the same is true for $L(z) = (q; q)^2 \la x; q\ra^{-1}$, since $\la x; q\ra \in \mT_\C\left(-x\right)$ only has zeros on $\{x = q^n : n \in \Z\}$. Furthermore, near the poles one has

\[
    \lim_{x \to 1} (1 - x) L(z) = \lim_{x \to q} \frac{(q; q)^2 (1 - x)}{\la x; q\ra} 
    = \lim_{x \to 1} \frac{(q; q)^2}{(qx; q)\left(qx^{-1}; q\right)} = 1,
\]
\[
    \lim_{x \to 1} (1 - x) R(z) = \lim_{x \to 1} \frac{1-x}{1-x} = 1,
\]
so that $L - R$ is holomorphic on $S$; \cref{lem:proofs-by-pole} now completes the proof of \cref{eq:inverse}. The proof of $\cref{eq:fractional}$ is analogous, by writing the right-hand side as $\frac{1}{1-x} + \sum_{n \ge 1} \frac{y^n}{1-q^n x} - \sum_{n \ge 1} \frac{q^nx^{-1}y^{-n}}{1-q^n x^{-1}}$ on the same punctured strip $S \setminus I$; the relevant function space is now $\mT_{S\setminus I}(y)$, and the fact that $y \not\in \left\{q^n : n \in \Z\right\}$ (since $|q| < |y| < 1$) allows us to apply \cref{lem:proofs-by-pole} (again, by identifying residues at $x = 1$).

As for \cref{eq:inverse2}, one can proceed similarly using \cref{lem:proofs-by-pole} for the function space $\mT_{S \setminus I}\left(x^{-2}\right)$, except that the presence of double poles at $x = 1$ makes the identification harder. Alternatively, one can recover \cref{eq:inverse2} by taking the derivative in $y$ of \cref{eq:fractional} at $y = q/x$, and multiplying by $q/x$.
\end{proof}

\begin{remark}
Identity \cref{eq:inverse}, with the substitutions $x \mapsto x^2$, $q \mapsto q^2$, also follows from \cref{eq:fractional} by taking $y = -q/x$.  We note that it would be more difficult to prove such identities of meromorphic functions if we viewed them strictly as formal series. Indeed, there is no adequate ring of formal series containing the right-hand sides of \cref{eq:inverse,eq:inverse2,eq:fractional} both before and after substituting $x \mapsto xq$; attempting such proofs would lead to expressions of the form $\sum_{n \in \Z} x^n$, which correspond in the meromorphic setting to sums of the type $\frac{1}{1-x} + \frac{x^{-1}}{1-x^{-1}} = 0$.
\end{remark}

\iffalse
\begin{corollary}
One has 
\[
    \sum_{n \ge 0} \frac{q^n}{(q; q)_n^2} = 
    \frac{\left(q^4; q^4\right)\la -q; q^4\ra}{(q; q)^2}
    \qquad \text{ and }
    \qquad 
    \sum_{m, n \ge 0} \frac{q^{2mn+m+n}}{(q; q)_m^2 (q; q)_n^2} = \frac{1}{(q; q)^2}.
\]
%\fix{verify these numerically!}
\end{corollary}

\begin{proof}
Use \cref{cor:basic-identities} to rewrite the left-hand sides of \cref{eq:inverse} and \cref{eq:fractional} as 
\[
    (q; q)^2 \sum_{n \ge 0} \frac{x^n}{(q; q)_n} \sum_{n \ge 0} \frac{q^nx^{-n}}{(q; q)_n}, \qquad \text{resp.} \qquad (q; q)^2 \sum_{m, n \ge 0} \frac{q^{mn}x^m y^n}{(q; q)_m(q; q)_n}\sum_{m, n \ge 0} \frac{q^{mn+m+n}x^{-m} y^{-n}}{(q; q)_m(q; q)_n}.
\]
Expanding and identifying the coefficients of $x^0$, respectively $x^0 y^0$ in these expressions yields the left-hand sides in the corollary, up to a factor of $(q; q)^2$. The product form of the first identity's right-hand side follows from the triple product identity \cref{eq:triple} for $q \mapsto q^4$ and $x \mapsto -q$.
\end{proof}
\fi

\begin{remark}
From the fact that \cref{eq:inverse2} is the square of \cref{eq:inverse}, one can eventually deduce Besge's identity for the convolution $\sum_{k=1}^{n-1} \sigma(k) \sigma(n-k)$, where $\sigma$ is the sum-of-divisors function.
\end{remark}

\section{Higher-order product identities and generalized eta functions} \label{sec:higher-prod-id}

\subsection{M-coefficients and higher-order identities}\label{subsec:higher-order}
There are two main ingredients to proving the nonuple and undecuple product identities. The first one consists of the triple and quintuple identities proven in \cref{subsec:identification} (applied at two different instances), and the second one consists of the following lemma, which formalizes the multiplication identities anticipated in \cref{eq:m-coeffs}.

\begin{lemma}[Multiplication identities for $\mT_\C(\alpha x^d)$] \label{lem:m-identities}
Let $\alpha, \beta \in \C^\times$, and $a, b, u, v$ be positive integers. Then for $k, j \in \Z$ and any complete residue system $R$ of integers modulo $au^2+bv^2$, one has
\begin{equation} \label{eq:m-expansion}
    \vect{\alpha x^a}{k}(u z) \cdot \vect{\beta x^b}{j}(v z) = 
    \sum_{\ell \in R} M_\ell(q) \vect{\alpha^u \beta^v q^{a\binom{u}{2} + b \binom{v}{2}} x^{au^2+bv^2}}{\ell},
\end{equation}
where the coefficients $M_\ell(q) \in \C$ also depend on $a, b, \alpha, \beta, u, v$ but not on $z$, and are computed as follows. Let $d := \gcd(au, bv)$, $a' := au/d$, $b' := bv/d$ and $\ell' := (\ell - uk - vj)/d$. If $\ell' \not\in \Z$, then $M_\ell(q) = 0$. Else, letting $(n_0, m_0)$ be an integer solution to $a'n_0 + b'm_0 = 1$, one has
\begin{equation} \label{eq:m-coeffs-formula}
    M_\ell(q)
    =
    \gamma_\ell \sum_{n \in \Z} 
    \big(\alpha^{b'} \beta^{-a'}\big)^n 
    q^{\frac{ab'^2 + ba'^2}{2}n^2 + \delta_\ell n},
\end{equation}
\[
    \gamma_\ell = \big( \alpha^{n_0} \beta^{m_0} \big)^{\ell'} q^{a\binom{n_0 \ell' }{2} + b\binom{m_0 \ell'}{2} + (kn_0 + jm_0)\ell'},
    \qquad\quad
    \delta_\ell 
    =
    (ab'n_0 - ba'm_0)\ell' - \frac{ab' - ba'}{2} + kb' - ja'.
    %\left(a\ell' n_0 - \frac{a}{2} + k\right)b' - \left(b\ell'm_0 - \frac{b}{2} + j\right)a'.
\]
\end{lemma}

\begin{remark}
When $u = v = 1$ and $R$ is a complete residue system modulo $a+b$, \cref{lem:m-identities} states that
\[
    \vect{\alpha x^a}{k} \cdot \vect{\beta x^b}{j}
    =
    \sum_{\ell \in R} M_\ell(q) \vect{\alpha \beta x^{a+b}}{\ell},
\]
which is precisely the type of identity anticipated in \cref{eq:m-coeffs}; $M_\ell(q)$ is also a bit simpler in this case since $ab' = ba'$. If $\gcd(a, b) = 1$ then $M_\ell(q)$ is always nonzero, and we get $d = 1$, $a' = a$, $b' = b$, $\ell' = \ell - k - j$; so in particular, $\delta_\ell$ becomes $ab(n_0 - m_0)\ell' + kb - ja$. If additionally $b = 1$, then we can take $n_0 = 1$ and $m_0 = 1-a$, which simplifies $M_\ell(q)$ even further.

But in practice, one may avoid computing $\gamma_\ell$ and $\delta_\ell$ by the formulae above, since knowing that $M_\ell(q)$ takes the form in \cref{eq:m-coeffs-formula} allows for direct numerical identification (by computing the first few terms of the Fourier series in $\tau$ of the coefficient of $x^\ell$ in the original product).
\end{remark}

\begin{proof}
By \cref{lem:basic}.(iv), $S_u\left(\vect{\alpha x^a}{k}\right) \in \mT_\C\big(\alpha^u q^{a\binom{u}{2}} x^{au^2} \big)$ and $S_v\left(\vect{\beta x^b}{j}\right) \in \mT_\C\big(\beta^v q^{b\binom{v}{2}} x^{bv^2} \big)$. Hence their product lies in the product $\mT_\C$-space, which is spanned by the basis vectors on the right-hand side of \cref{eq:m-expansion} for any complete residue system $R$ modulo $au^2+bv^2$ (recall the proof of \cref{prop:polynomial-spaces}). It remains to compute the coefficients $M_\ell(q)$, i.e.\ the coefficients of $x^\ell$ in
\[
   \vect{\alpha x^a}{k}(u z) \cdot \vect{\beta x^b}{j}(v z)
   =
   \sum_{n \in \Z} \alpha^n q^{a\binom{n}{2} + kn} x^{aun + uk} 
    \cdot 
    \sum_{m \in \Z} \beta^m q^{b\binom{m}{2}+jm} x^{bvm+vj}.
\]
By multiplying these series and identifying coefficients of $x^\ell$, we find that
\[
    M_\ell(q)
    =
    \sum_{aun+bvm=\ell-uk-vj} 
    \alpha^n \beta^m q^{a\binom{n}{2} + kn + b\binom{m}{2} + jm}
    =
    \sum_{a'n + b'm = \ell'} 
    \alpha^n \beta^m q^{a\binom{n}{2} + kn + b\binom{m}{2} + jm}.
\]
Hence if $\ell' \not\in \Z$, there are no terms in the summation above and thus $M_\ell(q) = 0$. Otherwise, there are infinitely many terms, given precisely by
\[
    (n, m) = \ell'(n_0, m_0) + p(b', -a'), 
    \qquad\qquad p \in \Z.
\]
By reindexing and replacing $p$ with $n$, this yields
\[
    M_\ell(q) 
    =
    \sum_{n \in \Z} \alpha^{n_0\ell'  + b'n} \beta^{m_0\ell'  - a'n} 
    q^{a\binom{n_0\ell' + b'n}{2} + k(n_0\ell'  + b'n) + b \binom{m_0\ell' - a'n}{2} + j(m_0\ell' - a'n)},
\]
which simplifies to the expression in \cref{eq:m-coeffs-formula}.
\end{proof}

\begin{remark}
Given $\alpha, \beta \in \C^\times$ and positive integers $a, b, u, v$, we are interested in the products below:
\begin{align}
    \label{eq:type1}
    \text{Type 1:}\qquad
    &\la \alpha x^{u a}; q^a \ra \la \beta x^{v b}; q^b \ra,
    \\
    \label{eq:type2}
    \text{Type 2:}\qquad
    &\la \alpha x^{u a}; q^a \ra \la \alpha^2 x^{2u a}q^a; q^{2a} \ra \la \beta x^{v b}; q^b \ra ,
    \\
    \label{eq:type3}
    \text{Type 3:}\qquad
    &\la \alpha x^{u a}; q^a \ra \la \alpha^2 x^{2u a}q^a; q^{2a} \ra \la \beta x^{v b}; q^b \ra \la \beta^2 x^{2v b} q^b; q^{2b} \ra.
\end{align}
At the same time, adequate substitutions in \cref{eq:triple,eq:quintuple} yield
\begin{align} \label{eq:substitutions-triple}
    \left(q^a; q^a\right)\la \alpha x^{u a}; q^a \ra 
    &= 
    \vect{-\alpha x^a}{0}(uz),
    \\
    \label{eq:substitutions-quintuple}
    \left(q^a; q^a\right) \la \alpha x^{u a}; q^a \ra \la \alpha^2 x^{2u a}q^a; q^{2a} \ra 
    &=
    \vect{\alpha^3 q^a x^{3a}}{0}(uz) 
    - 
    \alpha \vect{\alpha^3 q^a x^{3a}}{a}(uz),
\end{align}
In this context, one strategy for finding higher-order product identities proceeds as follows:
\begin{itemize}
\item[\textbf{Step}]\textbf{1.}\ Given an infinite product as in \cref{eq:type1,eq:type2,eq:type3}, use \cref{eq:substitutions-triple,eq:substitutions-quintuple} to express it as a sum of products of the form in \cref{lem:m-identities}, up to scalar factors of $\left(q^a; q^a\right)\left(q^b; q^b\right)$.

\item[\textbf{Step}]\textbf{2.}\ Apply \cref{lem:m-identities} repeatedly and group terms to obtain a representation of the given product in the canonical basis of the appropriate $\mT_\C\left(\gamma x^d\right)$ space. Having an additional $\mS_\C$ structure (using \cref{tbl:spaces} in case $\alpha, \beta \in \{\pm 1\}$) reduces this work by half, since in a space of the form $\mT_\C\left(s q^c x^d\right) \cap \mS_\C\left(t x^{d-2c}\right)$, the canonical basis vectors pair up (recall \cref{prop:monomial-spaces}).

\item[\textbf{Step}]\textbf{3.}\ After Step 2, the coefficients of the canonical basis vectors $\vect{\gamma x^d}{\ell}$ in the resulting expansion are linear combinations of theta-type series in $q$, as in \cref{eq:m-coeffs-formula}.
Use \cref{eq:substitutions-triple,eq:substitutions-quintuple} again, specialized at adequate values $x \mapsto f(q)$, to rewrite each of these $q$-coefficients as infinite products (when applicable). This results in product identities like \cref{eq:septuple} and \cref{eq:nonuple}.
\end{itemize}
\cref{tbl:products} gathers some results of this approach. 
\end{remark}
\begin{table}[ht]
\renewcommand{\arraystretch}{1.3}
    \captionsetup{width=0.7\linewidth}
    \centering \small
    \begin{tabular}{c|c|c|c|c|c|c|c|c|c}
         Identity 
         & Sq.\ Triple
         & Sext.
         & Sept.
         & Sept.\ 2*
         & Oct.
         & Non.*
         & Non. 2*
         & Sq.\ Quint.*
         & Undec.*
         \\\hline
         Location
         & Pr.\ \ref{prop:squared-triple}
         & \cref{eq:sextuple}
         & \cref{eq:septuple}
         & Pr.\ \ref{prop:septuple2}
         & Pr.\ \ref{prop:octuple} 
         & Pr.\ \ref{prop:nonuple}
         & Pr.\ \ref{prop:nonuple2}
         & Pr.\ \ref{prop:squared-quintuple}
         & Pr.\ \ref{prop:undecuple}
         \\\hline
         Type(s)
         & 1
         & 1
         & 1 and 2
         & 2
         & 2
         & 2
         & 2
         & 3
         & 3
         \\\hline
         $\mT_\C$-Factor
         & $x^2$
         & $x^4$
         & $-qx^5$
         & $-qx^6$
         & $-qx^4$
         & $q^2x^7$
         & $-q^5x^{13}$
         & $q^2x^6$
         & $q^6x^{15}$
         \\\hline
         $\mS_\C$-Factor 
         & $x^2$
         & $x^4$
         & $x^3$
         & $x^4$
         & $x^2$
         & $x^3$
         & $x^3$
         & $x^2$
         & $x^3$
    \end{tabular}
    \caption{Identities for products of types 1, 2, or 3, following \cref{eq:type1,eq:type2,eq:type3} (`*' indicates original results, to the best of the author's knowledge).}
    \label{tbl:products}
\end{table}
While the names of these product identities (sextuple, septuple, etc.)\ are somewhat arbitrary (mainly because different authors discovered them), they can be usually justified by counting each factor of the form $(q^a; q^b)$ once, and each double product $\la q^b x^a; q^a \ra$ twice. \cref{fig:products} records how the infinite products in these identities build on each other progressively:
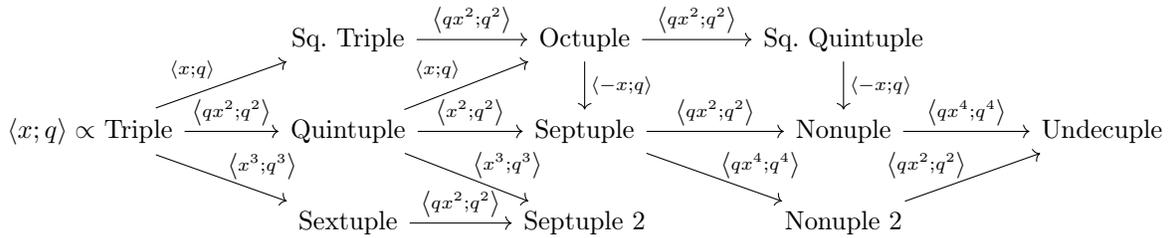
\begin{figure}[ht]
\captionsetup{width=0.8\linewidth}
\centering \small 
\begin{tikzcd} [column sep = large]
& \txt{Sq.\ Triple} \arrow{r}{\la qx^2; q^2 \ra}
& \txt{Octuple} \arrow{r}{\la qx^2; q^2 \ra} \arrow{d}{\la -x; q \ra}
& \txt{Sq.\ Quintuple} \arrow{d}{\la -x; q \ra}
\\
\txt{$\la x; q\ra \propto$ Triple} \arrow{r}{\la qx^2; q^2 \ra} \arrow{dr}[yshift=-0.1cm]{\la x^3; q^3 \ra} \arrow{ur}{\la x; q \ra}
& \txt{Quintuple} \arrow{r}{\la x^2; q^2 \ra} \arrow{dr}[yshift=-0.1cm]{\la x^3; q^3 \ra} \arrow{ur}{\la x; q \ra}
& \txt{Septuple} \arrow{r}{\la qx^2; q^2 \ra} \arrow{dr}[yshift=-0.1cm]{\la qx^4; q^4 \ra} 
& \txt{Nonuple} \arrow{r}{\la qx^4; q^4 \ra}  
& \txt{Undecuple} 
\\
& \text{Sextuple} \arrow{r}{\la qx^2; q^2 \ra}
& \text{Septuple 2}
& \text{Nonuple 2} \arrow{ur}[yshift=-0.1cm]{\la qx^2; q^2 \ra} 
\end{tikzcd}
\caption{Relations between infinite products; `$A \xrightarrow{F} B$' means that $B = A \cdot \alpha F$ for some entire function $F(z) = \la \pm q^b x^a; q^a \ra$ and some $\alpha \in \C^\times$ (which may depend on $q$).}
\label{fig:products}
\end{figure}

The strategy outlined in the previous remark would be cumbersome to formalize fully for type-2 and type-3 products (mainly due to variations in Step 3), but there is a relatively simple general structure for type-1 products. We note that identities for type-1 and closely related products are common in literature \cite{chu2007unification,adiga2014identities,yan2009several}, and often formulated in terms of Ramanujan's theta function.

\begin{proposition}[General type-1 identity] \label{prop:general-type-1}
Let $\alpha, \beta \in \C^\times$ and $a, b, u, v$ be positive integers. Define $d := \gcd(au, bv)$, $a' := au/d$ and $b' := bv/d$ as before, and let $c := ab'^2 + ba'^2$. Then there exist coefficients $\gamma_\ell \in \C$ and powers $p_\ell \in \Z$ such that for any complete residue system $R$ modulo $au^2 + bv^2$,
\[
    \la \alpha x^{u a}; q^a \ra \la \beta x^{v b}; q^b \ra
    =
    \sum_{\substack{\ell \in R \\ d \mid \ell}} 
    \gamma_\ell \frac{(q^c; q^c)\la -(-\alpha)^{b'}(-\beta)^{-a'}q^{p_\ell} ; q^c \ra}{\left(q^a; q^a\right)\left(q^b; q^b\right)}
    \vect{(-\alpha)^u (-\beta)^v q^{a\binom{u}{2} + b \binom{v}{2}} x^{au^2+bv^2}}{\ell}.
\]
It is often preferable to bring $p_\ell$ to the range $[0, c/2]$ using that $\la y; q^c \ra = -y\la q^c y; q^c \ra = \la q^c/y; q^c \ra$.
\end{proposition}
\begin{proof}
Use \cref{eq:substitutions-triple} and then \cref{lem:m-identities} (for $\alpha \mapsto -\alpha$, $\beta \mapsto -\beta$ and $k = j = 0$) to rewrite the left-hand side as
\[
    \frac{\vect{-\alpha x^a}{0}(uz) \cdot \vect{-\alpha x^b}{0}(vz)}{\left(q^a; q^a\right)\left(q^b; q^b\right)}
    =
    \sum_{d \mid \ell \in R} 
    \frac{M_\ell(q)}{\left(q^a; q^a\right)\left(q^b; q^b\right)}
    \vect{(-\alpha)^u (-\beta)^v q^{a\binom{u}{2} + b \binom{v}{2}} x^{au^2+bv^2}}{\ell}.
\]
It remains to identify the coefficients $M_\ell(q)$ as infinite products, which follows from their formula in \cref{eq:m-coeffs-formula}, and the triple product identity with adequate substitutions. The resulting choices of $\gamma_\ell$ are as in \cref{lem:m-identities} but for $\alpha \mapsto -\alpha$ and $\beta \mapsto -\beta$, while $p_\ell = \delta_\ell + (c/2)$ where $\delta_\ell$ are as in \cref{lem:m-identities}.
\end{proof}

\cref{prop:general-type-1} may seem cumbersome due to the large number of parameters; a particular case related to \cref{qtn:characters} is given below. Following \cite[(5.2)]{andrews1999a2}, denote $\chi_j^{(2,c)}(q) := (q; q)^{-1}\left(q^c; q^c\right) \la q^j; q^c \ra$ for all $j \in \Z$, and note that these products satisfy $\chi_{j}^{(2,c)}(q) = -q^j \chi_{j+c}^{(2,c)}(q) = \chi_{c-j}^{(2,c)}(q)$.

\begin{corollary}[$M(2, u^2+v^2)_2$ character identity]\label{cor:character-identity}
Let $u, v$ be positive integers with $\gcd(u, v) = 1$ and $u + v$ odd, and set $c := u^2 + v^2$. Let $(n_0, m_0)$ be an integer solution to $un_0 + vm_0 = 1$, and take $p := v n_0 - u m_0$ and $r := (u-v+c)/2$. Then there exist factors $\gamma_\ell$ of the form $(-1)^{a_\ell} q^{b_\ell}$ (depending on $u, v, n_0, m_0$) such that, as an identity of functions in $\mT_\C\big(-q^{\binom{u}{2} + \binom{v}{2}} x^c\big) \cap \mS_\C\left(x^{u+v}\right)$,
\begin{equation} \label{eq:character-id}
    (q; q)\la x^u; q \ra \la x^v; q \ra
    =
    \sum_{\ell = (u+v-c+2)/2}^{(u+v-1)/2} 
    \gamma_\ell\ \chi_{p\ell + r}^{(2, c)}(q)
    \left( \vect{-q^{\binom{u}{2} + \binom{v}{2}} x^{c}}{\ell} + \vect{-q^{\binom{u}{2} + \binom{v}{2}} x^{c}}{u+v-\ell} \right).
\end{equation}
Moreover, each character of $M(2, c)_2$ (given by $\chi_j^{(2,c)}(q)$ for $1 \le j \le (c-1)/2$) appears exactly once in \cref{eq:character-id}, after applying the symmetries $\chi_{j}^{(2,c)}(q) = -q^j \chi_{j+c}^{(2,c)}(q) = \chi_{c-j}^{(2,c)}(q)$ as needed.
\end{corollary}
\begin{proof}
Take $a = b = \alpha = \beta = 1$ and $R = \{(u+v-c+2)/2, \ldots, (u+v+c)/2\}$ in \cref{prop:general-type-1}, and compute
\[
    p_\ell \ =\ \delta_\ell + \frac{c}{2} 
    \ =\ (vn_0 - um_0)\ell - \frac{v-u}{2} + \frac{c}{2} 
    \ =\ p \ell + r,
\]
using the values $d = 1$ and $k = j = 0$ in \cref{lem:m-identities}. This produces equation \cref{eq:character-id}, except that $\ell$ varies in $R$ and the canonical basis vectors in the right-hand side are not paired. The pairing follows from the symmetry space $\mS_\C\left(x^{u+v}\right)$ of the left-hand side (and our adequate choice of $R$); the remaining value $\ell = (u+v+c)/2$ can be ignored since by \cref{prop:monomial-spaces}, the space $\mT_\C\big(-q^{\binom{u}{2} + \binom{v}{2}} x^c\big) \cap \mS_\C\left(x^{u+v}\right)$ is spanned by the sums $\vect{-q^{\binom{u}{2} + \binom{v}{2}} x^{c}}{\ell} + \vect{-q^{\binom{u}{2} + \binom{v}{2}} x^{c}}{u+v-\ell}$ for $(u+v-c+2)/2 \le \ell \le (u+v-1)/2$. Finally, it is easy to check that
\[
    \gcd(p, c) = 1 \qquad\text{and}\qquad p \ell + r \equiv 0 \pmod{c}
    \qquad \text{for} \qquad 
    \ell \equiv \frac{u+v+c}{2} \pmod{c},
\]
so that the index $p\ell+r$ covers each residue classes modulo $c$ \emph{except for $0$} exactly once when $R \setminus \{(u+v+c)/2\}$. Together with the aforementioned symmetries at $j \mapsto j+c$ and $j \mapsto c-j$ of the characters $\chi^{(2,c)}_{j}(q)$, this settles the second part of \cref{cor:character-identity}.
\end{proof}

\begin{example}
For $u = 1$ and $v = 2$ (thus $c = u^2 + v^2 = 5$), \cref{cor:character-identity} recovers the septuple product identity from \cref{eq:septuple}, which can be rephrased as
\[
\begin{aligned}
    (q; q)\la x; q\ra \la x^2; q\ra 
    =
    &-\chi_1^{(2,5)}(q) \left(\vect{-qx^5}{1} + \vect{-qx^5}{2}\right)
    \\
    &+\chi_2^{(2,5)}(q) \left(\vect{-qx^5}{0} + \vect{-qx^5}{3}\right).
\end{aligned}
\]
For $u = 2$ and $v = 3$, one obtains the analogous result
\[
\begin{aligned}
    (q; q)\la x^2; q\ra \la x^3; q \ra
    = &- q\chi_5^{(2,13)}(q) \left(\vect{-q^4 x^{13}}{-3} + \vect{-q^4 x^{13}}{8}\right) \\
    &- q\chi_3^{(2,13)}(q) \left(\vect{-q^4 x^{13}}{-2} + \vect{-q^4 x^{13}}{7}\right) \\
    &+ q\chi_2^{(2,13)}(q) \left(\vect{-q^4 x^{13}}{-1} + \vect{-q^4 x^{13}}{6}\right) \\
    &+ \hspace{0.16cm} \chi_6^{(2,13)}(q) \left(\vect{-q^4 x^{13}}{0} + \vect{-q^4 x^{13}}{5}\right) \\
    &+ q\chi_1^{(2,13)}(q) \left(\vect{-q^4 x^{13}}{1} + \vect{-q^4 x^{13}}{4}\right) \\
    &- \hspace{0.16cm} \chi_4^{(2,13)}(q) \left(\vect{-q^4 x^{13}}{2} + \vect{-q^4 x^{13}}{3}\right)
    ,
\end{aligned}
\]
which may be linked to the Andrews--Gordon identities \cite[Theorem 7.8]{andrews1998theory} as explained after \cref{qtn:characters}. On the other hand, the squared triple product identity from \cref{prop:squared-triple} follows by taking $a = b = u = v = 1$ in \cref{prop:general-type-1} (so that $au^2 + bv^2 = c = 2$), a case not covered by \cref{cor:character-identity}. Other such consequences of \cref{prop:general-type-1} are listed below.
\end{example}
\begin{corollary}[More type-1 identities]
As an identity of functions in $\mT_\C\left(x^4\right) \cap \mS_\C\left(x^2\right)$,
\begin{equation}\label{eq:sextuple}
\begin{aligned}
    \frac{(q; q)\left(q^3; q^3\right)}{\left(q^{12}; q^{12}\right)} 
    \la x; q\ra \la x^3; q^3\ra 
    &=
    \la-q^6; q^{12} \ra \vect{x^4}{0}
    +
    q \la-1; q^{12} \ra \vect{x^4}{2}
    \\[-5pt]
    &-
    \la-q^3; q^{12} \ra\left(\vect{x^4}{1} + \vect{x^4}{3} \right).
\end{aligned}
\end{equation}
(This is the sextuple product identity from \cite{zhu2011sextuple} in an equivalent form.) Moreover,
\begin{equation}\label{eq:order-10}
\begin{aligned}
    \frac{(q; q)^2}{\left(q^{10}; q^{10}\right)} \la x; q\ra \la x^3; q \ra 
    =
    &- \la -q^3; q^{10}\ra
    \left(\vect{q^3x^{10}}{1} + \vect{q^3x^{10}}{3}\right) \\[-5pt]
    &+ \la -q^4; q^{10}\ra
    \left(\vect{q^3x^{10}}{0} + \vect{q^3x^{10}}{4}\right)\\
    &- q\la-q; q^{10}\ra \hspace{0.15cm}
    \left(\vect{q^3x^{10}}{-1} + \vect{q^3x^{10}}{5}\right)
    \\
    &+ q\la-q^2; q^{10}\ra
    \left(\vect{q^3x^{10}}{-2} + \vect{q^3x^{10}}{6}\right)\\
    &- q\la-q^5; q^{10}\ra
    \vect{q^3x^{10}}{-3} 
    \ +\  q\la-1; q^{10}\ra\vect{q^3x^{10}}{2}
\end{aligned}
\end{equation}
inside $\mT_\C\left(q^3 x^{10}\right) \cap \mS_\C\left(x^4\right)$. Finally, in $\mT_\C\left(-q^6 x^{21}\right) \cap \mS_\C\left(x^9\right)$,
\begin{equation}\label{eq:order-21}
\begin{aligned}
    \left(q^3; q^3\right)\la x^3; q \ra \la x^6; q^3\ra
    &=
    \frac{\vect{-q^6 x^{21}}{0} + \vect{-q^6x^{21}}{9}}
    {\la q; q^7\ra \la q^2; q^7 \ra}
    \\
    &-
    \frac{\vect{-q^6 x^{21}}{3} + \vect{-q^6x^{21}}{6}}
    {\la q^2; q^7\ra \la q^3; q^7 \ra}
    -
    q
    \frac{\vect{-q^6 x^{21}}{-3} + \vect{-q^6x^{21}}{12}}
    {\la q^3; q^7\ra \la q; q^7 \ra}.
\end{aligned}
\end{equation}
\end{corollary}
\begin{proof}
For \cref{eq:sextuple}, use \cref{prop:general-type-1} for $\alpha = \beta = u = v = 1$, $a = 1$ and $b = 3$ (so $c = 1 \cdot 3^2 + 3 \cdot 1^2 = 12$ and $au^2 + bv^2 = 4$), then multiply both sides by $\left(q^a; q^a\right)\left(q^b; q^b\right) / \left(q^{c}; q^{c}\right)$. For \cref{eq:order-10}, repeat the same argument with the only difference that $v = 3$ and $b = 1$ instead (so $c = 1 \cdot 3^2 + 1 \cdot 1^2 = 10$ and $au^2 + bv^2 = 10$). Finally, for \cref{eq:order-21}, take $\alpha = \beta = 1$, $u = 3$, $v = 2$, $a = 1$ and $b = 3$ in \cref{prop:general-type-1}; this leads to $d = 3$, $a' = 1$ and $b' = 2$; thus $c = 1 \cdot 2^2 + 3 \cdot 1^2 = 7$ and $au^2 + bv^2 = 21$.
\end{proof}

\begin{remark}
While the freedom in constructing type-1 identities is large, we listed \cref{eq:order-10} and \cref{eq:order-21} since they will lead to identities of generalized eta functions in \cref{subsec:generalized-eta}; in fact, the choices of $a, b, u$ and $v$ were designed to produce the exponents of 10 and 21 in the right-hand sides.
\end{remark}
Moving on to type-2 identities, the first example is that of the septuple product identity once again. Indeed, the septuple product $\la x; q \ra \la qx^2; q^2 \ra \cdot \la x^2; q^2 \ra$ also assumes the form in \cref{eq:type2}; following the framework outlined before and using the quintuple identity repeatedly in Step 3 recovers \cref{eq:septuple}. A similar argument applies to the the following new variation of the septuple identity.
\begin{proposition}[Second septuple product identity] \label{prop:septuple2} 
In $\mT_\C\left(-qx^6\right) \cap \mS_\C\left(x^4\right)$, one has
\begin{equation} \label{eq:septuple2}
    \left(q^3; q^3\right)\la x; q \ra \la qx^2; q^2 \ra \la x^3; q^3 \ra 
    =
    \frac{\vect{-qx^6}{0} + \vect{-qx^6}{4}}{\left(q^3; q^6\right) \la q; q^6 \ra} - \frac{\vect{-qx^6}{1} + \vect{-qx^6}{3}}{\left(q^3; q^6\right) \la q^2; q^6 \ra}.
\end{equation}
\end{proposition}
\begin{proof}
Write $(q; q)\la x; q \ra \la qx^2; q^2 \ra = \vect{qx^3}{0} - \vect{qx^3}{1}$ and $\left(q^3; q^3\right)\la x^3; q^3 \ra = \vect{-x^3}{0}$ as in \cref{eq:substitutions-quintuple,eq:substitutions-triple}. Multiply these two equalities and apply \cref{lem:m-identities} twice, both times with $a = b = 3$ and $u = v = 1$ (thus $au^2 + bv^2 = 6$ and $ab'^2 + ba'^2 = 6$). Then apply \cref{eq:substitutions-triple} for $a = 6$ and $u = 1$ to express the resulting $q$-coefficients as infinite products, and divide by $(q; q)$ to recover \cref{prop:septuple2}.
\end{proof}

\begin{remark}
We noted in \cref{subsec:dim-bounds} that \cref{thm:bases-proportional,thm:bases-proportional-2} give proportionalities of canonical bases for $\mT_\C\left(-qx^d\right) \cap \mS_\C\left(x^{d-2}\right) \propto \mT_\Hminus(f_d)$ when $d \in \{4, 5\}$, corresponding to the octuple and septuple product identities; in this context, \cref{prop:septuple2} may correspond to the case $d = 6$ for some $f_6$. This possibility is supported by the framework in \cref{qtn:characters}, since the $q$-coefficients in the right-hand side of \cref{prop:septuple2} are $M(2, 6)_2$ characters according to the notation in \cite[(5.2)]{andrews1999a2}.
\end{remark}
\begin{proposition}[Ewell's octuple product identity \cite{ewell1982octuple}] \label{prop:octuple}
In $\mT_\C\left(-qx^4\right) \cap \mS_\C\left(x^2\right)$, one has
\begin{equation}\label{eq:octuple}
\begin{aligned}
    (q; q)\left(q; q^2\right) \la x; q \ra^2 \la qx^2; q^2 \ra &= 
    \hspace{0.19cm} \left(-q; q^2\right)\left(\vect{-qx^4}{0} + \vect{-qx^4}{2}\right) \\
    &- 2\left(-q^2; q^2\right) \vect{-qx^4}{1}.
\end{aligned}
\end{equation}
\end{proposition}
\begin{proof}
Expand $(q; q)\la x; q\ra = \vect{-x}{0}$ and $(q; q)\la x; q\ra \la qx^2; q^2\ra = \vect{qx^3}{0} - \vect{qx^3}{1}$, multiply the two equalities and apply \cref{lem:m-identities} twice for $a = 3$ and $b = u = v = 1$ (thus $au^2 + bv^2 = 4$ and $ab'^2 + ba'^2 = 12$). Then group terms, apply \cref{eq:substitutions-quintuple} with $a = 4$ and $u = 1$ for the coefficient of $\vect{-qx^4}{0} + \vect{-qx^4}{2}$, and \cref{eq:substitutions-triple} with $a = 12$ and $u = 1$ for the coefficient of $\vect{-qx^4}{1}$. To finish, divide everything by $\left(q^2; q^2\right)$ and use that $\la -q; q^4 \ra = \left(-q; q^2\right)$, $\left(q^4; q^4\right)/\left(q^2; q^2\right) = \left(-q^2; q^2\right)$.
\end{proof}
With this machinery we can also prove the (first) nonuple product identity introduced early in \cref{prop:nonuple}, and then follow it with the anticipated \emph{second} nonuple identity.
\begin{proof}[Proof of \cref{prop:nonuple}]
Recall that the (first) nonuple product in \cref{eq:nonuple} is proportional to
\[
    \la x; q \ra \la qx^2; q^2 \ra^2 \la x^2; q^2 \ra 
    =
    \la x; q \ra \la qx^2; q^2 \ra \cdot \la x^2; q \ra,
\]
which takes the type-2 form in \cref{eq:type2} for $a = b = u = 1$ and $v = 2$ (and $\alpha = \beta = 1$). After using \cref{eq:substitutions-triple,eq:substitutions-quintuple} and expanding, apply \cref{lem:m-identities} twice for $a = 3, b = 1$, $u = 1$ and $v = 2$, leading to the exponents $au^2 + bv^2 = 7$ in the canonical basis vectors and $ab'^2 + ba'^2 = 3 \cdot 2^2 + 1 \cdot 3^2 = 21$ in the $M_\ell$ coefficients. Then group terms, apply \cref{eq:substitutions-triple} with $a = 21$ and $u = 1$ for the coefficient of $\vect{q^2x^7}{5}$, and \cref{eq:substitutions-quintuple} with $a = 7$ and $u = 1$ for the other three $q$-coefficients. Scaling by the appropriate $q$-factor now recovers \cref{eq:nonuple}.
\end{proof}
\begin{proposition}[Second nonuple product identity] \label{prop:nonuple2}
In $\mT_\C\left(-q^5 x^{13}\right) \cap \mS_\C\left(x^3\right)$, one has
\begin{equation} \label{eq:nonuple2}
\begin{aligned}
    (q; q)\la x; q \ra \la qx^2; q^2 \ra \la x^2; q^2 \ra &\la qx^4; q^2 \ra
    \\
    =
    -&\frac{\vect{-q^5x^{13}}{1} + \vect{-q^5x^{13}}{2}}{\la q^3,q^4,q^5,q^6 ; q^{13}\ra \la q^2 ; q^{26} \ra}
    +
    \frac{\vect{-q^5x^{13}}{0} + \vect{-q^5x^{13}}{3}}{\la q,q^2,q^4,q^5 ; q^{13}\ra \la q^6 ; q^{26} \ra}
    \\[7pt]
    -\ q&\frac{\vect{-q^5x^{13}}{-1} + \vect{-q^5x^{13}}{4}}{\la q,q^2,q^4,q^6 ; q^{13}\ra \la q^{10} ; q^{26} \ra}
    -q^2
    \frac{\vect{-q^5x^{13}}{-2} + \vect{-q^5x^{13}}{5}}{\la q^2,q^3,q^4,q^5 ; q^{13}\ra \la q^{14} ; q^{26} \ra}
    \\[7pt]
    +\ q&\frac{\vect{-q^5x^{13}}{-3} + \vect{-q^5x^{13}}{6}}{\la q,q^2,q^3,q^4 ; q^{13}\ra \la q^{18} ; q^{26} \ra}
    -q
    \frac{\vect{-q^5x^{13}}{-4} + \vect{-q^5x^{13}}{7}}{\la q,q^3,q^5,q^6 ; q^{13}\ra \la q^{22} ; q^{26} \ra},
\end{aligned}
\end{equation}
where we used the common notation $\la q^{a_1}, q^{a_2}, \ldots, q^{a_k}; q^b \ra := \prod_{j=1}^k \la q^{a_j}; q^b \ra$.
\end{proposition}
\begin{proof}
The second nonuple product in \cref{eq:nonuple2} can be rewritten (up to a factor of $(q; q)$) as
\[
    \la x^2; q \ra \la qx^4; q^2 \ra \cdot \la x; q \ra,
\]
which takes the form of the type-2 product \cref{eq:type2} for $a = b = v = 1$ and $u = 2$ (and $\alpha = \beta = 1$). As before, use \cref{eq:substitutions-triple,eq:substitutions-quintuple}, then apply \cref{lem:m-identities} twice for $a = 3, b = 1, u = 2$ and $v = 1$ (so that $au^2 + bv^2 = 13$ and $ab'^2 + ba'^2 = 39$). Group terms and apply \cref{eq:substitutions-quintuple} with $a = 13$ and $u = 1$ for each of the resulting $q$-coefficients, and scale by the appropriate $q$-factor to conclude.
\end{proof}

\begin{remark}
What is special about the first and second nonuple product identities, among all identities in this section, is that they do not reduce to trivial equalities when specialized at $x = 1$. In fact, setting $x = 1$ here will lead to new identities of generalized eta functions in \cref{subsec:generalized-eta}.
\end{remark}

We are left with only two high-order product identities, both of type 3:
\begin{proposition}[Squared quintuple product identity] \label{prop:squared-quintuple}
In $\mT_\C\left(q^2x^6\right) \cap \mS_\C\left(x^2\right)$, one has
\[
\begin{aligned}
    \frac{(q; q)^2}{\left(q^6; q^6\right)} \la x; q \ra^2 \la qx^2; q^2 \ra^2 
    = 
    &\la-q^3; q^6 \ra \left(\vect{q^2x^6}{0} + \vect{q^2x^6}{2}\right)\hspace{0.12cm} - 2q\la-q; q^6 \ra\vect{q^2x^6}{4} \\
    +\ q&\la-1; q^6 \ra \left(\vect{q^2x^6}{-1} + \vect{q^2x^6}{3}\right) - 2\la-q^2; q^6 \ra \vect{q^2x^6}{1}.
\end{aligned}
\]
\end{proposition}
\begin{proof}
Square the quintuple product identity $(q; q)\la x; q\ra \la qx^2; q^2\ra = \vect{qx^3}{0} - \vect{qx^3}{1}$, and then proceed exactly as in the proof of \cref{prop:septuple2} (dividing by $(q^6; q^6)$ in the end).
\end{proof}

\begin{proposition}[Undecuple product identity] \label{prop:undecuple}
In $\mT_\C\left(q^6 x^{15} \right) \cap \mS_\C\left(x^3 \right)$, one has
\[
\begin{aligned}
    \frac{(q; q)^2}{\left(q^{15}; q^{15}\right)}\la x; q \ra \la x^2q; q^2 \ra^2 &\la x^2; q^2 \ra \la x^4q; q^2 \ra
    \\
    =
    &-\la -q^6; q^{15} \ra \left(\vect{q^6 x^{15}}{1} + \vect{q^6 x^{15}}{2}\right) \\
    &+
    \left(\la -q^7; q^{15} \ra + q\la -q^2; q^{15} \ra\right) \left(\vect{q^6 x^{15}}{0} + \vect{q^6 x^{15}}{3}\right) \\
    &- q^2 \la -1; q^{15} \ra\left(\vect{q^6 x^{15}}{-1} + \vect{q^6 x^{15}}{4}\right) \\
    &- q \la -q^3; q^{15} \ra \left(\vect{q^6 x^{15}}{-2} + \vect{q^6 x^{15}}{5}\right) \\
    &+ q\left(\la -q^4; q^{15}\ra + q\la -q; q^{15} \ra\right) \left(\vect{q^6 x^{15}}{-3} + \vect{q^6 x^{15}}{6}\right) \\
    &- q\la -q^6; q^{15} \ra \left(\vect{q^6 x^{15}}{-4} + \vect{q^6 x^{15}}{7}\right) \\
    &- q^2\la -q^3; q^{15} \ra \left(\vect{q^6 x^{15}}{-5} + \vect{q^6 x^{15}}{8}\right) + 2q^2\la -q^5; q^{15} \ra\vect{q^6 x^{15}}{9}.
\end{aligned}
\]
\end{proposition}
\begin{proof}
Rewrite the undecuple product (up to a factor depending on $q$) as
\[
    \la x; q \ra \la x^2q ; q^2 \ra 
    \cdot \la x^2; q \ra \la x^4 q; q^2 \ra,
\]
which has the form of the type-3 product \cref{eq:type3} for $a = b = u = 1$ and $v = 2$. Apply \cref{lem:m-identities} four times for $a = b = 3$, $u = 1$ and $v = 2$, so that $au^2 + bv^2 = 15$ and $ab'^2 + ba'^2 = 15$. Then group terms, apply \cref{eq:substitutions-triple} with $a = 15$ and $u = 1$ repeatedly, and divide by $\left(q^{15}; q^{15}\right)$ to conclude.
\end{proof}

\begin{question}\label{qtn:finitized}
In \cref{prop:Cauchy}, we reproved a finitized version of Jacobi's triple product identity due to Cauchy; we used that $\dim \mT_\C(f_N) \le 1$, where $f_N(z)$ was a degree-1 rational function in $x$ approaching $-x$ as $N \to \infty$. Are there similar finitized analogues of the higher-order product identities in this section, provable using the finite-dimensionality of the appropriate $\mT_\C(f_N)$ spaces?
\end{question}

\subsection{Identities for quotients of generalized eta functions} \label{subsec:generalized-eta}
In \cref{lem:proofs-by-value}, we saw how to prove product identities via specializations of $x$ at roots of unity. Here we reverse this process, and specialize product identities from \cref{subsec:higher-order} to deduce identities of generalized eta functions. 

\begin{notation}[Eta quotients] \label{not:eta} 
Recall the notations from \cref{eq:generalized-eta} for the Dedekind eta function $\eta(\tau) := q^{1/24} (q; q)$, and the generalized eta functions $\E_g(\tau) := q^{N B(g/N)/2}\la q^g; q^N \ra$ of a fixed level $N \ge 1$ (for $\tau \in \H^+$). As mentioned after \cref{cor:eta-poly}, these are not to be confused with the Eisenstein series $E_{2k}(\tau)$, which do not appear in this paper. An \emph{eta quotient} \cite{gordon1993multiplicative,martin1997eta} is an expression
\[
    \prod_{1 \le d \mid n} \eta(d\tau)^{r_d},
\]
for some $n \ge 1$ and $r_d \in \Z$; we define \emph{generalized eta quotients} analogously, allowing only functions $\E_g$ of the same level $N$ (however, $g \in \Z$ may vary).
\end{notation}

\begin{remark}
$\eta$ is a half-integral weight modular form for the full modular group $\Gamma$, and eta quotients are useful in computing bases of modular forms for congruence subgroups containing $\Gamma_0(n)$. Similarly, the generalized eta functions $\E_g$ satisfy transformation formulae at the action of $\Gamma_0(N)$ \cite[Corollary 2]{yang2004transformation}; one can design generalized eta quotients that are invariant under the action of $\Gamma_1(N)$, and use them to produce generators of function fields associated to general genus-zero congruence subgroups \cite{yang2004transformation}. The functions $\E_g$ (and a further generalization $\E_{g,h}$ thereof) were also studied in \cite{berndt1973generalized}.

For our purposes, the generalized eta functions provide the ``right'' normalization of the double infinite products $\la q^g; q^N \ra$, in the sense that when specializing a product identity at a value $x = \pm q^r$, all extraneous powers of $q$ will be encapsulated in the functions $\eta$ and $\E_g$; this \emph{must} be the case because (non-constant) powers of $q$ do not transform nicely under $\tau \mapsto -1/\tau$. The following lemma provides a few easy facts about $\eta$ and $\E_g$, with proofs left to the reader.
\end{remark}
\begin{lemma}\label{lem:basic-E}
Let $\tau \in \H^+$ and fix a level $N$.
\begin{itemize}
    \item[(i).] If $N$ is odd, $\prod_{g=1}^{\lfloor N/2 \rfloor} \E_g(\tau) = \frac{\eta(\tau)}{\eta(N\tau)}$.
    \item[(ii).] If $N$ is even, $\prod_{g=1}^{(N/2)-1} \E_g(\tau) = \frac{\eta(\tau)}{\eta(N\tau/2)}$.
    \item[(iii).] For any $g \in \Z$, $\E_g(\tau) = \E_{N-g}(\tau) = -\E_{g+N}(\tau)$. In particular, $\E_0(\tau) = 0$.
\end{itemize}
\end{lemma}
Our goal here is to compute more complicated sums of quotients of generalized eta functions in terms of the better-understood function $\eta$. For $N \in \{2, 3, 4, 6\}$, each function $\E_g(\tau)$ can be expressed directly as an eta quotient; for instance, when $N = 6$, quick computations show that
\[
    \E_3(\tau) = \frac{\eta(3\tau)^2}{\eta(6\tau)^2}
    ,\qquad\quad 
    \E_2(\tau) = \frac{\eta(2\tau)}{\eta(6\tau)}
    ,\qquad\quad 
    \E_1(\tau) = \frac{\eta(\tau)\eta(6\tau)}{\eta(2\tau)\eta(3\tau)}.
\]
One can obtain more related identities when $N = 6$ by specializing \cref{prop:septuple2} and \cref{prop:squared-quintuple} at $x \in \{\pm 1, \pm i\}$, but we omit these here (similarly, \cref{eq:sextuple} and \cref{prop:octuple} lead to identities for $N = 4$). For the level $N = 5$, we remark that $\E_1(\tau)^{-1}$ and $\E_2(\tau)^{-1}$ are essentially the Rogers--Ramanujan functions from \cref{eq:rog-ram}, and that the Rogers--Ramanujan continued fraction $\E_2/\E_1$ (recall \cref{ex:two-var-rog-ram}) is a modular function for the principal congruence subgroup $\Gamma(5)$.

Our identities concern the higher levels $N \in \{7, 10, 13\}$. We implicitly use \cref{lem:basic-E}, the fact that $\la -x; q \ra = \la x^2; q^2 \ra / \la x; q\ra$, and the specialization of canonical basis vectors from \cref{eq:canonical-root} in our proofs; we also leave easy computational details to the reader, with the assurance that all the results below were verified numerically.

\begin{proposition}[Level-7 generalized eta quotients] \label{prop:level-7}
For the level $N = 7$ and $\tau \in \H^+$,
\begin{align}
    \label{eq:level-7-1}
    &\sum_{g=1}^3 \frac{\E_g(2\tau)}{\E_{3g}(\tau)}
    = 
    2\frac{\eta(14\tau)^2}{\eta(7\tau)^2},
    &
    &\sum_{g=1}^3 \frac{\E_{3g}(\tau)}{\E_g(2\tau)}
    = 
    \frac{\eta(7\tau)^2}{\eta(14\tau)^2},
    \\
    \label{eq:level-7-2}
    &\sum_{g=1}^3 
    \frac{\E_g(\tau)^2 \E_{3g}(\tau)}{\E_{3g}(2\tau)}
    =
    2\frac{\eta(2\tau)^2}{\eta(7\tau)^2},
    &
    &\sum_{g=1}^3 
    \E_g(\tau)\E_{3g}(3\tau)
    =
    0.
\end{align}
%checked numerically :)
\end{proposition}
\begin{proof}
The first equality in \cref{eq:level-7-1} follows by specializing the nonuple product identity from \cref{prop:nonuple} at $x = 1$; the left-hand side of \cref{eq:nonuple} becomes null due to the factor of $\la x; q\ra$, and the last term in \cref{eq:nonuple} results (after suitable scaling) in the quotient $2\eta(14\tau)^2/\eta(7\tau)^2$ above.

The second equality in \cref{eq:level-7-1} is equivalent to an identity of Hickerson; indeed, by dividing the fact that $(4.6)$ equals $(4.8)$ in \cite{hickerson1988seventh} by $-q^3(q; q)\left(q^7; q^7\right)^2$ and simplifying, one obtains:
\[
\frac{\la q; q^7 \ra}{\la q^4; q^{14} \ra} - \frac{\la q^2; q^7 \ra}{\la q^6; q^{14} \ra} + \frac{\la q^3; q^7 \ra}{q\la q^2; q^{14} \ra} = \frac{\la q^7; q^{14} \ra}{q}.
\]
Finally, the two equalities in \cref{eq:level-7-2} follow by specializing our type-1 product identity \cref{eq:order-21} at $x = -q^{1/3}$ (and applying \cref{eq:substitutions-quintuple} with $a = 7$ for the specialization of each pair of canonical basis vectors), respectively at $x = 1$.
\end{proof}

\begin{proof}[Proof of \cref{cor:eta-poly}]
Letting $X_g(\tau) := \E_g(2\tau)/\E_{3g}(\tau)$, \cref{eq:level-7-1} provides identities for $X_1 + X_2 + X_3$ and $\frac{1}{X_1} + \frac{1}{X_2} + \frac{1}{X_3}$. Combining these with the fact that
\[
    X_1(\tau)X_2(\tau)X_3(\tau) 
    =
    -\frac{\prod_{g=1}^3 \E_g(2\tau)}{\prod_{g=1}^3 \E_g(\tau)}
    =
    -\frac{\eta(2\tau)\eta(7\tau)}{\eta(\tau)\eta(14\tau)},
\]
which follows from \cref{lem:basic-E}, we obtain the three Vieta identities for $(X - X_1)(X - X_2)(X - X_3)$ that constitute \cref{cor:eta-poly}.
\end{proof}

\begin{proposition}[Level-10 generalized eta quotients]
For the level $N = 10$ and $\tau \in \H^+$,
\begin{align*}
    \frac{\E_1(2\tau)\E_2(2\tau)}{\E_1(\tau)\E_2(\tau)}
    &=
    \frac{1}{2}\left(
    \frac{\eta(2\tau)^4}{\eta(\tau)^2\eta(10\tau)^2} - \frac{\eta(10\tau)^2}{\eta(5\tau)^2}
    - \frac{\eta(\tau)\eta(5\tau)}{\eta(10\tau)^2}
    \right),
    \\
    \frac{\E_3(2\tau)\E_4(2\tau)}{\E_3(\tau)\E_4(\tau)}
    &=
     \frac{1}{2}\left(
    \frac{\eta(2\tau)^4}{\eta(\tau)^2\eta(10\tau)^2} - \frac{\eta(10\tau)^2}{\eta(5\tau)^2}
    + \frac{\eta(\tau)\eta(5\tau)}{\eta(10\tau)^2}
    \right).
\end{align*}
%checked numerically :)
\end{proposition}

\begin{proof}
We obtain separate identities for the sum and the difference of the two quotients above. The equality
\[
    \frac{\E_1(2\tau)\E_2(2\tau)}{\E_1(\tau)\E_2(\tau)} + 
    \frac{\E_3(2\tau)\E_4(2\tau)}{\E_3(\tau)\E_4(\tau)}
    =
    \frac{\eta(2\tau)^4}{\eta(\tau)^2\eta(10\tau)^2} - \frac{\eta(10\tau)^2}{\eta(5\tau)^2}
\]
follows by specializing our type-1 identity \cref{eq:order-10} at $x = -1$. The second equality,
\[
    \frac{\E_3(2\tau)\E_4(2\tau)}{\E_3(\tau)\E_4(\tau)}
    -
    \frac{\E_1(2\tau)\E_2(2\tau)}{\E_1(\tau)\E_2(\tau)}
    =
    \frac{\eta(\tau)\eta(5\tau)}{\eta(10\tau)^2},
\]
is equivalent (after multiplying by $\eta(10\tau)\E_1(\tau) \E_2(\tau) \E_3(\tau) \E_4(\tau) = \eta(\tau)\eta(10\tau)/\eta(5\tau)$ and a power of $q$) to the fact that
\begin{equation} \label{eq:undecuple-det}
    U\left(q, q^{10}\right) - qU\left(q^3, q^{10}\right) = \frac{(q; q)^2}{\left(q^{10}; q^{10}\right)},
\end{equation}
where $U(x, q) := (q; q)\la x; q \ra \la qx^2; q^2 \ra^2 \la x^2; q^2 \ra \la q x^4 ; q^2 \ra$ denotes the undecuple product from \cref{prop:undecuple}. Finally, \cref{eq:undecuple-det} is precisely the statement that the determinant of the matrix in \cref{eq:explicit-matrix} equals $1$; we will prove this in the next section (non-circularly), based on a generalization of \cref{prop:2x2-determinant} and 4 identities of Slater \cite[(94),(96),(98),(99)]{slater1952further}, which are also implied by \cref{thm:bases-proportional}. 
\end{proof}

\begin{remark}
It would be interesting to see if our undecuple product identity (\cref{prop:undecuple}) can be used to prove \cref{eq:undecuple-det}, or if it can be combined with \cref{eq:undecuple-det} to produce other identities of eta quotients.
\end{remark}

\begin{proposition}[Level-13 generalized eta quotients] \label{prop:level-13}
For the level $N = 13$ and $\tau \in \H^+$,
\[
\sum_{g=1}^6
\frac{\E_{g}(\tau)\E_{2g}(\tau)\E_{6g}(\tau)}{\E_g(2\tau)} = 0.
\]
%checked numerically :)
\end{proposition}

\begin{proof}
Specialize the second nonuple product identity from \cref{prop:nonuple2} at $x = 1$; as before, the left-hand side of \cref{eq:nonuple2} becomes null due to the factor of $\la x; q \ra$.
\end{proof}

\begin{question}
Is there an analogue of the polynomial identity in \cref{cor:eta-poly} for the level $N = 13$ (or for other levels), building on \cref{prop:level-13}? One could attempt, for example, to determine closed forms (in terms of $\eta$) for the coefficients of $X^4, X^3, X^2$ and $X$ in
\[
    \prod_{g=1}^6 \left(X -
    \frac{\E_{g}(\tau)\E_{2g}(\tau)\E_{6g}(\tau)}{\E_g(2\tau)}\right).
\]

More broadly, is there a wider class of generalized eta quotients that occur as roots of polynomials whose coefficients are Dedekind eta quotients, and can this be used to compute special values of the $\E_g$ quotients as algebraic integers?
\end{question}

\section{Rogers--Ramanujan type identities and mock theta functions}\label{sec:rog-ram}

%\fix{CHANGE THIS!}
%\cref{subsec:two-var-generaliz} describes several interrelated ways to generalize Rogers--Ramanujan type identities to results in two variables, and shows that our main results in \cref{thm:bases-proportional,thm:bases-proportional-2} are equivalent to two such generalizations, given in \cref{thm:twisted-rog-ram,thm:twisted-rog-ram-var}. These theorems concern \emph{twisted} sums (in the sense of \cref{eq:w-coeffs}), which we study in greater detail in \cref{subsec:w-coeffs}; the proofs of our main results are then given in \cref{subsec:proofs-of-main-thms}. Afterwards, \cref{subsec:mock-theta} gives applications of these results to mock theta functions, and \cref{subsec:w-coeffs} concerns the change-of-basis identities anticipated in \cref{prop:2x2-determinant}.

\subsection{Two-variable generalizations of Rogers--Ramanujan-type identities}\label{subsec:two-var-generaliz}
There are two main ways to extend one-variable Rogers--Ramanujan-type sums to two-variable expressions in $x$ and $q$, living inside some space $\mT_D(f_1, \ldots, f_m)$. These ways correspond to \cref{lem:proofs-by-fourier,lem:proofs-by-value}, i.e., to identifying Fourier coefficients or special values). For instance, the sum
\begin{equation} \label{eq:two-variable-statement-sum}
    \sum_{m \in \Z, n \ge 0} \frac{q^{\binom{m}{2} + \binom{n+1}{2}}}{(q; q)_n} (-x)^{m+n}
    =
    \cdots +
    \sum_{n \ge 0} \frac{q^{n^2+n}}{(q; q)_n} x^0
    -
    \sum_{n \ge 0} \frac{q^{n^2}}{(q; q)_n} x^1
    +
    \cdots
    \quad \in \mT_\C\left(qx^2 -x \right),
\end{equation}
from \cref{prop:2var-rog-ram-var}, extends the Rogers--Ramanujan sums from \cref{eq:rog-ram} in the sense that it contains them as Fourier coefficients (and so \cref{prop:2var-rog-ram-var} generalizes \cref{eq:rog-ram}). Generalizations by value identification are more common in literature; in fact, the original way to generalize the Rogers--Ramanujan identities from \cref{eq:rog-ram} to a statement in two variables is due to Rogers:
%(reference Hardy's survey; GIVE A THEOREM AND A CHAPTER NUMBER, ETC.; do a proper reference!)
\begin{proposition}[Rogers, 1894] \label{prop:rogers}
%https://londmathsoc.onlinelibrary.wiley.com/doi/pdf/10.1112/plms/s1-25.1.318?casa_token=Z-slORQSLvUAAAAA:1oxps9ephedguAq3-bfT9FfYDRauTg2TPEHRP7pzx7uhy1uz5S4Vlk5lMc5O4AiVFLKS0BozLh7vO2w
In $\mT_\Hplus(1, qx)$, one has (see \cite[p.~292--294]{hardy1979introduction} and \cite[p.~330]{rogers1893second})
\begin{equation} \label{eq:rogers}
\begin{aligned}
    \sum_{n \geq 0} \frac{q^{n^2}}{(q; q)_n} x^n 
    &= 
    \sum_{n \geq 0}(-1)^n \frac{q^{5\binom{n}{2}+2n}}{(q; q)_n \left(q^n x; q\right)} \left(1 - q^{2n} x\right) x^{2n}
    \\
    &=
    \sum_{n \ge 0} 
    \frac{q^{n^2}}{\left(q^2; q^2\right)_n} x^n \left(-q^{2n+2} x; q^2\right).
\end{aligned}
\end{equation}
In fact, the left-hand side lies in $\mT_\C(1, qx)$, and thus all the singularities on the right are removable.
\end{proposition}
\begin{remark}
The two Rogers--Ramanujan identities in \cref{eq:rog-ram} can be recovered from the first equality in \cref{eq:rogers} by letting $x \to 1$, respectively $x \to q$, and then using \cref{eq:canonical-formula} with $d = 5$. Also, the series in the left-hand side of \cref{eq:rogers} will reappear in  \cref{cor:rogers-type-sums}, which will raise the problem of finding a similar identity for $\sum_{n \geq 0} q^{n^2} (q^4; q^4)_n^{-1} x^{2n} \in \mT_\C\left(qx^2, 1\right)$; see \cref{qtn:rogers-type-sums}.
\end{remark}
\begin{proof}[Proof in our framework]
We saw in \cref{ex:two-var-rog-ram,ex:two-var-rog-ram-ctd} that $\sum_{n \ge 0} \frac{q^{n^2}}{(q; q)_n} x^n \in \mT_\C(1, qx)$, and that $\dim \mT_\C(1, qx) = \dim \mT_\Hplus(1, qx) = 1$. A significantly lengthier computation shows that the right-hand side lies in $\mT_\Hplus(1, qx)$ as well; see the function $H_1$ in \cite[p.~292--294]{hardy1979introduction}. Identifying the coefficients of $x^0$, which are equal to $1$ on both sides of \cref{eq:rogers}, completes the proof of the first equality.

For the second equality, it is easier to take $x \mapsto x^2$; denote 
\[
    L(z) := \sum_{n \geq 0} \frac{q^{n^2}}{(q; q)_n} x^{2n}
    \qquad\text{and}\qquad  
    R(z) := \sum_{n \ge 0} 
    \frac{q^{n^2}}{\left(q^2; q^2\right)_n} x^{2n} \left(-q^{2n+2} x^2; q^2\right),
\]
where $R(z)$ is almost a twisted sum in the sense of \cref{eq:w-coeffs}. Since $L(z/2) \in \mT_\C(1, qx)$, one can obtain
\[
    \begin{cases}
    L(z) - L(z+\tau/2) = qx^2 L(z+\tau) \\
    L(z+\tau/2) - L(z+\tau) = q^2x^2 L(z+3\tau/2) 
    \quad \Rightarrow \quad 
    L \in \mT_\C\left(1 + qx^2 + q^2 x^2, -q^5 x^4\right).
    \\
    L(z+\tau) - L(z+3\tau/2) = q^3x^2 L(z+2\tau).
    \end{cases}
\]
Hence it is natural to attempt to show that $R \in \mT_\C\left(1 + qx^2 + q^2 x^2, -q^5 x^4\right)$ as well, and this follows with little effort after expanding $R(z) - \left(1 + q^2 x^2\right)R(z+\tau)$. Now from \cref{prop:upper-bounds}.(ii) for $n_0 = -3$ one obtains that $\dim \mT_\C\left(1 + qx^2 + q^2 x^2, -q^5 x^4\right) \le 4$, and thus by \cref{lem:proofs-by-fourier} it suffices to check that $\hat{L}(k) = \hat{R}(k)$ for $k \in \{-3, -2, -1, 0\}$. But these Fourier coefficients are $0, 0, 0, 1$ respectively for both $L$ and $R$, proving that $L = R$.
\end{proof}
Other proofs of the Rogers--Ramanujan identities based on two-variable generalizations are due to Sills \cite{sills2019finite}, who interpreted finitized versions of Rogers--Ramanujan type sums as Fourier coefficients (in $z$) of two-variable sums, and to Bressoud \cite{bressoud1983easy}, who generalized Cauchy's finite triple product from \cref{prop:Cauchy}. Unlike these results and Rogers' \cref{prop:rogers}, our \cref{prop:2var-rog-ram} (concerning the sum in \cref{eq:two-variable-statement-sum}) can be recovered easily from the Rogers--Ramanujan identities in \cref{eq:rog-ram}:

\begin{proof}[Proof that \cref{eq:rog-ram} $\Leftrightarrow$ \cref{prop:2var-rog-ram}]
The first equality in \cref{prop:2var-rog-ram} and the membership to $\mT_\C\left(qx^2 - x\right)$ follow immediately from the expansions
\[
    (q; q)\la x; q\ra = \sum_{m \in \Z} q^{\binom{m}{2}} (-x)^m
    \in \mT_\C(-x)
    \qquad \text{and} \qquad 
    (qx; q) = \sum_{n \ge 0} (-1)^n \frac{q^{\binom{n+1}{2}}}{(q; q)_n} (-x)^n \in \mT_\C(1-qx),
\]
due to \cref{eq:triple} and \cref{eq:basic1}. The second equality in \cref{prop:2var-rog-ram} implies \cref{eq:rog-ram} by identifying the coefficients of $x^0$ as $x^1$, as seen in \cref{eq:two-variable-statement-sum}. Conversely, identifying the coefficients of $x^0$ and $x^1$ is enough to prove an equality in $\mT_\Hminus\left(qx^2 - x\right)$, by \cref{lem:proofs-by-fourier} (or in this case, directly by \cref{eq:fourier-coeffs-basis}).
\end{proof}

But although \cref{prop:2var-rog-ram} is ultimately just a reformulation of \cref{eq:2var-rog-ram}, it makes apparent a connection to the septuple product identity in \cref{eq:septuple}, which is our main result in \cref{thm:bases-proportional}. In its turn, the bases proportionality in \cref{thm:bases-proportional} is equivalent to yet another two-variable generalization of \cref{eq:rog-ram}, concerning \emph{twisted} sums (in the sense of \cref{eq:w-coeffs}) and given below.

\begin{theorem}[Twisted two-variable Rogers--Ramanujan]\label{thm:twisted-rog-ram}
For $\tau \in \H^+$ and $z \in \C$, one has
\[
    \begin{pmatrix}
    \sum_{n \in \Z} q^{n^2} x^{2n}  \left( q^{n+1}x; q\right) \vspace{0.1cm} \\ 
    \sum_{n \in \Z} q^{n^2 + n}x^{2n+1}  \left( q^{n+1}x; q\right)
    \end{pmatrix}
    = 
    \begin{pmatrix}
    A(q) & B(q) \vspace{0.1cm}\\
    C(q) & D(q)
    \end{pmatrix}
    \begin{pmatrix}
    \vect{qx^2}{0} \vspace{0.1cm} \\
    \vect{qx^2}{1}
    \end{pmatrix},
\]
where 
\begin{equation} \label{eq:explicit-matrix}
    \begin{pmatrix}
    A(q) & B(q) \vspace{0.1cm}\\
    C(q) & D(q)
    \end{pmatrix}
    =
    \frac{\left(q^{10}; q^{10}\right)}{(q; q)}
    \begin{pmatrix}
    \la q; q^{10}\ra \la q^8; q^{20} \ra 
    &
    -q\la q^4; q^{10} \ra \la q^2; q^{20} \ra 
    \vspace{0.1cm}
    \\
    -\la q^3; q^{10}\ra \la q^4; q^{20} \ra 
    &
    \la q^2; q^{10} \ra \la q^6; q^{20} \ra
    \end{pmatrix}.
\end{equation}
\end{theorem}

\begin{remark}
One can recover the Rogers--Ramanujan identities by taking $x = 1$ above, as we will see shortly; thus \cref{thm:bases-proportional} is a generalization of \cref{eq:rog-ram} by value identification, unlike \cref{prop:2var-rog-ram}.
Also, \cref{cor:3-var-2x2-matrix-ctd} and \cref{eq:slater} will show that the matrix in \cref{eq:explicit-matrix} has determinant equal to $1$. 
\end{remark}

\begin{proof}[Proof that \cref{thm:twisted-rog-ram} $\Leftrightarrow$ \cref{thm:bases-proportional}]
Using the triple product identity in the form $\vect{-qx}{0} = (q;q) \left(x^{-1}; q\right) (qx; q)$,
%\[
%    (q; q) \la -x; q\ra \la qx^2; q^2\ra = \vect{-qx^3}{0} + \vect{-qx^3}{1}
%    \qquad\text{and} \qquad 
%    (q;q) \left(x^{-1}; q\right) (qx; q) = \vect{-qx}{0},
%\]
one can rewrite \cref{thm:bases-proportional} as
\[
    (q; q)\la x^2; q \ra
    \begin{pmatrix}
    \left(x^{-1}; q\right)\vect{qx^2-x}{0} 
    \vspace{0.1cm} \\
    \left(x^{-1}; q\right)\vect{qx^2-x}{1}
    \end{pmatrix}
    = 
    -\vect{-qx}{0}
    \begin{pmatrix}
    \vect{-qx^5}{1} + \vect{-qx^5}{2} 
    \vspace{0.1cm} \\
    \vect{-qx^5}{0} + \vect{-qx^5}{3} 
    \end{pmatrix},
\]
Now recall the formulae for canonical basis vectors from \cref{cor:kind2} (use $y = 1$):
\[
    \begin{pmatrix}
    \left(x^{-1}; q\right)\vect{qx^2-x}{0} 
    \vspace{0.1cm} \\
    \left(x^{-1}; q\right)\vect{qx^2-x}{1}
    \end{pmatrix}
    =
    \begin{pmatrix}
    \sum_{n \in \Z}
    q^{n^2} x^{2n} \left(q^{-n+1}x^{-1}; q\right)
    \vspace{0.1cm} \\
    \sum_{n \in \Z}
    q^{n^2+n} x^{2n+1} \left(q^{-n}x^{-1}; q\right)
    \end{pmatrix}.
\]
As observed in the proof of \cref{prop:exact-formulae-kind2}, the two series in the right-hand side are just \emph{twisted} versions of the series $\vect{qx^2}{0}$ and $\vect{qx^2}{1}$, so they lie in $\mT_\C\left(qx^2\right)$. But by \cref{prop:monomial-spaces}, we have $\mT_\C\left(qx^2\right) \subset \mS_\C(1)$, and so these series are invariant under $x \mapsto x^{-1}$; thus 
\[
    \begin{pmatrix}
    \sum_{n}
    q^{n^2} x^{2n} \left(q^{-n+1}x^{-1}; q\right)
    \vspace{0.1cm} \\
    \sum_{n}
    q^{n^2+n} x^{2n+1} \left(q^{-n}x^{-1}; q\right)
    \end{pmatrix}
    =
    \begin{pmatrix}
    \sum_{n}
    q^{n^2} x^{-2n} (q^{-n+1}x; q)
    \vspace{0.1cm} \\
    \sum_{n}
    q^{n^2+n} x^{-2n-1} (q^{-n}x; q)
    \end{pmatrix}
    =
    \begin{pmatrix}
    \sum_{n}
    q^{n^2} x^{2n} (q^{n+1}x; q)
    \vspace{0.1cm} \\
    \sum_{n}
    q^{n^2+n} x^{2n+1} (q^{n+1}x; q)
    \end{pmatrix},
\]
where all summations are over $n \in \Z$, and we took $n \mapsto -n$ and $n \mapsto -n-1$ in the last equality. Overall, we have found that \cref{thm:bases-proportional} is equivalent to
\begin{equation} \label{eq:first-thm-reformulation}
    (q; q) \la x^2; q \ra \begin{pmatrix}
    \sum_{n \in \Z} q^{n^2} x^{2n}  \left( q^{n+1}x; q\right) \vspace{0.1cm} \\ 
    \sum_{n \in \Z} q^{n^2 + n}x^{2n+1}  \left( q^{n+1}x; q\right)
    \end{pmatrix}
    = 
    -\vect{-qx}{0}
    \begin{pmatrix}
    \vect{-qx^5}{1} + \vect{-qx^5}{2} 
    \vspace{0.1cm} \\
    \vect{-qx^5}{0} + \vect{-qx^5}{3} 
    \end{pmatrix}.
\end{equation}
Now by applying the multiplication identities from \cref{lem:m-identities} four times, and using \cref{eq:substitutions-quintuple} with $a = 10$ on each $q$-coefficient (as in Step 3 from \cref{subsec:higher-order}), we find that
\[
    \vect{-qx}{0}
    \begin{pmatrix}
    \vect{-qx^5}{1} + \vect{-qx^5}{2} 
    \vspace{0.1cm} \\
    \vect{-qx^5}{0} + \vect{-qx^5}{3} 
    \end{pmatrix}
    =
    (q; q)
    \begin{pmatrix}
    A(q) & B(q) \vspace{0.1cm}\\
    C(q) & D(q)
    \end{pmatrix}
    \begin{pmatrix}
    \vect{q^2x^6}{2} - \vect{q^2x^6}{0}
    \vspace{0.1cm} \\
    \vect{q^2x^6}{3} - \vect{q^2x^6}{-1}
    \end{pmatrix},
\]
where $A(q), B(q), C(q), D(q)$ are exactly as in \cref{eq:explicit-matrix}. Hence to prove that \cref{thm:bases-proportional} is equivalent to \cref{thm:twisted-rog-ram}, it suffices to show that
\begin{equation} \label{eq:to-show-quintuple}
    \la x^2; q \ra 
    \begin{pmatrix}
    \vect{qx^2}{0}
    \vspace{0.1cm} \\
    \vect{qx^2}{1}
    \end{pmatrix}
    \stackrel{?}{=}
    \begin{pmatrix}
    \vect{q^2x^6}{0} - \vect{q^2x^6}{2}
    \vspace{0.1cm} \\
    \vect{q^2x^6}{-1} - \vect{q^2x^6}{3}
    \end{pmatrix}.
\end{equation}
Now by \cref{eq:substitutions-triple} and \cref{eq:substitutions-quintuple}, we have
\[
\begin{aligned}
    \la x^2; q \ra 
    \begin{pmatrix}
    \vect{qx^2}{0}
    \vspace{0.1cm} \\
    \vect{qx^2}{1}
    \end{pmatrix}
    &=
    \left(q^2; q^2\right) 
    \la x^2; q^2 \ra 
    \la qx^2; q^2 \ra 
    \begin{pmatrix}
    \la -qx^2; q^2 \ra
    \vspace{0.1cm} \\
    x\la -q^2 x^2; q^2 \ra
    \end{pmatrix}
    \\
    &=
    \left(q^2; q^2\right) 
    \begin{pmatrix}
    \la x^2; q^2 \ra \la q^2x^4; q^4 \ra
    \vspace{0.1cm} \\
    -x^3\la qx^2; q^2 \ra \la q^4x^4; q^4 \ra
    \end{pmatrix}
    \qquad
    =
    \begin{pmatrix}
    \vect{q^2x^6}{0} - \vect{q^2x^6}{2}
    \vspace{0.1cm} \\
    -x^3\left(\vect{q^5x^6}{0} - q\vect{q^5x^6}{2}\right)
    \end{pmatrix},
\end{aligned}
\]
which simplifies to the right-hand side of \cref{eq:to-show-quintuple}, using that $q\vect{q^5 x^6}{2} = \vect{q^5x^6}{-4}$ and $x^3 \in \mT_\C\left(q^{-3}\right)$. Thus indeed \cref{thm:bases-proportional} $\Leftrightarrow$ \cref{thm:twisted-rog-ram}.
\end{proof}

We can now give a first proof of \cref{thm:twisted-rog-ram} based on \cref{lem:proofs-by-value} and some mock theta identities of Watson, thus establishing our main result in \cref{thm:bases-proportional}. A second proof based on expanding twisted sums (\cref{lem:t-identities}) will follow in \cref{subsec:w-coeffs}.

\begin{proof}[First proof of \cref{thm:twisted-rog-ram}]
By \cref{lem:proofs-by-value} for $\mT_\C\left(qx^2\right)$, it suffices to prove \cref{thm:twisted-rog-ram} when $x \in \{1, -1\}$. After noting that $\left(q^{n+1}; q\right) = 0$ for $n < 0$, the statement at $x = 1$ (i.e., $z = 0$) becomes
\[
    \begin{pmatrix}
    \sum_{n \ge 0} q^{n^2}\left(q^{n+1}; q\right) \vspace{0.1cm} \\ 
    \sum_{n \ge 0} q^{n^2 + n}\left(q^{n+1}; q\right)
    \end{pmatrix}
    \stackrel{?}{=} 
    \begin{pmatrix}
    A(q) & B(q) \vspace{0.1cm}\\
    C(q) & D(q)
    \end{pmatrix}
    \begin{pmatrix}
    \vect{qx^2}{0}(0)\vspace{0.1cm} \\
    \vect{qx^2}{1}(0)
    \end{pmatrix}.
\]
Now in last part of the previous proof, we saw (unconditionally, using \cref{lem:m-identities}) that 
\[
    \begin{pmatrix}
    A(q) & B(q) \vspace{0.1cm}\\
    C(q) & D(q)
    \end{pmatrix}
    \begin{pmatrix}
    \vect{qx^2}{0}\vspace{0.1cm} \\
    \vect{qx^2}{1}
    \end{pmatrix}
    =
    \frac{-\vect{-qx}{0}}{(q; q)\la x^2; q \ra}
    \begin{pmatrix}
    \vect{-qx^5}{1} + \vect{-qx^5}{2} 
    \vspace{0.1cm} \\
    \vect{-qx^5}{0} + \vect{-qx^5}{3} 
    \end{pmatrix}
\]
as an identity of meromorphic functions of $z$. Using that $\vect{-qx}{0} = (q; q)\la xq; q\ra$ and letting $z \to 0$, our claim to prove becomes
\[
    \begin{pmatrix}
    \sum_{n \ge 0} q^{n^2}\left(q^{n+1}; q\right) \vspace{0.1cm} \\ 
    \sum_{n \ge 0} q^{n^2 + n}\left(q^{n+1}; q\right)
    \end{pmatrix}
    \stackrel{?}{=} 
    \frac{1}{2}
    \begin{pmatrix}
    \vect{-qx^5}{1}(0) + \vect{-qx^5}{2}(0)
    \vspace{0.1cm} \\
    \vect{-qx^5}{0}(0) + \vect{-qx^5}{3}(0)
    \end{pmatrix}
    \\
    =
    \begin{pmatrix}
    \sum_{n \in \Z} (-1)^n q^{\frac{5n^2+n}{2}}
    \vspace{0.1cm} \\
    \sum_{n \in \Z} (-1)^n q^{\frac{5n^2+3n}{2}}
    \end{pmatrix},
\]
which is equivalent to the Rogers--Ramanujan identities from \cref{eq:rog-ram} (which we deduced from \cref{prop:rogers}), after dividing by $(q; q)$ and applying the triple product identity.

Next, the statement of \cref{thm:twisted-rog-ram} at $x = -1$ (i.e., $z = 1/2$) becomes, after applying the triple product identity in \cref{eq:triple} for $x \mapsto -qx^2$ and $x \mapsto -q^2x^2$,
\[
    \begin{pmatrix}
    \sum_{n \in \Z} q^{n^2} \left(-q^{n+1}; q\right)
    \\
    -\sum_{n \in \Z} q^{n^2} \left(-q^{n+1}; q\right)
    \end{pmatrix}
    \stackrel{?}{=}
    \begin{pmatrix}
    A(q) & B(q) \vspace{0.1cm}\\
    C(q) & D(q)
    \end{pmatrix}
    \begin{pmatrix}
    \left(q^2; q^2\right)\la -q; q^2 \ra\vspace{0.1cm} \\
    -\left(q^2; q^2\right)\la -q^2; q^2 \ra
    \end{pmatrix}.
\]
But since we have already shown that \cref{thm:twisted-rog-ram} holds at $x = 1$, we know that
\begin{equation} \label{eq:5-and-10}
    (q; q)
    \begin{pmatrix}
    \la q; q^5 \ra^{-1} \vspace{0.1cm} \\ 
    \la q^2; q^5 \ra^{-1}
    \end{pmatrix}
    =
    \begin{pmatrix}
    \sum_{n \ge 0} q^{n^2}\left(q^{n+1}; q\right) \vspace{0.1cm} \\ 
    \sum_{n \ge 0} q^{n^2 + n}\left(q^{n+1}; q\right)
    \end{pmatrix}
    =
    \begin{pmatrix}
    A(q) & B(q) \vspace{0.1cm}\\
    C(q) & D(q)
    \end{pmatrix}
    \begin{pmatrix}
    \left(q^2; q^2\right)\la -q; q^2 \ra\vspace{0.1cm}\vspace{0.1cm} \\
    \left(q^2; q^2\right)\la -q^2; q^2 \ra
    \end{pmatrix},
\end{equation}
and by subtracting this from the previous equation it remains to prove that
\[
    \begin{pmatrix}
    \sum_{n \in \Z} q^{n^2} \left(-q^{n+1}; q\right)
    \\
    -\sum_{n \in \Z} q^{n^2} \left(-q^{n+1}; q\right)
    \end{pmatrix}
    \stackrel{?}{=}
    (q; q)
    \begin{pmatrix}
    \la q; q^5 \ra^{-1} \vspace{0.1cm} \\ 
    \la q^2; q^5 \ra^{-1}
    \end{pmatrix}
    +
    \begin{pmatrix}
    A(q) & B(q) \vspace{0.1cm}\\
    C(q) & D(q)
    \end{pmatrix}
    \begin{pmatrix}
    0\vspace{0.1cm} \\
    -2\left(q^2; q^2\right)\la -q^2; q^2 \ra
    \end{pmatrix}.
\]
In the left-hand side, break the sums into $n \ge 0$ and $n < 0$, and use the notation from \cref{eq:mock-theta}. On the right, substitute $B(q) = -q\left(-q; q\right)/\la q^8; q^{20}\ra$ and $D(q) = \left(-q; q\right)/\la q^4; q^{20} \ra$, which follow from \cref{eq:explicit-matrix}. Dividing both sides by $(-q; q)$ and using the triple product identity in the form $\sum_{n \in \Z} q^{n^2+n} = \left(q^2; q^2\right) \la -q^2; q^2 \ra$, we see that the claim above is equivalent to
\begin{equation}\label{eq:watson}
    \begin{pmatrix}
    f_0 + 2\psi_0(q)
    \\
    f_1 + 2\psi_1(q)
    \end{pmatrix}
    \stackrel{?}{=}
    \frac{(q; q)}{(-q; q)}
    \begin{pmatrix}
    \la q; q^5 \ra^{-1} \vspace{0.1cm} \\ 
    -\la q^2; q^5 \ra^{-1}
    \end{pmatrix}
    +
    4q \sum_{n \ge 0} q^{n^2+n} 
    \begin{pmatrix}
    \la q^8; q^{20} \ra^{-1}
    \vspace{0.1cm} \\
    \la q^4; q^{20} \ra^{-1}
    \end{pmatrix}.
\end{equation}
But \cref{eq:watson} is precisely the content of equations $(3_0) + 2(1_0)$ and $(3_1) + 2(1_1)$ from \cite{watson1937mock}, which completes our proof assuming the work of Watson.
\end{proof}

\begin{remark}
Watson \cite{watson1937mock} deduced his relations $(1_0), (1_1), (3_0)$ and $(3_1)$ using $q$-series manipulations based on the Rogers--Ramanujan identities from \cref{eq:rog-ram}, the formulae in \cref{cor:basic-identities} (Watson's $(E_1)$ and $(E_2)$), and specializations of the triple product identity (Watson's $(J_1)$ and $(J_2)$). The proof above also shows that \cref{thm:twisted-rog-ram} \emph{implies} the mock theta identities in \cref{eq:watson} by taking $x = \pm 1$, which is relevant since we will give a self-contained proof of \cref{thm:twisted-rog-ram} in \cref{subsec:w-coeffs}.
\end{remark}

Taking $x = i$ in \cref{thm:twisted-rog-ram} yields two relatives of the Rogers--Ramanujan identities:

\begin{corollary}[Imaginary Rogers--Ramanujan identities] \label{cor:more-rog-ram}
For $\tau \in \H^+$, one has
\[
\begin{aligned}
    \sum_{n \in \Z} (-1)^n \frac{q^{n^2}}{(iq; q)_n}
    &=
    \frac{\left(q; q^2\right)}{(iq; q)}
    \left(q^{10}; q^{10}\right) \la q; q^{10} \ra \la q^8; q^{20} \ra ,
    \\
    \sum_{n \in \Z} (-1)^n \frac{q^{n^2+n}}{(iq; q)_n}
    &=
    i\frac{\left(q; q^2\right)}{(iq; q)}\left(q^{10}; q^{10}\right) \la q^3; q^{10} \ra \la q^4; q^{20} \ra.
\end{aligned}
\]
\iffalse
Moreover, for any $r > 0$, one has
\begin{gather} \label{eq:r-corollary-1}
    \sum_{n \in \Z} \frac{q^{n^2-rn}}{\left(q; q^2\right)_n}
    =
    \frac{\left(q^2; q^2\right)\la -q^{1-r}; q^2 \ra}{\left(-q^r; q^2\right) \left(q; q^2\right)},
    \qquad\qquad
    \sum_{n \in \Z} \frac{q^{n^2-rn}}{\left(-q; q^2\right)_n}
    =
    \frac{\left(q^2; q^2\right)\la -q^{1-r}; q^2 \ra}{\left(q^r; q^2\right)\left(-q; q^2\right)},
    \\ \label{eq:r-corollary-2}
    \sum_{n \in \Z} \frac{q^{n^2-(r-1)n}}{\left(-q^2; q^2\right)_n}
    =
    \frac{\left(q^2; q^2\right)\la -q^{2-r}; q^2 \ra}{\left(q^r; q^2\right) \left(-q^2; q^2\right)}
    .
\end{gather}
\fi
\end{corollary}

Having proven \cref{thm:bases-proportional}, we now move on to our parallel result in \cref{thm:bases-proportional-2}. Its proof follows a similar strategy, but it is significantly easier for reasons to be explained shortly.
%\fix{Remark the fact that the LHS is essentially the same as that of the previous proposition up to multipl by ... (both express the unique function (up to a constant) in the one-dimensional $\mT_\C(qx^2-x)$, in different bases. This motivates having the proposition as a whole (as opposed to two separated identities...)}
\begin{proposition}[Two-variable Rogers--Ramanujan variation] \label{prop:2var-rog-ram-var}
In $\mT_\Hminus\left(qx^2 - q^{-1}\right)$, one has
\[
    \frac{(q; q)\la x; q\ra (qx; q)}{(-q; q)^2\left(- x^{-1}; q\right)}
    %-(q; q)(q;q^2)^2 \frac{\la x; q \ra^2}{x(x^{-2}; q^2)}
    =
    2\left(-q^2; q^2\right) \vect{qx^2 - q^{-1}}{0}
    - \left(-q; q^2\right)\vect{qx^2 - q^{-1}}{1}.
%don't change the scalar because these are the coeffs you get in M(2, 4)_2, etc.
\]
\end{proposition}
\begin{remark}
The $q$-coefficients in the right-hand side above are the same (up to a sign) as the $q$-coefficients on the right of the octuple identity from \cref{eq:octuple}. Moreover, multiplying the product above by the left-hand side of \cref{thm:bases-proportional-2} yields the octuple product in \cref{eq:octuple}. Hence under \cref{thm:bases-proportional-2}, the octuple identity and \cref{prop:2var-rog-ram-var} are equivalent.
\end{remark}
\begin{proof}[Proof of \cref{prop:2var-rog-ram-var}]
Rewrite the left-hand side as 
\[
\begin{aligned}
    \frac{-x^{-1} (q; q)\la x; q\ra^2}{(-q; q)^2\left( x^{-2}; q^2\right)}
    &= 
    -\frac{x^{-1}}{(-q; q)^2} \left(\left(-q; q^2\right)^2 \vect{x^2}{0} - 2\left(-q^2; q^2\right)^2\vect{x^2}{1}\right) \cdot \sum_{n \ge 0} \frac{1}{\left(q^2; q^2\right)_n} x^{-2n}
    \\
    &= 2\frac{\left(-q^2; q^2\right)}{\left(-q; q^2\right)} \sum_{m \in \Z, n \ge 0} \frac{q^{m^2}}{(q^2; q^2)_n}x^{2m-2n} - \frac{\left(-q; q^2\right)}{\left(-q^2; q^2\right)} \sum_{m \in \Z, n \ge 0} \frac{q^{m^2+m}}{(q^2; q^2)_n}x^{2m-2n+1},
\end{aligned}
\]
where we used the squared triple product identity from \cref{prop:squared-triple}, and \cref{eq:basic2} in the first equality and simply multiplied sums out in the second equality. Now using \cref{tbl:spaces}, we see that the left-hand side product lies in $\mT_\Hminus\left(q(-x)^2 \left(1 - x^{-2} q^{-2}\right)\right) = \mT_\Hminus\left(qx^2 - q^{-1}\right)$, so it is a linear combination of $\vect{qx^2-q^{-1}}{0}$ and $\vect{qx^2-q^{-1}}{1}$. To finish it suffices to identify the coefficients of $x^0$ and $x^1$ in the sum above; this amounts to the identities
\begin{equation} \label{eq:rog-ram-var}
    \sum_{n \ge 0} \frac{q^{n^2}}{(q^2; q^2)_n} = \left(-q; q^2\right) \qquad\text{and}\qquad 
    \sum_{n \ge 0} \frac{q^{n^2+n}}{(q^2; q^2)_n} = \left(-q^2; q^2\right),
\end{equation}
which follow from \cref{eq:basic1} for $q \mapsto q^2$ and $x \mapsto -q$, respectively $x \mapsto -q^2$.
\end{proof}
\begin{remark}
While \cref{prop:2var-rog-ram} generalizes the original Rogers--Ramanujan identities in \cref{eq:rog-ram}, \cref{prop:2var-rog-ram-var} is a generalization of the (much easier) Rogers--Ramanujan type identities in \cref{eq:rog-ram-var}.
\end{remark}

Next, we provide an analogue of \cref{thm:twisted-rog-ram}, which is equivalent to \cref{thm:bases-proportional-2}. This analogue turns out to be a direct consequence of our work in \cref{subsec:canonical-basis-vectors}, but we list it as a theorem to illustrate the parallelism between \cref{thm:bases-proportional,thm:bases-proportional-2}.

\begin{theorem}[Twisted two-variable Rogers--Ramanujan variation] \label{thm:twisted-rog-ram-var}
For $\tau \in \H^+$ and $z \in \H^-$,
\[
    \left(x^{-2}; q^2\right)
    \begin{pmatrix}
    \sum_{n \in \Z} q^{n^2} x^{2n} \left(-q^{2n+1}; q^2\right)
    \vspace{0.1cm} \\
    %\vect{qx^2-q^{-1}}{0} \\
    \sum_{n \in \Z} q^{n^2+n} x^{2n+1} \left(-q^{2n+2}; q^2\right)
    %\vect{qx^2-q^{-1}}{1}
    \end{pmatrix}
    =
    %\begin{pmatrix}
    %\left(-q; q^2\right)^{-1} & 0 
    %\vspace{0.1cm} \\
    %0 & \left(-q^2; q^2\right)^{-1}
    %\end{pmatrix}
    \begin{pmatrix}
    \vect{qx^2}{0}
    \vspace{0.1cm} \\
    \vect{qx^2}{1}
    \end{pmatrix}.
\]
\end{theorem}

\begin{proof}
Take $y = q^{-1}$ in \cref{cor:kind1}, then multiply the results by $\left(-q; q^2\right)$ and $\left(-q^2; q^2\right)$. 
\end{proof}

\begin{proof}[Proof that \cref{thm:bases-proportional-2} $\Leftrightarrow$ \cref{thm:twisted-rog-ram-var}]
Let $z \in \H^-$. Using that $\vect{-x^2}{1} = \left(q^2; q^2\right)\la qx^2; q^2 \ra$ (by \cref{eq:canonical-formula}) and that $\la x; q \ra = -x (qx; q)\left(x^{-1}; q\right)$, one can rephrase \cref{thm:bases-proportional-2} as
\[
    \vect{-x^2}{1}\left(x^{-2}; q^2\right)
    \begin{pmatrix}
    \vect{qx^2-q^{-1}}{0} 
    \vspace{0.1cm} \\
    \vect{qx^2-q^{-1}}{1} 
    \end{pmatrix}
    =
    (q; q)
    \begin{pmatrix}
    \vect{-qx^4}{1}
    \vspace{0.1cm} \\
    \vect{-qx^4}{0} + \vect{-qx^4}{2} 
    \end{pmatrix}.
\]
Using the formulae for canonical basis vectors from \cref{cor:kind1}, we can further rewrite this as
\[
    \vect{-x^2}{1}\left(x^{-2}; q^2\right)
    \begin{pmatrix}
    \sum_{n \in \Z} q^{n^2} x^{2n} \frac{\left(-q^{n+1}; q^2\right)}{\left(-q; q^2\right)}
    \vspace{0.1cm} \\
    %\vect{qx^2-q^{-1}}{0} \\
    \sum_{n \in \Z} q^{n^2+n} x^{2n+1} \frac{\left(-q^{n+2}; q^2\right)}{\left(-q^2; q^2\right)}
    %\vect{qx^2-q^{-1}}{1}
    \end{pmatrix}
    =
    (q; q)
    \begin{pmatrix}
    \vect{-qx^4}{1}
    \vspace{0.1cm} \\
    \vect{-qx^4}{0} + \vect{-qx^4}{2} 
    \end{pmatrix}.
\]
Now from \cref{lem:m-identities} and the triple product identity from \cref{eq:substitutions-triple} for $a = 4$, we find that
\[
    \vect{-x^2}{1}
    \begin{pmatrix}
    \vect{qx^2}{0} 
    \vspace{0.1cm} \\
    \vect{qx^2}{1} 
    \end{pmatrix}
    =
    \left(q^4; q^4\right)
    \begin{pmatrix}
    \la q^2; q^4 \ra & 0 
    \vspace{0.1cm} \\
    0 & \la q; q^4 \ra
    \end{pmatrix}
    \begin{pmatrix}
    \vect{-qx^4}{1}
    \vspace{0.1cm} \\
    \vect{-qx^4}{0} + \vect{-qx^4}{2} 
    \end{pmatrix}.
\]  
This proves that \cref{thm:twisted-rog-ram-var} is equivalent to \cref{thm:bases-proportional-2} since $\left(q^4; q^4\right)\la q^2; q^4 \ra / (q; q) = \left(-q; q^2\right)$ and $\left(q^4; q^4\right)\la q; q^4 \ra / (q; q) = \left(-q^2; q^2\right)$.
\end{proof}

\begin{remark}
Comparing our main results in \cref{thm:bases-proportional} (equivalent to \cref{thm:twisted-rog-ram}) and \cref{thm:bases-proportional-2} (equivalent to \cref{thm:twisted-rog-ram-var}), the latter one is characterized by the
\emph{non-mixing of even and odd powers of $x$} in the basis vectors $\vect{qx^2 - q^{-1}}{k}$ (which originates from the fact that the symmetry factor $qx^2 - q^{-1}$ only contains even powers). This has several consequences:
\begin{itemize}
    \item[(i).] There is no need for a $2 \times 2$ matrix in \cref{thm:twisted-rog-ram-var} (or by a different choice of scalars, the matrix would be diagonal), unlike in \cref{thm:twisted-rog-ram}. 
    \item[(ii).] \cref{thm:bases-proportional-2} and \cref{thm:twisted-rog-ram-var} were easier to prove than \cref{thm:bases-proportional} and \cref{thm:twisted-rog-ram}.
    \item[(iii).] By identifying even and odd powers of $x$, \cref{prop:2var-rog-ram-var} splits into two identities for $\vect{qx^2 - q^{-1}}{0}$ and $\vect{qx^2 - q^{-1}}{1}$, which are ultimately equivalent to \cref{cor:kind2}, \cref{thm:twisted-rog-ram-var} and \cref{thm:bases-proportional-2}. By contrast, \cref{thm:bases-proportional} is stronger than \cref{prop:2var-rog-ram}.
    \item[(iv).] There is no nonzero entire function in $\mT_\Hminus\left(qx^2 - q^{-1}\right)$ (i.e., $\mT_\C\left(qx^2 - q^{-1}\right) = \{0\}$). This is ultimately because $\frac{1}{1-x^2}$ and $\frac{x}{1-x^2}$ (which correspond to the tails of the coefficients of $q^0$ in $\vect{qx^2 - q^{-1}}{0}$ and $\vect{qx^2 - q^{-1}}{1}$) are linearly independent. By contrast, $\mT_\C\left(qx^2 - x\right)$ is spanned by the left-hand side of \cref{prop:2var-rog-ram}.
\end{itemize}
The non-mixing phenomenon extends (with respect to residue classes mod $d$) to any $\mT_\Hminus\left(\alpha x^d + \beta\right)$, which is ultimately what allowed us to obtain closed forms for canonical basis vectors of the first kind in \cref{prop:exact-formulae-kind1}. Similarly, the mixing phenomenon is manifested in all nontrivial cases of canonical basis vectors of the second kind, from \cref{prop:exact-formulae-kind2}.
\end{remark}

%We end this subsection with a few immediate consequences of \cref{thm:twisted-rog-ram} and \cref{thm:twisted-rog-ram-var,thm:twisted-rog-ram-var}.

\iffalse
\begin{proof}
For the first two identities, take $x = i$ in \cref{thm:twisted-rog-ram}; for the third one, take $q \mapsto -q$ in the top equality from \cref{thm:twisted-rog-ram-var}, and then $x = iq^{-r/2}$. For the fourth and fifth identities, take $x = q^{-r/2}$ in \cref{thm:twisted-rog-ram-var}. In the last 3 identities, one also needs to use \cref{eq:canonical-formula} for $d = 2$ and $\alpha = q$:
\[
    \vect{qx^2}{k}
    =
    \left(q^2; q^2\right) \la -q^{1+k} x^2; q^2 \ra x^k,
    \qquad\qquad 
    k \in \{0, 1\},
\]
specialized at the above values of $x$.
\end{proof}
\fi

\begin{remark}
Numerical generation also suggests identities for the cross-products in the (now proven) \cref{thm:bases-proportional,thm:bases-proportional-2}:
\begin{equation}\label{eq:cross-product-1}
\begin{aligned}
    &\vect{qx^2 - x}{1}\left(\vect{-qx^5}{1} + \vect{-qx^5}{2}\right)\\[-5pt]
    = \
    &\vect{qx^2 - x}{0}\left(\vect{-qx^5}{0} + \vect{-qx^5}{3}\right) 
    =  
    \frac{\left(q^{35}; q^{35}\right)}{\left(x^{-1}; q\right)} 
    \Big( 
    \la q^{15}; q^{35} \ra \left(\vect{-q^2 x^7}{0} + \vect{-q^2 x^7}{3}\right) 
    \\[-2pt]
    &\qquad\qquad\qquad\qquad\qquad \qquad\qquad\qquad\qquad\qquad
    -q^2\la q^5; q^{35} \ra \left(\vect{-q^2 x^7}{1} + \vect{-q^2 x^7}{2}\right)
    \\
    &\qquad\qquad\qquad\qquad\qquad \qquad\qquad\qquad\qquad\quad\ 
    -q\la q^{10}; q^{35} \ra \left(\vect{-q^2 x^7}{-1} + \vect{-q^2 x^7}{4}\right)
    \Big),
\end{aligned}
\end{equation}
respectively
\begin{equation} \label{eq:cross-product-2}
\begin{aligned}
    & \vect{qx^2 - q^{-1}}{1}\vect{-qx^4}{1} 
    \\[-5pt]
    = \
    &\vect{qx^2 - q^{-1}}{0}\left(\vect{-qx^4}{0} + \vect{-qx^4}{2}\right)
    =  
    \frac{\left(q^2; q^2\right)}{\left(x^{-2}; q^2\right)}
     \left(\vect{-q^2 x^6}{0} + \vect{-q^2 x^6}{2}\right).
\end{aligned}
\end{equation}
While \cref{eq:cross-product-2} can be proven easily using the results in this section and \cref{lem:m-identities}, \cref{eq:cross-product-1} appears to be more difficult. By specialization at $x = 1$, \cref{eq:cross-product-1} would imply an additional identity for the generalized eta functions of level $N = 7$ (using \cref{eq:generalized-eta}), which would belong in \cref{prop:level-7}:
\[
    \sum_{g = 1}^3  \E_g(5\tau) \E_{3g}(\tau) = 
    -
    \frac{\eta(\tau)\eta(5\tau)}{\eta(7\tau)\eta(35\tau)}.
\]
\end{remark}

\subsection{W-coefficients, second proofs of main theorems and change-of-basis identities} \label{subsec:w-coeffs}
A different approach to proving \cref{thm:bases-proportional} and \cref{thm:bases-proportional-2} proceeds by expanding the twisted sums in the left-hand sides of \cref{thm:twisted-rog-ram} and \cref{thm:bases-proportional-2}, using the lemma below. We mention this approach since it will also lead to the anticipated change-of-basis identities (such as \cref{prop:2x2-determinant}).

%\fix{MAYBE USE ANY POWER OF THE TWIST BY $q$, BUT NO $m$ (leave $m$ as remark...)?? Give analogues for the twists by just $(xyq^{2n+\cdots}; q^2)$ leading to \cref{thm:twisted-rog-ram-var} by $y = -x^{-1}$ (some things should be zero at $a = 2$...)!! (mention the doubly-twisted identities then...; but those are much harder!)}
%Maybe introduce q-binomial coeff. (do you ever need it? maybe just introduce it then...), etc. 

\begin{lemma}[Twisted sum identities] \label{lem:t-identities}
Let $0 \neq \alpha \in \C$, $s \in \{\pm 1\}$ and $d, k, a \in \Z$ with $0 \le k \le d-1$ and $1 \le a \mid d$. If $sa = d$ assume that $|y| < \left\vert \alpha q^k \right\vert$, and if $sa = -d$ assume that $|y| < \left\vert q^{d-k}/\alpha\right\vert$. Then
\begin{equation}\label{eq:twisted-series}
    \underbrace{\sum_{n \in \Z} \alpha^n q^{d\binom{n}{2}+kn} x^{dn+k}}_{\text{Original series $\vect{\alpha x^d}{k}$}} \underbrace{\vphantom{\sum_{n \in \Z}} \left(q^{san}x^{sa} y; q^{a}\right)}_{\text{Twist factor}}
    =
    \sum_{\ell=0}^{d-1} W_{k,\ell}(y, q)
    \vect{\alpha x^d}{\ell},
\end{equation}
where $\{W_{k,\ell}(y, q)\}_\ell$ also depend on $\alpha, s, d, a$ (but not on $x$), and are given by
\begin{equation}\label{eq:w-coeffs-value}
    W_{k,\ell}(y, q)
    = 
    \one_{\ell \equiv k\ (\textnormal{mod }a)} \cdot
    q^{a \binom{j}{2}}
    \sum_{n \ge \one_{s\ell < sk}}
    \alpha^{-sn}\ \frac{q^{a \binom{b}{2} n^2 + s(b(\ell-k)-\ell)n + \frac{s-1}{2}dn}}{\left(q^{a}; q^{a}\right)_{bn+j}} (-y)^{bn+j},
\end{equation}
where $b := d/a$ and $j := (\ell - k)/(sa)$.
\end{lemma}
%\begin{remark}
%One can add yet another parameter $m \ge 1$ to the formulae above, by twisting by the finite product $\left(q^{an}x^a y; q^a\right)_m$ instead; the associated coefficients $W_{k,\ell}(y, q)$ will then involve $q$-binomial coefficients.
%\end{remark}

\begin{remark}
The condition $\ell \equiv k \pmod{a}$ in \cref{eq:w-coeffs-value} indicates the extent of the non-mixing phenomenon for Fourier coefficients described in the end of \cref{subsec:two-var-generaliz}; there is no mixing when $a = d$, maximum mixing when $a = 1$, and some mixing when $1 < a < d$. Also, the general case of \cref{lem:t-identities} can be recovered from the special case $a = 1$, using the substitutions $d \mapsto b$, $q \mapsto q^a$ and adequate substitutions for $\alpha$; however, it is more convenient to state and prove \cref{lem:t-identities} for a general $a \mid d$.
\end{remark}

\begin{proof}
Let $L(z)$ denote the left-hand side of \cref{eq:twisted-series}, where $x = e^{2\pi i z}$. One can check as in the proof of \cref{prop:exact-formulae-kind1} that $L$ converges absolutely and locally uniformly to a well-defined entire function under the given assumptions. It is then immediate that $L \in \mT_\C\left(\alpha x^d \right)$, since it is a twisted version of $\vect{\alpha x^d}{k}$ as in \cref{eq:w-coeffs}; hence $L$ has an expansion as in the right-hand side of \cref{eq:twisted-series}, where $W_{k,\ell}(y, q) = \hat{L}(\ell)$. Expanding the series of $L(z)$ using \cref{eq:basic1} and identifying the coefficients of $x^\ell$ leads to the formula in \cref{eq:w-coeffs-value}, after a short computation (treating the cases $s \in \{\pm 1\}$ separately).
\end{proof}

We further give a couple of consequences of \cref{lem:t-identities}, related to our previous results.
\begin{corollary} \label{cor:d=1} For $y \in \C$ with $|y| < |x|$ and $y \not\in \left\{q^n : n \le 0\right\}$, one has
\[
    \sum_{n \in \Z}
    \frac{q^{\binom{n}{2}}}{\left(y; q \right)_n} x^n 
    =
    \frac{(q; q)\la -x; q \ra}{(y; q)(-y/x; q)},
\]
generalizing \cref{eq:triple}, \cref{eq:basic1} and \cref{thm:twisted-rog-ram-var}. In fact, this is a particular case of Ramanujan's $_1\psi_1$ summation from \cref{prop:ramanujan-1psi1} (by taking $x \mapsto -x/a$, $a \to \infty$ and $b = y$).
\end{corollary}
\begin{proof}
\cref{lem:t-identities} yields $\sum_{n \in \Z} q^{\binom{n}{2}}x^n \left(q^n xy; q\right) = \sum_{n \ge 0} \frac{(-y)^n}{(q; q)_n} \cdot \vect{x}{0}$ for $|y| < 1$; take $y \mapsto y/x$, divide by $(y; q)$, and use \cref{eq:triple},\cref{eq:basic2} to finish. An alternative proof proceeds by noting that both sides lie in the one-dimensional space $\mT_D\left(x - y\right)$ (as functions of $z$), for a suitable half-plane $D$.
\end{proof}

\begin{corollary} \label{cor:3-var-2x2-matrix}
Let $y \in \C$ with $|y| \le 1$. Then as an equality of pairs of functions in $\mT_\C\left(qx^2\right)$,
\[
    \begin{aligned}
    \left(x^{-1} y; q\right) 
    \begin{pmatrix}
    \vect{qx^2 - yx}{0} \vspace{0.1cm} \\ 
    \vect{qx^2 - yx}{1} 
    \end{pmatrix}
    &=
    \begin{pmatrix}
    \sum_{n \in \Z}
        q^{n^2} x^{2n} \left(q^{n+1}xy; q\right)
    \vspace{0.1cm} \\
    \sum_{n \in \Z}
        q^{n^2+n} x^{2n+1} \left(q^{n+1}xy; q\right)
    \end{pmatrix}
    \\
    &=
    \begin{pmatrix}
    \sum_{n \ge 0} \frac{q^{n(n+1)}}{(q; q)_{2n}} y^{2n}
    & 
    - \sum_{n \ge 0} \frac{q^{(n+1)^2}}{(q; q)_{2n+1}} y^{2n+1}
    \vspace{0.1cm} \\
    -\sum_{n \ge 0} \frac{q^{n(n+1)}}{(q; q)_{2n+1}} y^{2n+1}
    &
    \sum_{n \ge 0} \frac{q^{n^2}}{(q; q)_{2n}} y^{2n}
    \end{pmatrix}
    \begin{pmatrix}
    \vect{qx^2}{0} \vspace{0.1cm} \\
    \vect{qx^2}{1} 
    \end{pmatrix}.
    \end{aligned}
\]
\end{corollary}

\begin{proof}
For the first equality, use \cref{cor:kind2} and the symmetry at $x \mapsto x^{-1}$ (due to $\mT_\C\left(qx^2\right) \subset \mS_\C(1)$). For the second equality, use \cref{lem:t-identities} for $a = 1$, $d = 2$, $s = 1$ and $\alpha = y = q$.
\end{proof}

\begin{remark}
Concerning the second equality in \cref{cor:3-var-2x2-matrix}, one can deduce the bottom-row identity from the top-row one by taking $x \mapsto x/\sqrt{q}$ and $y \mapsto y\sqrt{q}$; however, the matrix appearing in the formulation above will be relevant later (it turns out to have determinant $1$; see \cref{cor:3-var-2x2-matrix-ctd}).
\end{remark}

\begin{corollary}[Rogers-type sums] \label{cor:rogers-type-sums}
For $\tau \in \H^+$ and $z \in \C$, one has
\[
(q; q)
\sum_{n \ge 0} \frac{q^{n^2}}{(q; q)_n} x^{-2n}
= 
\sum_{n \ge 0} \frac{q^{n(n+1)}}{(q; q)_{2n}} x^{2n} \cdot \vect{qx^2}{0}
- \sum_{n \ge 0} \frac{q^{(n+1)^2}}{(q; q)_{2n+1}} x^{2n+1} \cdot \vect{qx^2}{1},
\]
where the left-hand side is (up to a trivial substitution) the sum from \cref{prop:rogers}.
\end{corollary}

\begin{proof}
Take $y = x^{-1}$ in \cref{cor:3-var-2x2-matrix} and identify the top row (alternatively, take $x \mapsto x/\sqrt{q}$, then $y = x^{-1}\sqrt{q}$ and identify the bottom row). Afterwards, substitute $x \mapsto x^{-1}$ on both sides, noting that $\mT_\C\left(qx^2\right) \subset \mS_\C(1)$.
\end{proof}

\begin{remark}
More identities of Rogers--Ramanujan type sums follow by plugging in $x \in \{-i, -i/\sqrt{q}\}$ in \cref{cor:rogers-type-sums} and using \cref{prop:rogers}.
\iffalse 
\begin{corollary} \fix{...} By plugging in $x = -i$, respectively $x = -i/\sqrt{q}$ in Corollary 5.9, one obtains
\[
\begin{aligned}
\frac{\left(q^3; q^2\right)}{(q; q)} \sum_{n \ge 0} (-1)^n\frac{q^{n(n+1)}}{(q; q)_{2n}}
\ &=\ \sum_{n \ge 0} (-1)^n\frac{q^{n^2}}{(q; q)_n}
\ \ =\
\sum_{n \geq 0}(-1)^n \frac{q^{\frac{1}{2}(5n^2 - n)}}{(q; q)_n \left(-q^n; q\right)} \left(1 + q^{2n}\right)
, \\
\frac{\left(q^3; q^2\right)}{(q; q)} \sum_{n \ge 0} (-1)^n\frac{q^{n(n+1)}}{(q; q)_{2n+1}}
\ &=\ \sum_{n \ge 0} (-1)^n\frac{q^{n(n+1)}}{(q; q)_n}
\ =\
\sum_{n \geq 0}(-1)^n \frac{q^{\frac{1}{2}(5n^2 + 3n)}}{(q; q)_n \left(-q^{n+1}; q\right)} \left(1 + q^{2n+1}\right).
\end{aligned}
\]
\end{corollary}
\fi
\end{remark}

We are now ready to give second proofs for \cref{thm:twisted-rog-ram,thm:twisted-rog-ram-var}:

\begin{proof}[Second proof of \cref{thm:twisted-rog-ram}]
Taking $y = 1$ in \cref{cor:3-var-2x2-matrix}, the statement of \cref{thm:twisted-rog-ram} reduces to an equality of matrices
\begin{equation}\label{eq:slater}
    \begin{pmatrix}
    \sum_{n \ge 0} \frac{q^{n(n+1)}}{(q; q)_{2n}}
    & 
    - \sum_{n \ge 0} \frac{q^{(n+1)^2}}{(q; q)_{2n+1}}
    \vspace{0.1cm} \\
    -\sum_{n \ge 0} \frac{q^{n(n+1)}}{(q; q)_{2n+1}}
    &
    \sum_{n \ge 0} \frac{q^{n^2}}{(q; q)_{2n}}
    \end{pmatrix}
    \stackrel{?}{=}
    \begin{pmatrix}
    A(q) & B(q) \\
    C(q) & D(q)
    \end{pmatrix},
\end{equation}
where $A(q), B(q), C(q), D(q)$ are as in \cref{eq:explicit-matrix}. We note that \cref{eq:slater} is precisely the content of four Rogers--Ramanujan type identities of Slater \cite[(94),(96),(98),(99)]{slater1952further}, and we can prove them using \cref{prop:rogers}. Indeed, taking $q \mapsto q^4$ and then multiplying by $\begin{psmall}
    -q & q \\
    1 & 1
\end{psmall}$, one easily reduces \cref{eq:explicit-matrix} to
\[
    \begin{pmatrix}
    -\sum_{n \ge 0} \frac{q^{(n+1)^2}}{\left(q^4; q^4\right)_n}
    & 
    -\sum_{n \ge 0} \frac{q^{(n+1)^2}}{\left(q^4; q^4\right)_n}
    \vspace{0.1cm} \\
    \sum_{n \ge 0} \frac{q^{n^2}}{\left(q^4; q^4\right)_n}
    &
    \sum_{n \ge 0} \frac{q^{n^2}(-1)^n}{\left(q^4; q^4\right)_n}
    \end{pmatrix}
    \stackrel{?}{=}
    \begin{pmatrix}
    A\left(q^4\right) & B\left(q^4\right) \\
    C\left(q^4\right) & D\left(q^4\right)
    \end{pmatrix}
    \begin{pmatrix}
    -q & q \\
    1 & 1
    \end{pmatrix},
\]
which, in light of the symmetry $q \leftrightarrow -q$, boils down to
\[
    \sum_{n \ge 0} \frac{q^{n^2}}{\left(q^4; q^4\right)_n}
    \stackrel{?}{=}
    D\left(q^4\right) - qC\left(q^4\right)
    ,\qquad\qquad
    \sum_{n \ge 0} \frac{q^{n^2+2n}}{\left(q^4; q^4\right)_n}
    \stackrel{?}{=}
    A\left(q^4\right) - q^{-1} B\left(q^4\right).
\]
Now using \cref{eq:explicit-matrix} and the quintuple identity \cref{eq:substitutions-quintuple} for $a = 10$, one can expand $(q; q)A(q)$, $(q; q)B(q)$, $(q; q)C(q)$ and $(q; q)D(q)$ as explicit series where the $n$th term has a dominant power of $q^{30\binom{n}{2}}$. By taking $q \mapsto q^4$, adding these series and then using \cref{eq:substitutions-quintuple} for $a = 10$ again, one obtains
\[
\begin{aligned}
    \left(q^4; q^4\right) \left(D\left(q^4\right) - qC\left(q^4\right)\right) 
    &\ =\ \sum_{n \in \Z}
    (-1)^n q^{30\binom{n}{2}}\left(q^{13n} + q^{7n+1}\right)
    \ =\ 
    \left(q^{10}; q^{10}\right)\la -q; q^{10}\ra \la q^{12}; q^{20}\ra 
    , 
    \\ 
    \left(q^4; q^4\right) \left(A\left(q^4\right) - q^{-1}B\left(q^4\right)\right)
    &\ =\
    \sum_{n \in \Z}
    (-1)^n q^{30\binom{n}{2}}\left(q^{11n} + q^{n+3}\right)
    \ \ =\ 
    \left(q^{10}; q^{10}\right)\la -q^3; q^{10}\ra \la q^{16}; q^{20}\ra
    .
\end{aligned}
\]
After rearranging the products on the right, it remains to show that
\begin{equation} \label{eq:claim}
    \sum_{n \ge 0} \frac{q^{n^2}}{\left(q^4; q^4\right)_n} 
    \stackrel{?}{=}
    \frac{1}{\left(-q^2; q^2\right) \la q; q^5 \ra}
    ,\qquad\qquad
    \sum_{n \ge 0} \frac{q^{n^2+2n}}{\left(q^4; q^4\right)_n}
    \stackrel{?}{=}
    \frac{1}{\left(-q^2; q^2\right) \la q^2; q^5 \ra}.
\end{equation}
But by taking $x = 1$, respectively $x = q$ in the second equality from \cref{prop:rogers} (and using the Rogers--Ramanujan identities \cref{eq:rog-ram}), one obtains that
\begin{equation}\label{eq:known}
    \frac{1}{\la q; q^5 \ra} = \sum_{n \ge 0} \frac{q^{n^2}}{\left(q^2; q^2\right)_n} \left(-q^{2n+2}; q^2\right),
    \qquad\qquad
    \frac{1}{\la q^2; q^5 \ra} = \sum_{n \ge 0} \frac{q^{n^2+n}}{\left(q^2; q^2\right)_n} \left(-q^{2n+3}; q^2\right).
\end{equation}
The first equality in \cref{eq:known} settles the first equality in \cref{eq:claim} since $\left(q^4; q^4\right)_n = \left(q^2; q^2\right)_n \left(-q^2; q^2\right)_n$. 

For the second equality in \cref{eq:claim}, write $F(z) := \sum_{n \ge 0} \frac{q^{n^2}}{\left(q^2; q^2\right)_n}x^n \left(-q^{2n+2}x; q^2\right)$, and recall that \cref{prop:rogers} implies that $F \in \mT_\C(1, qx)$. In particular we have $F(\tau) = F(0) - q F(2\tau)$, and thus
\[
\begin{aligned}
    \sum_{n \ge 0} \frac{q^{n^2+n}}{\left(q^2; q^2\right)_n} \left(-q^{2n+3}; q^2\right)
    &=
    \sum_{n \ge 0} \frac{q^{n^2}}{\left(q^2; q^2\right)_n} \left(-q^{2n+2}; q^2\right)
    -
    \sum_{n \ge 0} \frac{q^{(n+1)^2}}{\left(q^2; q^2\right)_n} \left(-q^{2(n+1)+2}; q^2\right)
    \\
    &=
    \sum_{n \ge 0} \frac{q^{n^2}}{\left(q^2; q^2\right)_n} \left(-q^{2n+2}; q^2\right)
    -
    \sum_{n \ge 0} \frac{q^{n^2}\left(1-q^{2n}\right)}{\left(q^2; q^2\right)_n} \left(-q^{2n+2}; q^2\right)
    \\
    &= 
    \sum_{n \ge 0}
    \frac{q^{n^2+2n}}{\left(q^2; q^2\right)_n} \left(-q^{2n+2}; q^2\right)
    \qquad
    =
    \left(-q^2; q^2\right) \sum_{n \ge 0}
    \frac{q^{n^2+2n}}{\left(q^4; q^4\right)_n}
    ,
\end{aligned}
\]
which settles the second equality in \cref{eq:claim} in light of \cref{eq:known}. This completes our proof.
\end{proof}
\begin{remark}
As indicated in \cref{fig:layout}, the proof above also shows that \cref{thm:twisted-rog-ram} implies Slater's identities from \cref{eq:slater}. So while \cref{thm:twisted-rog-ram} generalizes the classical Rogers--Ramanujan identities from \cref{eq:rog-ram} (and Watson's identities from \cref{eq:watson}) by value identification, it also generalizes \cref{eq:slater} by Fourier identification.
\end{remark}

Similarly, we give another proof of \cref{thm:twisted-rog-ram-var} based on \cref{lem:t-identities}; in fact, an analogous argument applies to the more general second equality from \cref{prop:exact-formulae-kind1}.
\begin{proof}[Second proof of \cref{thm:twisted-rog-ram-var}]
Taking $a = d = 2$, $\alpha = q$, $s = 1$ and $y \in \left\{-qx^{-2}, -q^2x^{-2}\right\}$ in \cref{lem:t-identities}, one obtains 
\[
    \begin{pmatrix}
    \sum_{n \in \Z} q^{n^2} x^{2n} \left(-q^{2n+1}; q^2\right)
    \vspace{0.1cm} \\
    %\vect{qx^2-q^{-1}}{0} \\
    \sum_{n \in \Z} q^{n^2+n} x^{2n+1} \left(-q^{2n+2}; q^2\right)
    %\vect{qx^2-q^{-1}}{1}
    \end{pmatrix}
    =
    \begin{pmatrix}
    \sum_{n \ge 0} \frac{q^n}{\left(q^2; q^2\right)_n} x^{-2n} 
    & 0
    \vspace{0.1cm} \\
    0 & 
    \sum_{n \ge 0} \frac{q^n}{\left(q^2; q^2\right)_n} x^{-2n}
    \end{pmatrix}
    \begin{pmatrix}
    \vect{qx^2}{0} 
    \vspace{0.1cm} \\
    \vect{qx^2}{1}
    \end{pmatrix}.
\]
But \cref{eq:basic2} implies that $\sum_{n \ge 0} \frac{q^n}{\left(q^2; q^2\right)_n} x^{-2n} = \left(x^{-2}; q^2\right)^{-1}$, completing our proof.
\end{proof}

%\fix{Add the explicit description of the matrix from ... too! see connection to (gener. of) Bill Gosper's identity of 2x2 matrices in the comments of:} %http://oeis.org/search?q=1+++1+++3+++3+++6+++7++12++14&sort=&language=english&go=Search
We now leave \cref{thm:bases-proportional,thm:bases-proportional-2} behind and look into more general applications of our work. Note that for $s = -1$, \cref{lem:t-identities} can be rephrased in terms of column vectors as
\begin{equation} \label{eq:twist-change-canonical}
    \underbrace{\vphantom{\sum_n} 
    \begin{pmatrix}
    \sum_{n \in \Z} \alpha^n q^{d\binom{n}{2}+kn} x^{dn+k} \left(q^{-an} x^{-a} y; q^a\right)
    \end{pmatrix}_{0 \le k < d}
    }_{\text{Twisted canonical basis of $\mT_\C\left(\alpha x^d\right)$}} 
    =
    \underbrace{\vphantom{\sum_n} 
    \begin{pmatrix}
    W_{k,\ell}(y, q)
    \end{pmatrix}_{0 \le k, \ell < d}
    }_{\text{Change-of-basis matrix}}
    \underbrace{\vphantom{\sum_n}
    \begin{pmatrix}
    \vect{\alpha x^d}{\ell}
    \end{pmatrix}_{0 \le l < d}
    }_{\text{Canonical basis}},
\end{equation}
for $\alpha, y, d, a$ and $W_{k,\ell}$ as before. As the underbraces above suggest, it is natural to ask if twisting the canonical basis of $\mT_\C(\alpha x^d)$ on the right results in another basis on the left, i.e.\ if the matrix $\left(
    W_{k,\ell}(y, q)
    \right)_{0 \le k, \ell < d}$ 
is invertible. The answer is yes: we can explicitly compute the determinant of this matrix, and this leads to a large family of two-variable identities generalizing \cref{prop:2x2-determinant}, given in \cref{thm:change-of-basis}. A helpful idea will be to consider a third basis of $\mT_\C\left(\alpha x^d\right)$, which is proportional to the canonical basis of $\mT_\Hminus\left(\alpha x^d\left(1 - (qx)^{-a}y\right)\right)$ (using the bijection in \cref{eq:bijection}): 
\[
    \left(\left(x^{-a} y; q^a\right) \vect{\alpha x^d\left(1 - (qx)^{-a}y\right)}{k} \right)_{0 \le k < d} 
    \ \ \in\ 
    \mT_\C\big(\alpha x^d\big)^d.
\]
We first recall that $\left(W_{k,\ell}(y, q)\right)_{0 \le k, \ell < d}$ only contains nonzero entries when $k \equiv \ell \pmod{a}$ by \cref{eq:w-coeffs-value}, and that identities of such $d\times d$ matrices split into $a$ identities of $b \times b$ matrices where $b = d/a$ (similarly, such $d \times d$ determinants split as a product of $b \times b$ determinants). It suffices, thus, to treat the case $a = 1$, and all of our further work in this subsection generalizes to the case $a > 1$ by intercalating $b$ suitable matrices. Fixing $a = 1$, our main trick is to interpolate an infinite sequence of bases of $\mT_\C\left(\alpha x^d\right)$ between the twisted canonical basis from \cref{eq:twist-change-canonical} and the untwisted one:

\begin{lemma}[$y$-Interpolation] \label{lem:y-interpolation}
Let $V(x, y) := \big(
    \sum_{n \in \Z} \alpha^n q^{d\binom{n}{2}+kn} x^{dn+k} \left(q^{-n} x^{-1} y; q\big)
    \right)_{0 \le k < d}$ 
denote the column vector in the left-hand side of \cref{eq:twist-change-canonical}, for $1 = a \le d$ and $\alpha, y$ as in \cref{lem:t-identities}. Then
\[
    V(x, y)
    =
    M(y)
    \cdot
    V(x, qy),
\]
where $M(y) = \left(1 + \alpha y q^{-1}\right)^{-1}$ if $d = 1$, $M(y) = \begin{psmall}
    1 + \alpha y^2q^{-1} & -\alpha yq^{-1} \\
    -y & 1
    \end{psmall}$
if $d = 2$, and
\[
    M(y)
    =
    \begin{psmall}
    1 & 0 & 0 & \cdots & 0 & \alpha y^2q^{-1} & -\alpha yq^{-1} \\
    -y & 1 & 0 & \cdots & 0 & 0 & 0 \\
    0 & -y & 1 & \cdots & 0 & 0 & 0 \\
    & & & & \ddots & & \\
    0 & 0 & 0 & \cdots & 0 & -y & 1
    \end{psmall}
    \in 
    \C^{d \times d},
    \qquad 
    \text{if } d \ge 3.
\]
So while $V(x, y) \in \mT_\C\left(\alpha x^d\right)^d$, by swapping variables we have $V(y, x) \in \mT_\C\left(M(x)\right)$ for $d \ge 2$.
\end{lemma}
\begin{proof}
Letting $V_k(x, y)$ be the $k$th coordinate of $V(x, y)$, use the functional equation $\left(q^{-n} x^{-1} y; q\right) = \left(1 - q^{-n} x^{-1} y\right) \left(q^{-n} x^{-1} (qy); q\right)$ to obtain
\[
    V_k(x, y) = V_k(x, qy) - yV_{k-1}(x, qy),
\]
for $1 \le k < d$. For $k = 0$, the analogous expansion yields $V_0(x, y) = V_0(x, qy) - \alpha y q^{-1} V_{d-1}(x, y)$; if $d = 1$ this gives $V_0(x, y) = (1 + \alpha y q^{-1})^{-1} V_0(x, qy)$, and if $d \ge 2$ we get $V_0(x, y) = V_0(x, qy) + \alpha y^2 q^{-1} V_{d-2}(qy) - \alpha y q^{-1} V_{d-1}(qy)$ by applying the relation above for $k-1 = d$. Putting these together produces the matrices $M(y)$ above.
\end{proof}
\begin{remark}
Working in $\mT_\C\left(M(x)\right)$, our observation in \cref{ex:two-var-rog-ram} suggests iterating the functional equation in \cref{lem:y-interpolation} to obtain an infinite product of matrices; this leads to the following theorem.
\end{remark}

\begin{theorem}[Change-of-basis identities] \label{thm:change-of-basis}
Let $\alpha, y \in \C$ with $\alpha \neq 0$, and $d \in \Z$ with $d \ge 2$. Define the $d \times d$ matrices $U = (U_{k, \ell})_{0 \le k, \ell < d}$ and $W = (W_{k, \ell})_{0 \le k, \ell < d}$ by
\[
    U_{k, \ell} := \sum_{n \ge \one_{\ell > k}} \alpha^n \frac{q^{d\binom{n}{2} + kn}}{(q; q)_{dn+k-\ell}} y^{dn+k-\ell}, 
    \qquad 
    W_{k, \ell} := \sum_{n \ge \one_{\ell > k}} \alpha^n \frac{q^{\binom{d}{2}n^2 + (d(k-\ell-1)+\ell)n + \binom{k - \ell}{2}}}{(q; q)_{dn+k-\ell}} (-y)^{dn+k-\ell}.
\]
Then one has $\det U = \det W = 1$. Moreover,
\[
    W U = 
    \begin{psmall} 
    1 & 0 & \cdots & 0 & -\alpha y q^{-1}
    \\
    0 & 1 & \cdots & 0 & 0 
    \\
    & & & \ddots & \\
    0 & 0 & \cdots & 0 & 1
    \end{psmall}
    \qquad 
    \text{and}
    \qquad 
    W = \prod_{n \ge 0} M\left(q^n y\right),
\]
where $M$ is as in \cref{lem:y-interpolation} and the product of matrices is evaluated from left to right.
%Before, had: B A = \left(y^{k-\ell} \one_{k \ge l} \right)_{0 \le k, \ell < d}
\end{theorem}
\begin{remark}
One can compute $U^{-1}$ and $W^{-1}$ explicitly based on \cref{thm:change-of-basis}, and thus all minors of $U$ and $W$ as well. Also, the case $d = 1$ is not included in \cref{thm:change-of-basis} since it is essentially equivalent to \cref{cor:d=1}, and since then we have $\det U = (\det W)^{-1} = \left(-\alpha y q^{-1}; q\right)$ (for  $|\alpha y| < |q|$).
\end{remark}
%\fix{The infinite product of matrices is a particular case of an identity you found on stack-exchange...?????}
\begin{proof}
Assume without loss of generality that $|y| \le 1$, by the uniqueness of analytic continuation; then we have $1 - q^{-1} x^{-1} y \neq 0$ in $\{\Im z < -\Im \tau\}$, and thus the canonical basis vectors of $\mT_\Hminus \left(\alpha x^d \left(1 - q^{-1} x^{-1} y\right) \right)$ are well-defined by \cref{prop:canonical-basis-vectors}. Now consider the following three elements of $\mT_\C\left(\alpha x^d\right)^d$:
\begin{center}
\small
\begin{tikzcd} [column sep = -0.5cm]
    & 
    \begin{pmatrix} 
    \vect{\alpha x^d}{k}
    \end{pmatrix}_{0 \le k < d}
    \arrow{rd}{=\ U \ \cdot}
    \\
    \begin{pmatrix}
    \sum_{n \in \Z} \alpha^n q^{d\binom{n}{2}+kn} x^{dn+k} \left(q^{-n} x^{-1} y; q\right)
    \end{pmatrix}_{0 \le k < d}
    \arrow{rr}{=\ WU\ \cdot}
    \arrow{ru}{=\ W\ \cdot}
    &&
    \begin{pmatrix} 
    \left(x^{-1}y; q\right)\vect{\alpha x^d\left(1 - q^{-1}x^{-1}y\right)}{k}
    \end{pmatrix}_{0 \le k < d}
\end{tikzcd}
\end{center}
where the arrows indicate claimed relationships. The ``$= W \cdot\ $'' arrow is precisely the relation from \cref{eq:twist-change-canonical} for $a = 1$. The ``$= U \cdot\ $'' relation follows from \cref{eq:basic2} by expanding
\[
    \frac{1}{\left(x^{-1}y; q\right)} \vect{\alpha x^d}{k} 
    =
    \sum_{m \ge 0} \frac{\left(-x^{-1}y\right)^m}{(q; q)_m} \sum_{n \in \Z} \alpha^n q^{d\binom{n}{2}+kn} x^{dn+k},
\]
and identifying Fourier coefficients in the canonical basis of $\mT_\Hminus \left(\alpha x^d \left(1 - q^{-1} x^{-1} y\right) \right)$. The ``$= WU\ \cdot$'' relation follows from the previous two, and will be helpful in computing $WU$.

Next, the fact that $W = \prod_{n \ge 0} M\left(q^n y\right)$ follows immediately by iterating \cref{lem:y-interpolation}, using that
\[
    \lim_{p \to \infty} \sum_{n \in \Z} \alpha^n q^{d\binom{n}{2}+kn} x^{dn+k} \left(q^{-n} x^{-1} (q^p y); q\right)
    =
    \vect{\alpha x^d}{k}.
\]
In particular, this shows that $\det W = 1$ (since each $\det M\left(q^n y\right) = 1$), and in particular that $\left(\vect{\alpha x^d}{k}\right)_{0 \le k < d}$ is invertible is indeed a basis of $\mT_\C\left(\alpha x^d\right)$ as well. Hence we can compute $(WU)^{-1}$ as the unique change-of-basis matrix satisfying
\[
    \begin{pmatrix} 
    \left(x^{-1}y; q\right)\vect{\alpha x^d\left(1 - q^{-1}x^{-1}y\right)}{k}
    \end{pmatrix}_{0 \le k < d}
    =
    (WU)^{-1} 
    \begin{pmatrix}
    \sum_{n \in \Z} \alpha^n q^{d\binom{n}{2}+kn} x^{dn+k} \left(q^{-n} x^{-1} y; q\right)
    \end{pmatrix}_{0 \le k < d}.
\]
But by \cref{prop:exact-formulae-kind2} for $g(-z) = 1 - q^{-1}x^{-1}y$, we have
\[
    \left(x^{-1}y; q\right)\vect{\alpha x^d\left(1 - q^{-1}x^{-1}y\right)}{k}
    =
    \begin{cases}
    \sum_{n \in \Z} q^{d\binom{n}{2}} x^{dn} \left(q^{-n+1}x^{-1}y; q\right), 
    & k = 0,
    \\
    \sum_{n \in \Z} q^{d\binom{n}{2}+kn} x^{dn+k} \left(q^{-n}x^{-1}y; q\right),
    & k > 0.
    \end{cases}
\]
Using the relationship $V_0(x, qy) = V_0(x, y) + \alpha y q^{-1}$ from the proof of \cref{lem:y-interpolation}, this proves that
\[
    (WU)^{-1} = \begin{psmall} 
    1 & 0 & \cdots & 0 & \alpha y q^{-1}
    \\
    0 & 1 & \cdots & 0 & 0 
    \\
    & & & \ddots & \\
    0 & 0 & \cdots & 0 & 1
    \end{psmall}
    \qquad \Rightarrow 
    \qquad 
    WU = \begin{psmall} 
    1 & 0 & \cdots & 0 & -\alpha y q^{-1}
    \\
    0 & 1 & \cdots & 0 & 0 
    \\
    & & & \ddots & \\
    0 & 0 & \cdots & 0 & 1
    \end{psmall}.
\]
Finally, from $\det W = 1$ and $\det (WU) = 1$ we also deduce that $\det U = 1$.
\end{proof}

\begin{corollary} \label{cor:3-var-2x2-matrix-ctd}
The $2\times 2$ matrix from \cref{cor:3-var-2x2-matrix} has constant determinant $1$. In fact, it can be written as an infinite product (evaluated left-to-right)
\[
\begin{pmatrix}
\sum_{n \ge 0} \frac{q^{n(n+1)}}{(q; q)_{2n}} y^{2n}
& 
-\sum_{n \ge 0} \frac{q^{(n+1)^2}}{(q; q)_{2n+1}} y^{2n+1}
\vspace{0.3cm} 
\\
-\sum_{n \ge 0} \frac{q^{n(n+1)}}{(q; q)_{2n+1}} y^{2n+1}
&
\sum_{n \ge 0} \frac{q^{n^2}}{(q; q)_{2n}} y^{2n}
\end{pmatrix}
=
\prod_{n \ge 0}
\begin{pmatrix}
1 & 0
\\
-q^n y  & 1
\end{pmatrix}
\begin{pmatrix}
1 & -q^{n+1} y
\\
0 & 1
\end{pmatrix}.
\]
In particular, this establishes \cref{prop:2x2-determinant} by taking $y = -x$.
\end{corollary}
\begin{proof}
Taking $d = 2$ and $\alpha = q$ in \cref{thm:change-of-basis}, the matrix in \cref{cor:3-var-2x2-matrix} is precisely $U^{-1}$. But
\[
\begin{aligned}
    U^{-1} = \begin{pmatrix}
    1 & y \\
    0 & 1
    \end{pmatrix}
    W
    &=
    \begin{pmatrix}
    1 & y \\
    0 & 1
    \end{pmatrix}
    \prod_{n \ge 0} M\left(q^n y\right)
    \\
    &=
    \begin{pmatrix}
    1 & y \\
    0 & 1
    \end{pmatrix}
    \prod_{n \ge 0}
    \begin{pmatrix}
    1 + q^{2n}y^2 & 1 - q^n y \\
    -q^n y & 1
    \end{pmatrix}
    =
    \begin{pmatrix}
    1 & y \\
    0 & 1
    \end{pmatrix}
    \prod_{n \ge 0}
    \begin{pmatrix}
    1 & -q^n y \\
    0 & 1
    \end{pmatrix}
    \begin{pmatrix}
    1 & 0 \\
    -q^n y & 1
    \end{pmatrix},
\end{aligned}
\]
which completes our proof by noting that $\begin{psmall} 1 & y \\ 0 & 1 \end{psmall} \begin{psmall} 1 & -y \\ 0 & 1 \end{psmall} = 1$.
\end{proof}

\begin{remark}
Denoting
\[
    F(z) := 
    \sum_{n \ge 0} \frac{q^{n^2}}{\left(q^4; q^4\right)_n} x^{2n}
    \qquad
    \Rightarrow 
    \qquad 
    \frac{1}{2}
    \begin{pmatrix}
    F(z) + F(z+1/4) \vspace{0.1cm} \\
    F(z) - F(z+1/4)
    \end{pmatrix}
    =
    \begin{pmatrix}
    \sum_{n \ge 0} \frac{q^{4n^2}}{\left(q^4; q^4\right)_{2n}} x^{4n}
    \vspace{0.1cm} \\
    q\sum_{n \ge 0} \frac{q^{4(n^2+n)}}{\left(q^4; q^4\right)_{2n+1}} x^{4n+2}
    \end{pmatrix},
\]
the fact that the determinant in \cref{cor:3-var-2x2-matrix-ctd} is $1$ can be rephrased as (taking $q \mapsto q^4$)
\[
    \frac{F(z) + F\left(z + \frac{1}{4}\right)}{2}
    \frac{F\left(z + \tau\right) + F\left(z + \tau + \frac{1}{4}\right)}{2}
    -
    \frac{F(z) - F\left(z + \frac{1}{4}\right)}{2}
    \frac{F\left(z + \tau\right) - F\left(z + \tau + \frac{1}{4}\right)}{2}
    = 
    1,
\]
which further simplifies to $F(z)F(z + \tau + 1/4) + F(z + 1/4)F(z + \tau) = 2$. This follows in turn from the fact that $F \in \mT_\C\left(qx^2, 1\right)$ (which is easily checked); indeed, a short computation then shows that $F(z)F(z + \tau + 1/4) + F(z + 1/4)F(z + \tau) \in \mT_\C(1)$, so it must be a constant.
\end{remark}

\begin{question} \label{qtn:rogers-type-sums}
Note that taking $q \mapsto q^4$ in \cref{cor:3-var-2x2-matrix} connects Rogers' $\sum_{n \ge 0} \frac{q^{n^2}}{(q; q)_n} x^n$ from \cref{prop:rogers} to our function $F(z)$ from the previous remark. Moreover, the relations immediately preceding \cref{eq:claim} yield that %D\left(q^4\right) - qC\left(q^4\right), A\left(q^4\right) - q^{-1}B\left(q^4\right)
\[
    F(0) = \frac{1}{\left(q^4; q^4\right)} \sum_{n \in \Z}
    (-1)^n q^{30\binom{n}{2}}\left(q^{13n} + q^{7n+1}\right), 
    \qquad
    F(\tau) =
    \frac{1}{\left(q^4; q^4\right)}\sum_{n \in \Z}
    (-1)^n q^{30\binom{n}{2}}\left(q^{11n} + q^{n+3}\right).
\]
Combining these clues with the fact that $F \in \mT_\C\left(qx^2, 1\right)$, one is inclined to believe that there is a Rogers-type identity (similar to the first equality in \cref{prop:rogers}) for $F(z)$, in which the right-hand side contains dominant powers of $q^{30\binom{n}{2}}$ rather than $q^{5\binom{n}{2}}$. This would imply similar identities for the four sums in the matrix from \cref{prop:2x2-determinant} (thus generalizing \cref{eq:slater}), and due to \cref{cor:rogers-type-sums}, it would likely recover \cref{prop:rogers} as well. Can the reader guess such an identity for $F(z)$, and prove it by identifying Fourier coefficients in $\mT_\C\left(qx^2, 1\right)$? 
\iffalse
More generally, for $m \ge 1$,
\[
    G_m(z) := \sum_{n \ge 0} \frac{q^{n^2}}{\left(q^{m}; q^{m}\right)_n} x^{n}
    \qquad 
    \text{satisfies}
    \qquad 
    G_m(z) = qx G_m(z + 2\tau) + G_m(z + m\tau),
\]
and may be subject to \cref{prop:rogers}-type identities.
\fi
\end{question}

\subsection{Applications to mock theta functions} \label{subsec:mock-theta}
We start by proving \cref{cor:mock-theta-fifth} from \cref{sec:intro}:

\begin{proof}[Proof of \cref{cor:mock-theta-fifth}]
Consider \cref{eq:first-thm-reformulation}, which is just a reformulation of \cref{thm:bases-proportional} (or \cref{thm:twisted-rog-ram}). Dividing both sides by $(q; q)\left(1 - x^2\right)$ yields
\[
    \left(qx^2; q\right)\left(qx^{-2}; q\right)
    \begin{pmatrix}
    \sum_{n \in \Z} q^{n^2} x^{2n}  \left( q^{n+1}x; q\right) \vspace{0.1cm} \\ 
    \sum_{n \in \Z} q^{n^2 + n}x^{2n+1}  \left( q^{n+1}x; q\right)
    \end{pmatrix}
    = 
    -\la qx; q \ra
    \begin{pmatrix}
    \sum_{n \in \Z} (-1)^n q^{5\binom{n}{2} + 2n} \frac{x^{5n+1} - x^{-5n+2}}{1-x^2}
    \vspace{0.1cm} \\
    \sum_{n \in \Z} (-1)^n q^{5\binom{n}{2} + n} \frac{x^{5n} - x^{-5n+3}}{1-x^2}
    \end{pmatrix}
\]
where we used that $\vect{-qx}{0} = (q; q) \la qx; q \ra$.
Now take $x \to -1$, and divide by $(-q; q)$ to obtain
\[
    (q; q)^2 \left(q; q^2\right)
    \begin{pmatrix}
    \sum_{n \in \Z} \frac{q^{n^2}}{(-q; q)_n} \vspace{0.1cm} \\ 
    -\sum_{n \in \Z} \frac{q^{n^2 + n}}{(-q; q)_n}
    \end{pmatrix}
    =
    \begin{pmatrix}
    -\sum_{n \in \Z} q^{5\binom{n}{2} + 2n} (10n-1)
    \vspace{0.1cm} \\
    \sum_{n \in \Z} q^{5\binom{n}{2} + n} (10n-3)
    \end{pmatrix}.
\]
In the left-hand side, break the sums into $n \ge 0$ and $n < 0$, reindex and use our notations in \cref{eq:finite-product-negative} and \cref{eq:mock-theta}. In the right-hand side, substitute $n \mapsto -n$; this recovers \cref{cor:mock-theta-fifth}.
\end{proof}
%\fix{Also at $x = 1$ instead, maybe connect it to this link:} %https://mathworld.wolfram.com/Rogers--RamanujanContinuedFraction.html?

\begin{remark}
Both \cref{cor:mock-theta-fifth} and \cref{eq:watson} give identities for the sums of fifth-order mock theta functions $f_j(q) + 2\psi_j(q)$ (for $j \in \{0, 1\}$), using the notation from \cref{eq:mock-theta}. To list other such identities, we need:
\end{remark}

\begin{notation}[Mock theta functions] \label{not:mock-theta}
Following standard notation \cite{andrews1989ramanujan,chen2012partition,andrews1986fifth,andersen2016vector}, define the third-order mock theta functions introduced by Ramanujan \cite[p.~220]{ramanujan1995ramanujan}
\[
    \phi(q) := \sum_{n \ge 0} \frac{q^{n^2}}{\left(-q^2; q^2\right)_n},
    \qquad\qquad 
    \psi(q) := \sum_{n \ge 1} \frac{q^{n^2}}{\left(q; q^2\right)_n},
\]
and the fifth-order mock theta functions (in addition to those from \cref{eq:mock-theta})
\[
\begin{aligned}
    F_0(q) := \sum_{n \ge 0} \frac{q^{2n^2}}{\left(q; q^2\right)_n},&
    \qquad\qquad 
    F_1(q) := \sum_{n \ge 1} \frac{q^{2n^2 - 2n}}{\left(q; q^2\right)_n},
    \\
    \phi_0(q) := \sum_{n \ge 0} q^{n^2} \left(-q; q^2\right)_n,&
    \qquad\qquad 
    \phi_1(q) := \sum_{n \ge 1} q^{n^2} \left(-q; q^2\right)_{n-1}.
\end{aligned}
\]
Specializing 
\cref{thm:bases-proportional} and \cref{thm:twisted-rog-ram} will lead to identities of fifth-order mock theta functions, while \cref{thm:bases-proportional-2} and \cref{thm:twisted-rog-ram-var} correspond to third-order mock theta functions. For further readings on mock theta functions and mock modular forms, see \cite{gordon2012survey,zagier2009ramanujan,bringmann2017harmonic}. 
\end{notation}

%https://www.researchgate.net/publication/267438962_Mock_theta_functions_to_mock_theta_conjectures
%file:///C:/Users/alex/Downloads/1006.3194.pdf
\begin{proposition}[Sums of mock theta functions] \label{prop:mock-theta-sums}
Let $\tau \in \H^+$. Concerning third-order mock theta functions and the Dedekind eta function, one has
\[
\begin{aligned}
    \phi(q)
    + 2\sum_{n \ge 1} q^n \left(-q^2; q^2\right)_{n-1}
    = \sum_{n \in \Z} \frac{q^{n^2}}{\left(-q^2; q^2\right)_n} &= 
    q^{\frac{1}{24}} \frac{\eta(2\tau)^7}{\eta(\tau)^3 \eta(4\tau)^3},
    \\
    1 + 2\psi(q) + q\sum_{n \ge 0} \left(-q^2\right)^n\left(q; q^2\right)_n = \sum_{n \in \Z} \frac{q^{n^2}}{\left(q; q^2\right)_{n+1}} &= 
    q^{\frac{1}{24}}\frac{\eta(2\tau)^7}{\eta(\tau)^3\eta(4\tau)^3},
\end{aligned}
%GENERATE THEM TOO, TO FIND A SIMPLER FORM!
\] 
where the two right-hand sides happen to be equal. Concerning fifth-order mock theta functions and the generalized eta functions $\E_g$ of level $N = 5$ from \cref{eq:generalized-eta}, one has

\[
\begin{aligned}
    f_0(q) + 2\psi_0(q) = \sum_{n \in \Z} \frac{q^{n^2}}{(-q; q)_n} &=
    q^{\frac{1}{60}} \left(\frac{\eta(\tau)^2}{\eta(2\tau)\E_1(\tau)} + 4\frac{\eta(4\tau)^2}{\eta(2\tau) \E_2(4\tau)} \right),
    \\
    f_1(q) + 2\psi_1(q) = \sum_{n \in \Z} \frac{q^{n^2+n}}{(-q; q)_n} &= 
    q^{-\frac{11}{60}} \left(-\frac{\eta(\tau)^2}{\eta(2\tau)\E_2(\tau)} + 4\frac{\eta(4\tau)^2}{\eta(\tau) \E_1(4\tau)} \right),
    \\
    F_0(q) + \phi_0(-q) - 1 = 
    \sum_{n \in \Z} \frac{q^{2n^2}}{(q; q^2)_n} &= 
    q^{\frac{1}{120}}
    \Bigg( 
    \frac{\eta(4\tau)^6}
    {\eta(\tau) \eta(2\tau)^2 \eta(q^8)^2 \E_1(8\tau)}
    \\
    &\qquad
    +
    2\frac{\eta(2\tau)^4 \eta(8\tau)^4}
    {\eta(\tau) \eta(4\tau)^6 \E_2(2\tau)}
    -
    4\frac{\eta(2\tau)^2 \eta(8\tau)^6}
    {\eta(\tau) \eta(4\tau)^6 \E_1(8\tau)}
    \Bigg),
    \\
    F_1(q) + q^{-1} \phi_1(-q) = 
    \sum_{n \in \Z} \frac{q^{2n^2-2n}}{(q; q^2)_n}
    &=
    q^{-\frac{71}{120}}
    \Bigg(-\frac{\eta(4\tau)^6}
    {\eta(\tau) \eta(2\tau)^2 \eta(q^8)^2 \E_2(8\tau)}
    \\
    &\qquad
    +
    2\frac{\eta(2\tau)^4 \eta(8\tau)^4}
    {\eta(\tau) \eta(4\tau)^6 \E_1(2\tau)}
    +
    4\frac{\eta(2\tau)^2 \eta(8\tau)^6}
    {\eta(\tau) \eta(4\tau)^6 \E_2(8\tau)}
    \Bigg).
\end{aligned}
\]
(The advantage of writing the right-hand sides above in terms of $\eta$ and $\E_g$ is knowing which powers of $q$ are required to make the left-hand sides behave nicely under modular transformations.)
\end{proposition}
\begin{remark}
There are a few things to notice about \cref{prop:mock-theta-sums}:
\begin{itemize}
    \item[(i).] The first two identities might lead to formulae for $\phi(q) \pm 2\psi(q)$, as in \cite[(1.4) and (1.5)]{chen2012partition}. 
    \item[(ii).] The next two identities are just a convenient reformulation of \cref{eq:watson}, which can be derived from equations $(1_0)$ to $(3_1)$ from Watson's paper \cite{watson1937mock}.
    \item[(iii).] Taking $q \mapsto q^2$, the left-hand sides of the last two identities can be given similar but not trivially equivalent formulae, using the same equations $(1_0)$ to $(3_1)$ of Watson \cite{watson1937mock}.
    \item[(iv).] One can regard the left-hand sides as ``completions'' of mock theta functions (not in the sense of harmonic Maass forms), since they correspond to sums over $n \in \Z$ rather than $n \ge 0$ or $n \ge 1$; thus the fifth-order mock theta functions ``complete'' each other in this sense.
    \item[(v).] \cref{prop:mock-theta-sums} contains three pairs of (almost) twin identities, where each pair interchanges $\E_1 \leftrightarrow \E_2$. This supports the parallelism between the two classes of fifth-order mock theta functions defined by Ramanujan (i.e., those indexed by $0$ and those indexed by $1$).
\end{itemize}
\end{remark}

\begin{proof}
The first equality in each of the identities from \cref{prop:mock-theta-sums} follows immediately from \cref{not:mock-theta} and  \cref{eq:finite-product-negative}, so we focus on the second equalities. We also leave the trivial conversion from the notations $\left(q^N; q^N\right)$ and $\la q^g; q^N \ra$ to (generalized) eta functions to the reader.

The third-order identities follow by taking $x \in \{iq^{-1/2}, q^{-1/2}\}$ in \cref{thm:twisted-rog-ram-var}, or from the obvious substitutions in \cref{cor:d=1}. The first two fifth-order identities are equivalent to \cref{eq:watson} (which, as we have seen, follows by taking $x = 1$ and $x = -1$ in \cref{thm:twisted-rog-ram}). For the last two fifth-order identities, take $q \mapsto q^2$ and $x = q^{-1}$ in \cref{thm:twisted-rog-ram}, and use the facts that
\[
\begin{aligned}
    A(q) = 
    \frac{1}{\la -q; q^2\ra}
    \left(\frac{(q; q)}{\left(q^2; q^2\right)\la q; q^5 \ra} 
    - 
    2\frac{\left(q^4; q^4\right)^2}{\left(q^2; q^2\right)^2} B(q)
    \right),&
    \qquad\qquad
    B(q) = \frac{-q\left(q^2; q^2\right)}{(q; q)\la q^8; q^{20} \ra },
    \\
    C(q) = 
    \frac{1}{\la -q; q^2\ra}
    \left(\frac{(q; q)}{\left(q^2; q^2\right)\la q^2; q^5 \ra} 
    - 
    2\frac{\left(q^4; q^4\right)^2}{\left(q^2; q^2\right)^2} D(q)
    \right),&
    \qquad\qquad 
    D(q) = \frac{\left(q^2; q^2\right)}{(q; q)\la q^4; q^{20} \ra},
\end{aligned}
\]
which follow from \cref{eq:5-and-10} and \cref{eq:explicit-matrix}.
\end{proof}

In order to isolate individual mock theta functions rather than sums of them, the natural approach in our framework is to separate the positive and negative powers of $x$ in the canonical basis vectors $\vect{qx^2-x}{k}$ (for the fifth-order functions), respectively $\vect{qx^2-q^{-1}x}{k}$ (for the third-order functions). Indeed, writing $F_{\ge j}(z)$ for $\sum_{n \ge j} \hat{F}(n) x^n$ in $\H^-$, \cref{cor:kind1,cor:kind2} imply that
\begin{equation} \label{eq:truncations}
\begin{aligned}
    &\vect{qx^2-q^{-1}}{0}_{\ge 0} = 
    \sum_{n \ge 0} \frac{q^{n^2}}{\left(-q; q^2\right)_n} x^{2n}
    ,
    \qquad\qquad 
    \vect{qx^2-q^{-1}}{1}_{\ge 0} = 
    \sum_{n \ge 0} \frac{q^{n^2+n}}{\left(-q^2; q^2\right)_n} x^{2n+1}
    ,
    \\
    &\qquad \vect{qx^2-x}{0}_{\ge 1} = 
    \sum_{n \ge 1} q^{n^2} x^{2n} \left(q^{-n+1} x^{-1}; q\right)_{n-1}
    \ =\
    \sum_{n \ge 1} q^{\binom{n+1}{2}} (-x)^{n+1} (qx; q)_{n-1},
    \\
    &\qquad \vect{qx^2-x}{1}_{\ge 1} =
    \sum_{n \ge 0} q^{n^2+n} x^{2n+1} \left(q^{-n} x^{-1}; q\right)_n
    \ =\
    -\sum_{n \ge 0} q^{\binom{n+1}{2}} (-x)^{n+1} (qx; q)_n,
\end{aligned}
\end{equation}
which generalize the summations of $1 + \psi(-q)$, $\phi(q)$, $\psi_0(q)$, respectively $-\psi_1(q)$ from \cref{not:mock-theta} and \cref{eq:mock-theta} (by taking $x = i$, $x = 1$ and $x = -1$). These observations lead to \cref{prop:mock-theta-individual}.

%\begin{remark} 
%The Fourier coefficients of $\vect{qx^2-x}{0}_{\le 0}$ and $\vect{qx^2-x}{1}_{\le 0}$ correspond exactly to the finitized Rogers--Ramanujan type sums studied by Sills in [cite, ...], respectively [cite, ...]. \fix{OR WAS IT THE OTHER DIRECTION?}
%\end{remark}

\begin{proposition}[Individual mock theta functions as double sums] \label{prop:mock-theta-individual}
Let $\tau \in \H^+$. Using \cref{not:mock-theta} and \cref{eq:explicit-matrix}, one has
\begin{gather*}
    \begin{pmatrix}
    \psi_0(q) \\
    \psi_1(q) 
    \end{pmatrix}
    =
    \begin{pmatrix}
    A(q) & B(q) \\
    C(q) & D(q)
    \end{pmatrix}
    \begin{pmatrix}
    \sum_{2n > m \ge 0} \frac{q^{n^2}(-1)^m}{(q; q)_m} \vspace{0.1cm}\\
    \sum_{2n \ge m \ge 0} \frac{q^{n^2+n}(-1)^m}{(q; q)_m}
    \end{pmatrix},
    \\
    \begin{pmatrix}
    \phi_1(q) \\
    \phi_0(q) 
    \end{pmatrix}
    =
    \begin{pmatrix}
    A\left(q^2\right) & B\left(q^2\right) \\
    C\left(q^2\right) & D\left(q^2\right)
    \end{pmatrix}
    \begin{pmatrix}
    q\sum_{2n > m \ge 0} \frac{q^{2n^2-2n} (-q)^m}{\left(q^2; q^2\right)_m} 
    \vspace{0.1cm}\\
    \sum_{2n \ge m \ge 0} \frac{q^{2n^2} (-q)^m}{\left(q^2; q^2\right)_m} 
    \end{pmatrix}.
\end{gather*}
Similar but simpler identities can be given for the third-order mock theta functions $\psi(q)$ and $\phi(q)$.
\iffalse
\[
    \phi(q)
    =
    \frac{1}{\left(-q^2; q^2\right)} \sum_{n \ge m \ge 0} \frac{q^{n^2+m}}{\left(q^2; q^2\right)_m},
    \qquad\text{and}\qquad 
    1 + \psi(q)
    =
    \frac{1}{\left(-q; q^2\right)} \sum_{n \ge m \ge 0}  \frac{q^{n^2} (-1)^m}{\left(q^2; q^2\right)_m}.
\]
\fi
%(one advantage of ... over their initial definitions is that the denominators contain ... rather than ... (see what you already said!!); also, the factors of $A\left(q^2\right)$ etc.\ (which correspond to powers of $q^{10}$, etc.) indicate that some nontrivial amount of work was done, etc.)
\end{proposition}
\iffalse
\[
    \begin{pmatrix}
    \frac{1}{2} f_0(q) & \psi_0(q) \\
    \frac{1}{2} f_1(q) & \psi_1(q) 
    \end{pmatrix}
    =
    \begin{pmatrix}
    A(q) & B(q) \\
    C(q) & D(q)
    \end{pmatrix}
    \lim_{x \searrow -1} \begin{pmatrix}
    \sum_{\substack{0 \le m \ge 2n}} \frac{q^{n^2}x^m}{(q; q)_m}
    & \sum_{\substack{0 \le m < 2n}} \frac{q^{n^2}x^m}{(q; q)_m}\\[2pt]
    -\sum_{\substack{0 \le m > 2n}} \frac{q^{n^2+n}x^m}{(q; q)_m}
    & 
    -\sum_{\substack{0 \le m \le 2n}} \frac{q^{n^2+n}x^m}{(q; q)_m}
    \end{pmatrix}.
\]
\fi

\begin{proof}
From \cref{thm:twisted-rog-ram-var} and \cref{cor:kind2}, one obtains
\[
    \begin{pmatrix}
    \vect{qx^2-x}{0} \vspace{0.1cm}\\
    \vect{qx^2-x}{1}
    \end{pmatrix}
    =
    \begin{pmatrix}
    A(q) & B(q) \\
    C(q) & D(q)
    \end{pmatrix}
    \begin{pmatrix}
    \left(x^{-1}; q\right)^{-1} \vect{qx^2}{0}
    \vspace{0.1cm}\\
    \left(x^{-1}; q\right)^{-1} \vect{qx^2}{1}
    \end{pmatrix}.
\]
Writing $\left(x^{-1}; q\right)^{-1} = \sum_{m \ge 0} (q; q)^{-m} x^{-m}$ and restricting to powers $x^n$ with $n \ge 0$ yields
\[
    \begin{pmatrix}
    \vect{qx^2-x}{0}_{\ge 1} \vspace{0.1cm}\\
    \vect{qx^2-x}{1}_{\ge 1}
    \end{pmatrix}
    =
    \begin{pmatrix}
    A(q) & B(q) \\
    C(q) & D(q)
    \end{pmatrix}
    \begin{pmatrix}
    \sum_{2n > m \ge 0} \frac{q^{n^2}}{(q; q)_m} x^{2n-m} 
    \vspace{0.1cm} \\ 
    \sum_{2n\ge m \ge 0} \frac{q^{n^2+n}}{(q; q)_m} x^{2n+1-m} 
    \end{pmatrix}.
\]
Taking $x = -1$ recovers the first claimed equality in light of \cref{eq:truncations}. The second claimed equality follows analogously by taking $q \mapsto q^2$ above, and then $x = -q^{-1}$.
\end{proof}

\iffalse
\fix{MAYBE IGNORE THIS QUESTION ALTOGETHER; TOO HARD!}
\begin{question} \label{qtn:mock-theta}
\fix{open problem: odd/even powers (the most complicated functions disappear)!}
Can one use the work in \cref{subsec:mock-theta}, which stems from \cref{thm:bases-proportional}, to obtain more information about the fifth-order mock theta functions? As explained there, it seems possible that one could identify at least the even powers of $q$ in Ramanujan's mock theta identity for $\psi_0$, and the odd powers in the identity for $\psi_1$, using variations of our methods.
\fix{SAY WHAT THE mock theta CONJ.\ REDUCE TO NOW IN TERMS OF $\Phi$ and $\Psi$; get rid of $2's$ and $q^2$}
\end{question}
\fi

\begin{remark}
The mock theta conjectures (given different proofs in \cite{hickerson1988proof,folsom2008short}) relate the fifth-order mock theta functions to special values of the widely studied function \cite{ramanujan1988lost,dyson1944some,bringmann2010dyson,andrews2018four} %http://www.personal.psu.edu/gea1/pdf/324.pdf
\begin{equation}\label{eq:generalized-mock-theta}
    G(y, q) := \sum_{n \ge 0} \frac{q^{n^2}}{(qy; q)_n (q/y; q)_n},
\end{equation}
for $y \in \C$. The natural approach to this function in the context of this section is to consider the generalization
\[
    G(x, y, q) := \sum_{n \ge 0} \frac{q^{n^2}x^{2n}}{(qxy; q)_n (qx/y; q)_n}
    =
    \left(\frac{1}{(qxy; q)(qx/y; q)} \sum_{n \in \Z} q^{n^2}x^{2n} \left(q^{n+1}xy; q\right) \left(q^{n+1}x/y; q\right) \right)_{\ge 0},
\]
where $F_{\ge 0}$ indicates the truncation to nonnegative powers of $x$ as before. The sum over $n \in \Z$ above is a doubly-twisted version of the series $\sum_{n \in \Z} q^{n^2} x^{2n} \in \mT_\C\left(qx^2\right)$, in the sense of \cref{eq:w-coeffs}; hence it also belongs to $\mT_\C\left(qx^2\right)$, and can be expanded using a similar argument to \cref{lem:t-identities}. This leads to the following statement, which is ultimately a consequence of \cref{prop:fractional}.
\end{remark}

\begin{proposition}[Double twists] \label{prop:double-twist}
Let $y \in \C^\times$. In $\mT_\C\left(qx^2\right)$, one has
\[
\begin{aligned}
    (q; q)^2 \sum_{n \in \Z}
    q^{n^2} x^{2n} \left(q^{n+1} xy; q \right) \left(q^{n+1} x/y; q \right)
    &=
    \vect{qx^2}{0}\sum_{\substack{d \in \Z \\ d \textnormal{ even}}} \Big(q^{\frac{d^2 + 2|d|}{4}} \sum_{n \ge 0} (-1)^n q^{\binom{n+1}{2} + |d|n}\Big) y^d
    \\
    &- 
    \vect{qx^2}{1} \sum_{\substack{d \in \Z \\ d \textnormal{ odd}}} \Big(q^{\frac{(|d|+1)^2}{4}}  \sum_{n \ge 0} (-1)^n q^{\binom{n+1}{2} + |d|n}\Big) y^d.
\end{aligned}
\]
%http://www.personal.psu.edu/gea1/pdf/324.pdf
 %https://www.wcupa.edu/sciences-mathematics/mathematics/jMcLaughlin/documents/RLNfiniteJune262010.pdf
%https://gdz.sub.uni-goettingen.de/id/PPN356556735_0094?tify={%22pages%22:[655],%22panX%22:0.5,%22panY%22:0.776,%22view%22:%22info%22,%22zoom%22:0.39}
\end{proposition}

\begin{proof}
In light of the discussion above and \cref{lem:proofs-by-fourier}, it suffices to identify coefficients of $x^0$ and $x^1$ on both sides. Denoting the left-hand side by $L(z)$, \cref{eq:basic1} gives
\[
    L(z)
    =
    (q; q)^2
    \sum_{\substack{n \in \Z \\ m, p \ge 0}} \frac{q^{n^2 + \binom{m}{2} + \binom{p}{2} + (n+1)(m+p)}}{(q; q)_m (q; q)_p} (-y)^{m-p} x^{2n+m+p}.
\]
Collecting coefficients of $y^d$ in $\hat{L}(0)$ and $\hat{L}(1)$ for $d \in \Z$, a short computation yields
\[
    \hat{L}(0) 
    = 
    \sum_{d \in \Z, \text{ even}} q^{\frac{d^2 + 2|d|}{4}} S(|d|)\ y^{2d},
    \qquad\qquad\qquad
    \hat{L}(1)
    =
    -
    \sum_{d \in \Z, \text{ odd}} q^{\frac{\left(|d| + 1\right)^2}{4}} S(|d|)\ y^d,
\]
where $S(|d|) = (q; q)^2 \sum_{p \ge 0} \frac{q^p}{(q; q)_{p+|d|} (q; q)_p}$. But by \cref{eq:basic2}, $S(|d|)$ is precisely the coefficient of $x^{|d|}$ in $(q; q)^2 (x; q)^{-1} (q/x; q)^{-1}$ (for $|q| < |x| < 1$), which is expanded in \cref{prop:fractional}. Collecting terms in the right-hand side of \cref{eq:inverse} gives $S(|d|) = \sum_{n \ge 0} (-1)^n q^{\binom{n+1}{2} + |d|n}$, as we wanted.
\end{proof}

\begin{question} \label{qtn:double-twist}
Can one use \cref{prop:double-twist} to:
\begin{itemize}
    \item[(i).] Obtain a \cref{prop:mock-theta-sums}-type formula for the sum $\sum_{n \in \Z} q^{n^2} (-q; q)_n^{-2}$ (by specialization at $xy = x/y = -1$), and a \cref{prop:mock-theta-individual}-type formula for the third-order mock theta function $f(q) := \sum_{n \ge 0} q^{n^2} (-q; q)_n^{-2}$ (see, e.g., \cite{chen2012partition})?
    \item[(ii).] Obtain relevant identities for the sum $G(y, q)$ from \cref{eq:generalized-mock-theta}, via its generalization $G(x, y, q)$?
\end{itemize}
\end{question}

\section*{Acknowledgements}

The author wishes to thank Professor Terence Tao for his helpful guidance, suggestions and support for this project. The author is also deeply grateful to Professors Ole Warnaar and William Duke for insightful discussions and feedback.

\newpage

\bibliographystyle{plain}
\bibliography{main}

\end{document}